\documentclass[12pt]{article}

\title{Spherical Geometry of Hilbert Schemes of Conics in Adjoint Varieties}

\author{Minseong Kwon}
\date{September 25, 2023}

\usepackage[margin = 1in]{geometry} 
\usepackage{amsmath, bm}
\usepackage{amssymb}
\usepackage{mathrsfs}
\usepackage{MnSymbol}
\usepackage{graphicx}
\usepackage{float}
\usepackage{amsthm}
\usepackage{tikz-cd}
\usepackage{cancel}
\usepackage{dynkin-diagrams}
\usepackage{hyperref}
\usepackage{multirow}
\usepackage{tikz}

\usetikzlibrary{patterns.meta}

\newtheorem{theorem}{Theorem}[section]
\newtheorem{coro}[theorem]{Corollary}
\newtheorem{lemma}[theorem]{Lemma}
\newtheorem{proposition}[theorem]{Proposition}

\newtheorem*{theorem*}{Theorem}
\newtheorem*{maintheorem*}{Main Theorem}
\newtheorem*{coro*}{Corollary}
\newtheorem*{lemma*}{Lemma}
\newtheorem*{prop*}{Proposition}
\newtheorem*{claim*}{Claim}

\theoremstyle{definition}
\newtheorem{defn}[theorem]{Definition}

\newtheorem*{prob*}{Problem}

\newtheorem{exa}[theorem]{Example}
\newtheorem*{exa*}{Example}

\newtheorem*{exer*}{Exercise}

\newtheorem*{acknowledgements*}{Acknowledgements}
\newtheorem{rmk}[theorem]{Remark}
\newtheorem*{rmk*}{Remark}

\def\BigRoman{\uppercase\expandafter{\romannumeral\number\count 255}}
\def\Romannumeral{\afterassignment\BigRoman\count255=}
\DeclareMathAlphabet{\mathpzc}{OT1}{pzc}{m}{it}

\def \CC {\mathbb{C}}

\def \PP {\mathbb{P}}
\def \QQ {\mathbb{Q}}
\def \RR {\mathbb{R}}

\def \ZZ {\mathbb{Z}}

\def \Bcal {\mathcal{B}}
\def \Ccal {\mathcal{C}}
\def \Dcal {\mathcal{D}}
\def \Ecal {\mathcal{E}}
\def \Fcal {\mathcal{F}}

\def \Kcal {\mathcal{K}}
\def \Lcal {\mathcal{L}}

\def \Ocal {\mathcal{O}}
\def \Pcal {\mathcal{P}}

\def \Vcal {\mathcal{V}}

\def \Ffr {\mathfrak{F}}

\def \bfr {\mathfrak{b}}

\def \gfr {\mathfrak{g}}

\def \kfr {\mathfrak{k}}

\def \pfr {\mathfrak{p}}

\def \tfr {\mathfrak{t}}

\def \hbar {\bar{h}}

\def \Cbf {\mathbf{C}}

\def \Hbf {\mathbf{H}}

\usepackage[english]{babel}
\usepackage[utf8]{inputenc}
\usepackage{fancyhdr}

\begin{document}
\maketitle
\abstract{
    For each adjoint variety not of type $A$ or $C$, we study the irreducible component of the Hilbert scheme which parametrizes all smooth conics. We prove that its normalization is a spherical variety by using contact geometry, and then compute the colored fan of the normalization. As a corollary, we describe the conjugacy classes of conics in the adjoint variety and show smoothness of the normalization. Similar results on the Chow scheme of the adjoint variety are also presented.
    }
\tableofcontents

\section{Introduction}

Let $\gfr$ be a complex simple Lie algebra, and $G$ the simply connected Lie group associated to $\gfr$. Then $\gfr$ admits a natural irreducible $G$-representation called the \emph{adjoint representation}, and the unique closed $G$-orbit in $\PP(\gfr)$ is called the \emph{adjoint variety}. Among rational homogeneous spaces, adjoint varieties are of particular interest since they are the only known examples of Fano (complex-)contact manifolds. Here, a \emph{contact structure} on a complex manifold $Z$ means a holomorphic hyperplane distribution $D$ in the tangent bundle $TZ$ such that the Lie bracket of vector fields induces a bundle morphism
\[
    D \wedge D \rightarrow TZ / D
\]
which is everywhere nondegenerate. Existence of contact structures on a projective manifold is a restrictive condition, and it has been conjectured that every Fano contact manifold is isomorphic to an adjoint variety (\cite{Beauville1998FanoContact}).

It is natural to expect that geometry of subvarieties in an adjoint variety is influenced by its contact structure. For example, every line in an adjoint variety must be tangent to the contact structure. Moreover, the space of lines passing through a fixed point can be described in terms of the isotropy action on the contact hyperplane (\cite{Hwang1997RigidityHomogeneous}).

The simplest non-linear subvarieties in a given projective variety are conics, and spaces of conics in adjoint varieties also have been studied in various contexts. Let us recall some results related to our work.
\begin{itemize}
    \item In \cite{Wolf1965ComplexHomogeneous}, Wolf observed that each adjoint variety is a $C^{\infty}$ fiber bundle over a real manifold such that its fibers are smooth conics transverse to the contact structure. Moreover, the base is a symmetric Riemannian manifold and can be embedded into the space of smooth conics in the adjoint variety as a totally real submanifold.
    \item For the $G_{2}$-adjoint variety, in \cite[Chapitre 10]{Dufour2014ConstructionMetriques}, Dufour explicitly constructed a (non-compact) family of conics by deforming double lines, which consists of 8-, 7-, and 5-dimensional subfamilies of smooth conics, reducible conics, and double lines, respectively. Dufour generalized this phenomenon to parabolic geometry of type $G_{2}$ to produce quaternionic K\"{a}hler manifolds.
    \item In \cite[Section 7]{Manivel2021DoubleCayley}, Manivel found an incidence correspondence parametrizing conics in the $G_{2}$-adjoint variety. Its base is a projective manifold called the \emph{Cayley grassmannian}, which is a symmetric (hence spherical) variety.
    \item Birational geometry of the spaces of rational curves of degree $\le 3$ in rational homogeneous spaces was studied by Chung, Hong and Kiem in \cite{ChungHongKiem2012CompactifiedModuli}. Via blow-ups and blow-downs, they related three compactifications of the spaces of smooth rational curves obtained from the Hilbert schemes, semi-stable sheaves, and stable curves.
\end{itemize}

In this paper, we focus on the Hilbert scheme of conics in each adjoint variety, and prove the following theorem.
\begin{maintheorem*}[Theorem \ref{main thm: open symmetric orbit of twistor conics with satake diagram}, \ref{main thm: colored data of hilb}]
    Let $\gfr$ be a complex simple Lie algebra not of type $A$ or $C$, and $Z_{\gfr} \subset \PP(\gfr)$ the adjoint variety. Then for the irreducible component $\Hbf(Z_{\gfr})$ of the semi-normalization of the Hilbert scheme which parametrizes all smooth conics in $Z_{\gfr}$, the normalization $\Hbf^{nor}(Z_{\gfr})$ is a spherical variety and its colored fan can be explicitly computed.
\end{maintheorem*}
Remark that $\Hbf^{nor}(Z_{\gfr})$ is indeed the normalization of the irreducible component of the Hilbert scheme parametrizing all smooth conics, even before taking the semi-normalization. In the process of the proof of Main Theorem, we also obtain a similar result for the irreducible component $\Cbf(Z_{\gfr})$ of the Chow scheme parametrizing all smooth conics. The precise statement is given in Theorem \ref{main thm: colored data of chow}.

In Main Theorem, we use the language of spherical geometry to describe spaces of conics. A \emph{spherical variety} means a normal variety where a reductive algebraic group acts such that every Borel subgroup has an open orbit. This definition includes a plenty of classical examples equipped with algebraic group actions, such as toric varieties, symmetric varieties, etc. As toric varieties can be classified by their fans, spherical varieties can be classified in terms of combinatorial data called \emph{colored fans}, which is a consequence of Luna-Vust theory (\cite{LunaVust1983PlongementsEspaces}).

To prove Main Theorem, first we show that smooth conics transverse to the contact structure form a homogeneous symmetric variety, which is an open orbit in the space of smooth conics. To understand its boundary, we compute isotropy groups of closed orbits by analyzing orbits of singular conics. Then the colored fan of $\Cbf^{nor}(Z_{\gfr})$ follows from the description of colored fans of symmetric varieties established by Vust (\cite{Vust1990PlongementsEspaces}). Finally, using the natural morphism from the Hilbert scheme to the Chow scheme, we obtain the colored fan of $\Hbf^{nor}(Z_{\gfr})$.

As an application, we describe the conjugacy classes of conics in adjoint varieites, prove smoothness of $\mathbf{H}^{nor}(Z_{\gfr})$ and derive information on singularities of $\Cbf^{nor}(Z_{\gfr})$. For example, for the $G_{2}$-adjoint variety, we show that both $\Hbf(Z_{G_{2}})$ and $\Cbf(Z_{G_{2}})$ consist of three orbits: an 8-dimensional orbit of smooth conics transverse to the contact structure, a 7-dimensional orbit of reducible conics, and a 5-dimensional orbit of double lines. Moreover, their normalizations are isomorphic to the Cayley grassmannian (\cite{Manivel2018CayleyGrassmannian}), which is a smooth Fano symmetric variety.

This paper is organized as follows. In Section \ref{section: preliminary and assumption}, we explain preliminaries and fix our notation. In Section \ref{section: classes of conics}, we introduce several types of conics in adjoint varieties and state our main theorems. In Section \ref{section: counting conjugacy classes of conics}, we investigate smooth conics transverse to the contact structure and double lines. In particular, we show that smooth conics transverse to the contact structure form a homogeneous symmetric variety. Based on the results of Section \ref{section: counting conjugacy classes of conics}, we complete the proof of the main theorems in Section \ref{section: proof of main theorems}. Section \ref{section: structures of spaces of conics} is devoted to corollaries. In particular, after studying behavior of reducible conics, we describe the conjugacy classes of conics in adjoint varieties. Finally in Section \ref{section: direction of contact conics}, we study smooth conics tangent to the contact distribution, and show that their tangent directions do not dominate the contact distribution.

\begin{acknowledgements*}
    The author would like to sincerely thank Professor Jun-Muk Hwang for guidance and valuable suggestions. The author is grateful to Professor Jaehyun Hong and Professor Kyeong-Dong Park for detailed comments on the draft of this paper and helpful discussions on spherical varieties. This work was supported by the Institute for Basic Science (IBS-R032-D1-2022-a00).
\end{acknowledgements*}

\section{Preliminaries} \label{section: preliminary and assumption}

In this section, we explain our setting and notation. First of all, our base field is $\CC$, the field of complex numbers. A \emph{variety} means an integral separated scheme of finite type over $\CC$, and a \emph{point} in a variety means a closed point. For a $\CC$-vector space $V$, $\PP(V) := V -\{0\} /\CC^{\times}$ denotes the space of 1-dimensional subspaces of $V$.

\subsection{Lie Theory}

Our main reference on Lie theory is \cite{OnishchikVinberg1990LieGroups}. Let $\gfr$ be a semi-simple Lie algebra. Let $G$ be the simply connected Lie group associated to the Lie algebra $\gfr$. Choose a maximal torus $T$ in $G$, and a Borel subgroup $B$ containing $T$. We denote the Lie algebras of $T$ and $B$ by $\tfr$ and $\bfr$, respectively. The set of roots and the set of simple roots are denoted by $R$ and $S$, respectively. When $\gfr$ is simple, we use the numbering $S = \{\alpha_{1}, \, \ldots, \, \alpha_{\text{rank}\,\gfr}\}$ of simple roots given in \cite[Reference Chapter, Table 1]{OnishchikVinberg1990LieGroups} (which is different from the one in \cite{Bourbaki2002LieGroups}, especially for exceptional Lie algebras other than $G_{2}$). For each root $\alpha \in R$, $\gfr_{\alpha}$ means the root space corresponding to $\alpha$ so that the root decomposition of $\gfr$ is given by
\[
    \gfr = \tfr \oplus \bigoplus_{\alpha \in R} \gfr_{\alpha}.
\]

The character group of $T$ is denoted by $\chi(T)$, and each character $\lambda \in \chi(T)$ is regarded as a linear functional on $\tfr$. In this notation, if $H \in \tfr$, then the value of $\lambda$ at $\exp(H)$ is equal to $e^{\lambda(H)}$. The bracket $\langle \,,\,\rangle$ means the Killing form on $\gfr$, and the dual of a root $\alpha \in R$ is denoted by $H_{\alpha} \in \tfr$. More precisely, $H_{\alpha}$ is the element of $\tfr$ satisfying $\langle H_{\alpha}, \, H \rangle = \alpha(H)$ for all $H \in \tfr$. The pairing of two roots $\alpha$ and $\beta$ is defined as $\langle \alpha, \, \beta \rangle := \langle H_{\alpha}, \, H_{\beta} \rangle$, which extends to an inner product on $\chi(T) \otimes_{\ZZ} \RR$. The natural pairing of $\chi(T)$ and its dual $\chi_{*}(T)$ is also denoted by $\langle \lambda,\, \mu \rangle$ for $\lambda \in \chi(T)$ and $\mu \in \chi_{*}(T)$ so that the Cartan integer is given by $\langle \alpha \, | \, \beta \rangle = \langle \alpha, \, \beta^{\vee} \rangle$ for $\alpha, \, \beta \in R$ where $\beta^{\vee}$ is the coroot corresponding to $\beta$. A nonzero vector in $\gfr_{\alpha}$ for some $\alpha \in R$ is called a root vector, which is often denoted by $E_{\alpha}$. If a collection $\{E_{\alpha} \in \gfr_{\alpha} : \alpha \in R\}$ of root vectors is given, we define $N_{\alpha, \, \beta}$ for $\alpha, \, \beta \in R$ to be the complex number satisfying
\[
    [E_{\alpha}, \, E_{\beta}] = N_{\alpha, \, \beta} \cdot E_{\alpha + \beta}
\]
if $\alpha + \beta \in R$, and $N_{\alpha, \, \beta} = 0$ if $\alpha + \beta \not \in R$.

For a nonempty subset $I \subset S$, we denote by $P_{I}$ the parabolic subgroup containing $B$ generated by the complement $S\setminus I$ of $I$. That is, the Lie algebra of $P_{I}$ is
\[
    \pfr_{I} := \bfr \oplus \bigoplus_{\alpha \in R^{+} \cap \, \text{span}(S \setminus I)} \gfr_{-\alpha}
\]
where $R^{+}$ is the set of all positive roots. The opposite parabolic subgroup of $P_{I}$ is denoted by $P_{I}^{-}$. We define $W = W_{G}$ to be the Weyl group of $(G, \, T)$, and $W_{G,\, P_{I}}$ means the subgroup of $W$ generated by reflections with respect to $\alpha \in S \setminus I$ so that $P_{I} = B \cdot W_{G,\, P_{I}} \cdot B$.

\subsection{Adjoint Varieties}
From now on, let $\gfr$ be a simple Lie algebra, and $\rho \in R$ the highest root (with respect to the Borel subgroup $B$). Then the adjoint representation induces a $G$-action on $\PP(\gfr)$. Since the adjoint representation is irreducible, $\PP(\gfr)$ contains the unique closed $G$-orbit, which is the $G$-orbit containing $[E_{\rho}]$ where $E_{\rho} \in \gfr_{\rho}$ is a root vector. This projective subvariety of $\PP(\gfr)$ is called the \emph{adjoint variety} and denoted by $Z_{\gfr}$.

It is known that each $Z_{\gfr}$ is equipped with a hyperplane distribution described as follows. For the point $o:=[E_{\rho}] \in Z_{\gfr}$, the isotropy group $P:=\text{Stab}_{G}([E_{\rho}])$ is a parabolic subgroup containing $B$ and its Lie algebra is
\[
    \pfr = \tfr \oplus \bigoplus_{\alpha \in R, \, \langle \alpha, \, \rho \rangle \ge 0} \gfr_{\alpha}.
\]
Then the tangent space $T_{o} Z_{\gfr}$ of $Z_{\gfr}$ at $o$ is identified with $\gfr / \pfr$ as a $P$-module, and
\[
    \gfr / \pfr \simeq \bigoplus_{\alpha \in R, \, \langle \alpha, \, \rho \rangle < 0} \gfr_{\alpha}
\]
as vector spaces. Under these identifications, consider a hyperplane defined by
\[
    D_{o} := \bigoplus_{\alpha \in R, \, \langle \alpha| \, \rho \rangle = -1} \gfr_{\alpha} \left(= \bigoplus_{\alpha \in R \setminus \{-\rho\}, \, \langle \alpha, \, \rho \rangle < 0} \gfr_{\alpha}\right)
\]
Since $D_{o}$ is $P$-invariant, the $G$-action on $Z_{\gfr}$ induces a well-defined $G$-invariant holomorphic subbundle $D \subset T Z_{\gfr}$ on $Z_{\gfr}$ extending $D_{o}$. This hyperplane distribution $D$ is called the \emph{contact distribution} on the adjoint variety. See \cite{Wolf1965ComplexHomogeneous} and \cite{Lebrun1995FanoManifolds} for details.

Finally, observe that $D_{o}$ is even-dimensional, hence the (complex) dimension of $Z_{\gfr}$ is odd. We write $\dim_{\CC}(Z_{\gfr}) = 2n+1$ for some $n \in \ZZ_{\ge 0}$. If $\gfr$ is neither of type $A$ nor $C$ (see the assumption (\ref{assumption: pic num 1 not pn}) below), then we have $n \ge 2$. It is not hard to see that $n$ is given by
\[
    n = \left\{ \begin{array}[]{ll}
        2r - 3 & (\text{if $\gfr = B_{r}$, $r \ge 3$});\\
        2r - 4 & (\text{if $\gfr = D_{r}$, $r \ge 4$});\\
        10 & (\text{if $\gfr = E_{6}$}); \\
        16 & (\text{if $\gfr = E_{7}$}); \\
        28 & (\text{if $\gfr = E_{8}$}); \\
        7 & (\text{if $\gfr = F_{4}$}); \\
        2 & (\text{if $\gfr = G_{2}$}). \\
    \end{array} \right.
\]

\subsection{Hilbert Schemes and Chow Schemes}

Let us recall the relationship between the Hilbert scheme and the Chow scheme. Consider a projective variety $X$ equipped with an ample line bundle $\Ocal_{X}(1)$. For each polynomial $p(m) \in \QQ[m]$, there is a projective $\CC$-scheme $\text{Hilb}_{p(m)}(X, \, \Ocal_{X}(1))$, called the \emph{Hilbert scheme}, which is the universal moduli space of closed subschemes of $X$ with Hilbert polynomial $p(m)$. On the other hand, for two integers $d \in \ZZ_{\ge 0}$ and $d' \in \ZZ_{>0}$, there is another projective $\CC$-scheme $\text{Chow}_{d,\, d'}(X,\, \Ocal_{X}(1))$, called the \emph{Chow scheme}, which is the universal moduli space of non-negative proper algebraic $d$-cycles of degree $d'$ in $X$. Moreover, there is a natural morphism from the Hilbert schemes to the Chow schemes, which sends subschemes to their fundamental classes.

\begin{theorem}[{\cite[Theorem I.6.6 and Theorem I.7.3.1]{Kollar1996RationalCurves}}] \label{thm: map FC}
    For a projective variety $X$ with an ample line bundle $\Ocal_{X}(1)$ and a polynomial $p(m) \in \QQ[m]$ of degree $d$, there is a morphism
    \[
        FC: \text{Hilb}^{sn}_{p(m)} (X, \, \Ocal_{X}(1)) \rightarrow \coprod_{d'=1}^{\infty} \text{Chow}_{d, \, d'}(X,\, \Ocal_{X}(1))
    \]
    satisfying the following condition: For a closed subscheme $V \subset X$ with Hilbert polynomial $p(m)$, $FC$ is a local isomorphism near the point $[V] \in \text{Hilb}^{sn}_{p(m)}(X, \, \Ocal_{X}(1))$ if $V$ is reduced, has pure dimension and satisfies Serre's condition $S_{2}$. Here, by $\text{Hilb}^{sn}_{p(m)} (X, \, \Ocal_{X}(1))$ we denote the semi-normalization of the associated reduced scheme $(\text{Hilb}_{p(m)} (X, \, \Ocal_{X}(1)))^{red}$.
\end{theorem}

Note that since the semi-normalization morphism is bijective on (closed) points, we may identify points in the Hilbert scheme and points in its semi-normalization.

\subsection{Luna-Vust Theory for Symmetric Varieties} \label{section: Luna Vust theorey}

Let us review the embedding theory of spherical varieties, especially for symmetric varieties. Our main reference is \cite{Timashev2011HomogeneousSpaces} and \cite{Knop1991LunaVustTheory}.

A normal $G$-variety $X$ is called \emph{spherical} if a Borel subgroup has an open orbit in $X$, or equivalently if a Borel subgroup has only finitely many orbits. If $O$ is the open $G$-orbit in the spherical variety $X$, then $O$ is a homogeneous spherical variety, and we say that \emph{$X$ is a (spherical) $O$-embedding}.

Colored data associated to a homogeneous spherical variety $O$ can be constructed as follows. First, choose a maximal torus $T' \subset G$ and a Borel subgroup $B' \subset G$ containing $T'$. Here, we use the prime symbol for $T'$ and $B'$ to emphasize that they can be different from $T$ and $B$ chosen in the previous sections. Then define $\Lambda_{O}$ to be the sublattice in the character group $\chi(B') \simeq \chi(T')$ consisting of weights of rational $B'$-eigenfunctions, i.e. elements of $\CC(O)^{(B')} \subset \CC(O)^{\times}$. If $\CC(O)^{B'} \subset \CC(O)^{\times}$ denotes the set of $B'$-invariant rational functions on $O$, then since we have a short exact sequence
\[
    0 \rightarrow \CC^{\times} = \CC(O)^{B'} \rightarrow \CC(O)^{(B')} \rightarrow \Lambda_{O} \rightarrow 0,
\]
a valuation $v : \CC(O)^{\times} \rightarrow \QQ$ induces a group homomorphism $\Lambda_{O} \rightarrow \QQ$. In other words, a valuation corresponds to an element in the $\QQ$-vector space $\Ecal := \text{Hom}_{\ZZ}(\Lambda_{O}, \, \QQ)$. Moreover, this correspondence is injective for $G$-invariant valuations (\cite[Corollary 2.8]{Knop1991LunaVustTheory}), hence we may identify the set of $G$-invariant valuations on $O$ with its image in $\Ecal$, denoted by $\Vcal$ and called the \emph{valuation cone}. If we define
\[
    \Dcal(O) := \{\text{$B'$-stable prime Weil divisors of $O$}\}
\]
and consider the valuation induced by each element of $\Dcal(O)$, then a similar process yields a function $\epsilon : \Dcal(O) \rightarrow \Ecal$, which is not injective in general. The elements of $\Dcal(O)$ are called \emph{colors}.

Next, consider a \emph{simple} $O$-embedding $X$, meaning that $X$ is an $O$-embedding which contains exactly one closed orbit. If $Y \subset X$ is the unique closed orbit in $X$, define
\[
    \Fcal(X) := \{D \in \Dcal(O) : Y \subset \overline{D} \text{ in } X\}.
\]
We call each element of $\Fcal(X)$ a \emph{color} of $X$. Since $G$-stable prime divisors of $X$ can be considered as elements of $\Vcal$, it is possible to define a convex cone in $\Ecal$ by
\[
    \Ccal(X) := \QQ_{\ge 0} \langle \epsilon(\Fcal(X)), \, \text{$G$-stable prime divisors of $X$} \rangle.
\]

Now let $X$ be an arbitrary $O$-embedding. For every $G$-orbit $Y \subset X$, $Y$ has a $G$-stable open neighborhood
\[
    X_{Y} := \{x \in X : Y \subset \overline{G \cdot x}\}
\]
which is a simple $O$-embedding such that $Y$ is its unique closed orbit. Thus to $X$, we can associate a collection of pairs
\[
    \mathfrak{F}(X) := \{(\Ccal(X_{Y}), \, \Fcal(X_{Y})) : Y \text{ is a $G$-orbit in }X\}.
\]
Then $(\Ccal(X_{Y}), \, \Fcal(X_{Y}))$ and $\mathfrak{F}(X)$ are a colored cone and a colored fan, respectively, in the following sense.
\begin{defn} \label{defn: colored cone and fan}
    Let $\Vcal \subset \Ecal$ and $\epsilon : \Dcal(O) \rightarrow \Ecal$ as before.
    \begin{enumerate}
        \item A \emph{colored cone} is a pair $(\Ccal, \, \Fcal)$ of subsets $\Ccal \subset \Ecal$ and $\Fcal \subset \Dcal(O)$ such that
        \begin{enumerate}
            \item $\Ccal$ is a convex cone generated by $\epsilon(\Fcal)$ and finitely many elements in $\Vcal$;
            \item the relative interior of $\Ccal$ intersects with $\Vcal$.
        \end{enumerate}
        \item A colored cone $(\Ccal, \, \Fcal)$ is called \emph{strictly convex} if $\Ccal$ is strictly convex and $0 \not\in \epsilon(\Fcal)$.
        \item For a colored cone $(\Ccal, \, \Fcal)$, a pair $(\Ccal_{0}, \, \Fcal_{0})$ is called a \emph{colored face} of $(\Ccal, \, \Fcal)$ if $\Ccal_{0}$ is a face of the cone $\Ccal$, $\Fcal_{0} = \Fcal \cap \epsilon^{-1}(\Ccal_{0})$, and the relative interior of $\Ccal_{0}$ intersects with $\Vcal$.
        \item A nonempty finite set $\Ffr$ of colored cones in $(\Ecal,\, \Dcal(\Ocal))$ is called a \emph{colored fan} if
        \begin{enumerate}
            \item For every element of $\Ffr$, its colored faces are contained in $\Ffr$;
            \item For every $v \in \Vcal$, there is at most one element of $\Ffr$ of which relative interior contains $v$.
        \end{enumerate}
        \item A colored fan is called \emph{strictly convex} if it consists of strictly convex colored cones.
    \end{enumerate}
\end{defn}
\begin{theorem}[{\cite[Theorem 4.3]{Knop1991LunaVustTheory}, \cite[Section 15]{Timashev2011HomogeneousSpaces}}] \label{thm: classification of spherical}
    For a homogeneous spherical variety $O$, the map $X \mapsto \mathfrak{F}(X)$ is a bijection between isomorphism classes of $O$-embeddings, and strictly convex colored fans in $(\Ecal, \, \Dcal(O))$.
\end{theorem}

Under this correspondence, a simple $O$-embedding $X$ is corresponding to a colored fan consisting of $(\Ccal(X), \, \Fcal(X))$ and its colored faces. Conversely, every strictly convex colored cone is induced from a simple $O$-embedding.

A lot of geometric properties of spherical varieties can be expressed in terms of colored data. For example, we have the following lemmas.

\begin{lemma}[{\cite[Lemma 4.2]{Knop1991LunaVustTheory}}] \label{lem: number of orbits in terms of colored faces}
    For a spherical variety $X$, the assignment $Y \mapsto (\Ccal(X_{Y}), \, \Fcal(X_{Y}))$ between orbits in $X$ and elements of $\mathfrak{F}(X)$ is bijective and order-reversing. Here, the set of orbits is (partially) ordered by inclusion of closures.
\end{lemma}

\begin{lemma}[{\cite[Theorem 5.2]{Knop1991LunaVustTheory}}] \label{lem: properness criterion for spherical var}
    A spherical variety $X$ is complete if and only if the valuation cone $\Vcal$ is contained in the union of colored cones in the colored fan of $X$. In particular, if $X$ is simple, then $X$ is a complete variety if and only if $\Ccal(X)$ is generated by $\epsilon(\Fcal(X))$ and $\Vcal$.
\end{lemma}

\begin{lemma}[{\cite[Lemma 7.5]{Knop1991LunaVustTheory}}] \label{lem: stabilizer of colors and isotropy group of closed orbit}
    Let $O$ be a homogeneous spherical variety. Suppose that $X$ is a simple $O$-embedding and its unique closed orbit $Y$ is projective. If $T \subset G$ is an arbitrary maximal torus which is equal to $g^{-1} \cdot T' \cdot g$ for some $g \in G$ and $w_{0}$ is a representative of the longest element in the Weyl group of $(G, \, T)$, then the isotropy group of $Y$ containing $B := g^{-1} \cdot B' \cdot g$ is equal to
    \[
        \bigcap_{\Dcal \in \Dcal(O) \setminus \Fcal(X)} (w_{0} \cdot g^{-1} \cdot \text{Stab}_{G}(\Dcal) \cdot g \cdot w_{0}^{-1})^{-}
    \]
    where the colored data $\Dcal(O)$ and $\Fcal(X)$ are defined with respect to $B'$.
\end{lemma}

From now on, let us concentrate on symmetric varieties.

\begin{defn}
    A homogeneous variety $G/K$ is called \emph{symmetric} if there is a nontrivial holomorphic involution $\sigma: G \rightarrow G$ such that $G^{\sigma} \subset K \subset N_{G}(G^{\sigma})$ where $G^{\sigma}$ is the fixed point subgroup. A normal $G$-variety which contains an open $G$-orbit isomorphic to a homogeneous symmetric variety is also called \emph{symmetric}.
\end{defn}

A homogeneous symmetric variety is spherical (\cite[Theorem 26.14]{Timashev2011HomogeneousSpaces}, \cite[Section 1.3]{deConciniProcesi2006CompleteSymmetric}). Therefore we may apply Theorem \ref{thm: classification of spherical} to symmetric varieties. Indeed, Vust (\cite{Vust1990PlongementsEspaces}) obtained a more practical description of colored data for symmetric varieties, which is explained in the rest of this section, following \cite[Section 26]{Timashev2011HomogeneousSpaces}. For simplicity, we only consider the case when $K = G^{\sigma}$, which is a connected reductive subgroup since $G$ is simple and simply connected (\cite[Section 8]{Steinberg1968EndomorphismsLinear}).

For an involution $\sigma: G \rightarrow G$ and the homogeneous symmetric variety $O = G/G^{\sigma}$, let $T'$ be a maximally $\sigma$-split torus in $G$. That is, $T'$ is a maximal torus such that $T'$ is $\sigma$-stable and $\dim \{t \in T' : \sigma(t) = t^{-1}\}$ is maximal among all maximal tori. As before, $B'$ is a Borel subgroup containing $T'$. Then define $R'$ and $S'$ as the root system and the set of simple roots defined by $(G, \, T',\, B')$, respectively. Let $T'_{1}$ be the identity component of $\{t \in T' : \sigma(t) = t^{-1}\}$ so that $T'_{1}$ is a subtorus of $T'$. Consider subsets
\begin{align*}
    R'_{O}&:= \{\overline{\alpha'} \in \chi(T'_{1}) : \alpha' \in R'\} \setminus \{0\}, \\
    S'_{O}&:= \{\overline{\alpha'_{i}} \in \chi(T'_{1}) : \alpha'_{i} \in S'\} \setminus \{0\}
\end{align*}
where $\overline{\alpha'} := \alpha'|_{T'_{1}}$. Then $R'_{O}$ becomes a root system of $\chi(T_{1}') \otimes_{\ZZ} \QQ$ with simple roots in $S_{O}'$ (\cite[Lemma 26.16]{Timashev2011HomogeneousSpaces}), called the \emph{restricted root system}. Moreover, the lattice $\Lambda_{O}$ is isomorphic to the character group $\chi(T'/T' \cap G^{\sigma})= \chi(T_{1}'/T_{1}' \cap G^{\sigma})$. This lattice is a sublattice of $\chi(T'_{1})$ with finite index, thus the vector space $\Ecal$ is identified with $\chi_{*}(T'_{1}) \otimes \QQ$.
\begin{theorem}[{\cite[Section 26]{Timashev2011HomogeneousSpaces}, \cite[Section 2.4]{Vust1990PlongementsEspaces}, \cite[Section 2]{Ruzzi2011SmoothProjective}}] \label{thm: colored data of symmetric var}
    For a homogeneous symmetric variety $O = G/G^{\sigma}$, via the isomorphism $\Ecal \simeq \chi_{*}(T'_{1}) \otimes \QQ$, we have the following identifications:
    \begin{enumerate}
        \item The lattice $\Lambda_{O}$ is identified with the doubled weight lattice $2 \cdot (\ZZ \langle (R'_{O})^{\vee} \rangle)^{*} \subset \chi(T'_{1}) \otimes \QQ$.
        \item The images $\epsilon(\Dcal(O))$ in $\Ecal$ are exactly the halves of restricted simple coroots $\frac{1}{2} (S'_{O})^{\vee}$.
        \item $\Vcal$ is identified with the negative Weyl chamber of $R_{O}'$ in $\Ecal$.
    \end{enumerate}
    Moreover, if $G^{\sigma}$ is semi-simple, then the map $\epsilon : \Dcal(O) \rightarrow \Ecal$ is injective. In this case, if $\Dcal \in \Dcal(O)$ is sent to $\frac{1}{2} \lambda^{\vee}$ for some $\lambda \in S'_{O}$, then the stabilizer $\text{Stab}_{G}(\Dcal)$ of the divisor $\Dcal \subset O$ is
    \[
        \text{Stab}_{G}(\Dcal) = P'_{\{\alpha'_{j} \in S' : \overline{\alpha'_{j}} = \lambda \}},
    \]
    i.e. the parabolic subgroup containing $B'$ and generated by simple roots $\alpha'_{k} \in S'$ such that either $\overline{\alpha'_{k}} = 0$ or $\overline{\alpha'_{k}} \in S'_{O} \setminus \{ \lambda \}$.
\end{theorem}

There is a diagram which encodes information on the action of the involution $\sigma$ on the roots of $\gfr$, called the \emph{Satake diagram} (\cite[Section 26]{Timashev2011HomogeneousSpaces}), which can be constructed as follows.
\begin{enumerate}
    \item Start with the Dynkin diagram of $G$.
    \item For every simple root $\alpha'_{i}$ (with respect to $B'$ and $T'$ as before) which is $\sigma$-stable, mark the corresponding node by black.
    \item Mark the nodes corresponding to $\sigma$-unstable simple roots by white.
    \item If two $\sigma$-unstable simple roots $\alpha'_{i} \not= \alpha'_{j}$ satisfy $\overline{\alpha'_{i}} = \overline{\alpha'_{j}}$, then join the corresponding (white) nodes by a two-headed arrow.
\end{enumerate}
Satake diagrams play an important role in the classification of homogeneous symmetric varieties. For details and the list of all possible Satake diagrams arising from simple Lie groups, we refer to \cite[Section 26.5]{Timashev2011HomogeneousSpaces} and \cite[Table 1]{Ruzzi2010GeometricalDescription}.

\section{Types of Conics in Adjoint Varieties} \label{section: classes of conics}

Our main goal is to study deformations of smooth conics in each adjoint variety $Z_{\gfr} \subset \PP(\gfr)$. Here, \emph{conic} means a closed subscheme of $Z_{\gfr}$ whose Hilbert polynomial with respect to $\Ocal_{\PP(\gfr)}(1)|_{Z_{\gfr}}$ is equal to $2m+1 \in \QQ[m]$. Additionally, we assume that
\begin{equation} \tag{A} \label{assumption: pic num 1 not pn}
    Z_{\gfr} \text{ is of Picard number 1 and not biholomorphic to $\PP^{2n+1}$}.
\end{equation}
This assumption (\ref{assumption: pic num 1 not pn}) is assumed throughout this paper.

\begin{rmk}
    The assumption (\ref{assumption: pic num 1 not pn}) is satisfied if and only if one of the following conditions holds:
    \begin{enumerate}
        \item $\gfr$ is neither of type $A$, nor of type $C$.
        \item $\Ocal_{\PP(\gfr)}(1)|_{Z_{\gfr}}$ generates the Picard group of $Z_{\gfr}$.
    \end{enumerate}
    If $\gfr$ is of type $C$, then $(Z_{\gfr}, \, \Ocal_{\PP(\gfr)}(1)|_{Z_{\gfr}}) \simeq (\PP^{2n+1}, \, \Ocal_{\PP^{2n+1}}(2))$. Thus smooth conics in $Z_{\gfr}$ are exactly linear lines in $\PP^{2n+1}$. If $\gfr$ is of type $A$, then $(Z_{\gfr},\, \Ocal_{\PP(\gfr)}(1)|_{Z_{\gfr}})$ is isomorphic to the projectivized cotangent bundle $\PP (T^{*}\PP^{n+1})$ over the projective space $\PP^{n+1}$ equipped with the relative $\Ocal(1)$.
\end{rmk}

For each adjoint variety $Z_{\gfr} \subset \PP(\gfr)$ satisfying the assumption (\ref{assumption: pic num 1 not pn}), by taking closures, define two subsets
\[
    \Hbf(Z_{\gfr}) := \overline{\{\text{smooth conics in $Z_{\gfr}$}\}} \subset \text{Hilb}^{sn}_{2m+1}(Z_{\gfr}, \, \Ocal_{\PP(\gfr)}(1)|_{Z_{\gfr}})
\]
and
\[
    \Cbf(Z_{\gfr}) := \overline{\{\text{smooth conics in $Z_{\gfr}$}\}} \subset \text{Chow}_{1,\,2}(Z_{\gfr}, \, \Ocal_{\PP(\gfr)}(1)|_{Z_{\gfr}}).
\]
Since smoothness is an open property, these two subsets are the unions of some irreducible components of the Hilbert scheme and the Chow scheme. In fact, thanks to the following theorem, $\Hbf(Z_{\gfr})$ and $\Cbf(Z_{\gfr})$ are irreducible.

\begin{theorem}[\cite{KimPandharipande2001ConnectednessModuli}, \cite{Thomsen1996IrreducibilityM0n}] \label{thm: irreducibility of moduli of curves}
    For a rational homogeneous space $X$ and a homology class $\beta \in H_{2}(X,\, \ZZ)$, the coarse moduli space of stable curves of genus 0 representing $\beta$ is irreducible whenever it is nonempty.
\end{theorem}

From now on, we regard $\Hbf(Z_{\gfr})$ and $\Cbf(Z_{\gfr})$ as schemes equipped with their reduced scheme structures. Then $\Hbf(Z_{\gfr})$ and $\Cbf(Z_{\gfr})$ are projective $G$-varieties, and $FC$ becomes a $G$-equivariant morphism such that its restriction
\[
    \Hbf(Z_{\gfr}) \setminus \{[\text{double lines}]\} \rightarrow \Cbf(Z_{\gfr}) \setminus \{[\text{double lines}]\}
\]
is an isomorphism by Theorem \ref{thm: map FC}.

Recall that any conic in $\PP(\gfr)$ is contained in a unique linear plane. Thus the following definition makes sense.
\begin{defn}
    A conic in $Z_{\gfr}$ is called \emph{planar} if the unique plane containing it in $\PP(\gfr)$ is also contained in $Z_{\gfr}$. Otherwise the conic is called \emph{non-planar}.
\end{defn}
\noindent In Corollary \ref{coro: only one orbit of planar conics}, we prove that conjugacy classes of planar conics are determined by conjugacy classes of planes.

Furthermore, every singular conic is either a reducible conic (i.e. the union of two lines intersecting at one point) or a double line (i.e. a non-reduced quadric in a plane). Since every reducible conic in a rational homogeneous space admits a smoothing (see for example \cite[Theorem II.7.6]{Kollar1996RationalCurves}), every reducible conic in $Z_{\gfr}$ is a member of $\Hbf(Z_{\gfr})$ and $\Cbf(Z_{\gfr})$.

Let us introduce two more types of smooth conics, using the $G$-invariant contact distribution $D \subset T Z_{\gfr}$.
\begin{defn}
    Let $C$ be a smooth conic in $Z_{\gfr}$.
    \begin{enumerate}
        \item $C$ is called a \emph{twistor conic} if $T_{x} C \not\subset D_{x}$ for every $x \in C$.
        \item $C$ is called a \emph{contact conic} if $T_{x} C \subset D_{x}$ for every $x \in C$.
    \end{enumerate}
\end{defn}
It is well-known that every smooth conic is either a twistor conic or a contact conic. Indeed, if $f: \PP^{1} \rightarrow C \subset Z_{\gfr}$ is a smooth conic, then since $TZ_{\gfr} / D \simeq \Ocal_{\PP(\gfr)}(1)|_{Z_{\gfr}}$, we have a bundle morphism
\[
    \Ocal_{\PP^{1}}(2) \simeq T \PP^{1} \xrightarrow{df} f^{*}(TZ_{\gfr}) \rightarrow f^{*} (TZ_{\gfr} / D) \simeq \Ocal_{\PP^{1}}(2)
\]
which is either an isomorphism or the zero map.

\begin{exa} \label{example: twistor conic and contact conic}
    Let us give some examples of twistor conics and contact conics.
    \begin{enumerate}
        \item Let $C_{\rho}$ be the intersection of $Z_{\gfr}$ and a plane $\PP(E_{\rho}, \, H_{\rho}, \, E_{-\rho})$ in $\PP(\gfr)$. Then $C_{\rho}$ is a smooth conic parametrized by $\overline{\exp(\gfr_{-\rho}) \cdot o}$ where $o:= [E_{\rho}] \in Z_{\gfr}$, hence it is a twistor conic. Indeed, this conic is a fiber of the twistor fibration constructed in \cite{Wolf1965ComplexHomogeneous}.
        \item Every smooth planar conic is a contact conic, since every line in $Z_{\gfr}$ is tangent to $D$ (\cite{Hwang1997RigidityHomogeneous}, \cite{Kebekus2001LinesContact}).
    \end{enumerate}
\end{exa}

Note that if $\gfr = G_{2}$, Example \ref{example: twistor conic and contact conic} does not provide any example of contact conics since there is no plane in $Z_{G_{2}}$ (\cite{LandsbergManivel2003ProjectiveGeometry}). In fact, in Theorem \ref{thm: full orbit structure}, we show that there is no contact conic in $Z_{G_{2}}$.

\subsection{Statement of Main Theorems} \label{subsection: main theorems}

Now we state our main theorems. Recall that we always assume the assumption (\ref{assumption: pic num 1 not pn}).

\begin{theorem} \label{main thm: open symmetric orbit of twistor conics with satake diagram}
    Twistor conics form open $G$-orbits in $\Cbf(Z_{\gfr})$ and $\Hbf(Z_{\gfr})$, and both are isomorphic to a $4n$-dimensional homogeneous symmetric variety $O_{\gfr} := G/K$. Here, $K$ is a connected semi-simple Lie subgroup of $G$ and $K=G^{\sigma}$ for some holomorphic involution $\sigma : G \rightarrow G$. Moreover, the Satake diagram of $O_{\gfr}$ is given in Table \ref{table: Satake diagram and restricted root system}.
\end{theorem}

\begin{table}
    \begin{center}
        \begin{tabular}{|c|c|c|c|c|}
            \hline
            $\gfr$ & $\gfr^{\sigma}$(=Lie algebra of $K$) & Satake Diagram of $O_{\gfr}$ & $R'_{O_{\gfr}}$ & $S'_{O_{\gfr}}$ \\
            \hline
            \hline
            &&&&\\[-1em]
            $B_{r}$ & \multirow{2}{*}{$B_{r-2} \oplus A_{1} \oplus A_{1}$} & \multirow{2}{*}{{\dynkin[labels={1,,4,,,r}] BI}} & \multirow{2}{*}{$B_{4}$} & \multirow{2}{*}{$\lambda_{i}=\overline{\alpha'_{i}}$ ($1 \le i \le 4$)} \\
            ($r \ge 4$) &&&& \\
             &&&&\\[-1em]
            \hline
            &&&&\\[-1em]
            \multirow{2}{*}{$B_{3}$} & \multirow{2}{*}{$A_{1} \oplus A_{1} \oplus A_{1}$} & \multirow{2}{*}{{\dynkin[labels={1,2,3}] B{ooo}}} & \multirow{2}{*}{$B_{3}$} & \multirow{2}{*}{$\lambda_{i}=\overline{\alpha'_{i}}$ ($1 \le i \le 3$)} \\
            &&&& \\
            &&&&\\[-1em]
            \hline
            &&&&\\[-1em]
            $D_{r}$ & \multirow{2}{*}{$D_{r-2} \oplus A_{1} \oplus A_{1}$} & \multirow{2}{*}{{\dynkin[labels={1,,4,,,r-1,r}] D{Ia}}} & \multirow{2}{*}{$B_{4}$}& \multirow{2}{*}{$\lambda_{i}=\overline{\alpha'_{i}}$ ($1 \le i \le 4$)}\\
            ($r \ge 6$) &&&& \\
             &&&&\\[-1em]
            \hline
            &&&&\\[-1em]
            \multirow{2}{*}{$D_{5}$} & \multirow{2}{*}{$D_{3} \oplus A_{1} \oplus A_{1}$} & \multirow{2}{*}{{\dynkin[labels={1,,,4,5}, involutions={[bend left=50]45}] D{ooooo}}} & \multirow{2}{*}{$B_{4}$} & $\lambda_{i}=\overline{\alpha'_{i}}$ ($1 \le i \le 3$) \\
            &&&& $\lambda_{4} = \overline{\alpha'_{4}}=\overline{\alpha'_{5}}$ \\
            &&&&\\[-1em]
            \hline
            &&&&\\[-1em]
            \multirow{2}{*}{$D_{4}$} & \multirow{2}{*}{$A_{1} \oplus A_{1} \oplus A_{1} \oplus A_{1}$} & \multirow{2}{*}{{\dynkin[labels={1,2,3,4}] D{oooo}}} & \multirow{2}{*}{$D_{4}$}& \multirow{2}{*}{$\lambda_{i}=\overline{\alpha'_{i}}$ ($1 \le i \le 4$)} \\
            &&&& \\
            &&&&\\[-1em]
            \hline
            &&&&\\[-1em]
            \multirow{4}{*}{$E_{6}$} & \multirow{4}{*}{$A_{5} \oplus A_{1}$} & \multirow{4}{*}{{\dynkin[upside down, labels*={1,6,2,3,4,5}, involutions={16;35}] E{oooooo}}} & \multirow{4}{*}{$F_{4}$} & $\lambda_{1} = \overline{\alpha'_{1}} = \overline{\alpha'_{5}}$ \\
            &&&& $\lambda_{2} = \overline{\alpha'_{2}} = \overline{\alpha'_{4}}$ \\
            &&&& $\lambda_{3} = \overline{\alpha'_{3}}$ \\
            &&&& $\lambda_{4} = \overline{\alpha'_{6}}$ \\
            &&&&\\[-1em]
            \hline
            &&&&\\[-1em]
            \multirow{4}{*}{$E_{7}$} & \multirow{4}{*}{$D_{6} \oplus A_{1}$} & \multirow{4}{*}{{\dynkin[backwards, upside down, labels={6,,5,4,,2,}] E{o*oo*o*}}} & \multirow{4}{*}{$F_{4}$} & $\lambda_{1} = \overline{\alpha'_{2}}$ \\
            &&&& $\lambda_{2} = \overline{\alpha'_{4}}$ \\
            &&&& $\lambda_{3} = \overline{\alpha'_{5}}$ \\
            &&&& $\lambda_{4} = \overline{\alpha'_{6}}$ \\
            &&&&\\[-1em]
            \hline
            &&&&\\[-1em]
            \multirow{4}{*}{$E_{8}$} & \multirow{4}{*}{$E_{7} \oplus A_{1}$} & \multirow{4}{*}{{\dynkin[backwards, upside down, labels={7,,,,,3,2,1}] E{o****ooo}}} & \multirow{4}{*}{$F_{4}$} & $\lambda_{1} = \overline{\alpha'_{7}}$ \\
            &&&& $\lambda_{2} = \overline{\alpha'_{3}}$ \\
            &&&& $\lambda_{3} = \overline{\alpha'_{2}}$ \\
            &&&& $\lambda_{4} = \overline{\alpha'_{1}}$ \\
            &&&&\\[-1em]
            \hline
            \multirow{2}{*}{$F_{4}$} & \multirow{2}{*}{$C_{3} \oplus A_{1}$} & \multirow{2}{*}{{\dynkin[backwards, labels={4,3,2,1}] F{oooo}}} & \multirow{2}{*}{$F_{4}$} & \multirow{2}{*}{$\lambda_{i} = \overline{\alpha'_{i}}$ ($1 \le i \le 4$)} \\
            &&&& \\
            \hline
            \multirow{2}{*}{$G_{2}$} & \multirow{2}{*}{$A_{1} \oplus A_{1}$} & \multirow{2}{*}{{\dynkin[backwards, labels={2,1}] G{oo}}} & \multirow{2}{*}{$G_{2}$} & \multirow{2}{*}{$\lambda_{i} = \overline{\alpha'_{i}}$ ($1 \le i \le 2$)} \\
            &&&& \\
            \hline
        \end{tabular}
        \caption{\label{table: Satake diagram and restricted root system}The Satake Diagram and the Restricted Root System of $O_{\gfr}$.}
    \end{center}
\end{table}

\begin{theorem} \label{main thm: colored data of chow}
    The normalization $\Cbf^{nor}(Z_{\gfr})$ is a simple spherical $O_{\gfr}$-embedding with its colored cone in Table \ref{table: sph data of Chow}.
\end{theorem}

\begin{table}
    \begin{center}
        \begin{tabular}{|c|c|c|}
            \hline
            $\gfr$ & ($\Ccal(\Cbf^{nor}(Z_{\gfr}))$, $\Fcal(\Cbf^{nor}(Z_{\gfr}))$) \\
            \hline
            \hline
            $B_{r}$ $(r \ge 4)$, & \multirow{2}{*}{($\QQ_{\ge 0} \langle -\gamma_{1}, \, -\gamma_{2}, \, -\gamma_{4}, \, \lambda_{2}^{\vee}, \, \lambda_{4}^{\vee} \rangle$, $\{\Dcal_{2}, \, \Dcal_{4}\}$)} \\
            $D_{r}$ $(r \ge 5)$ & \\
            \hline
            $B_{3}$ & ($\QQ_{\ge 0} \langle -\gamma_{1}, \, -\gamma_{2}, \, -\gamma_{3}, \, \lambda_{2}^{\vee}\rangle$, $\{\Dcal_{2}\}$) \\
            \hline
            $D_{4}$ & ($\QQ_{\ge 0} \langle -\gamma_{1}, \, -\gamma_{2}, \, -\gamma_{3}, \, -\gamma_{4}, \, \lambda_{2}^{\vee}\rangle$, $\{\Dcal_{2}\}$) \\
            \hline
            $E_{r}$ $(r = 6,\,7,\,8)$, & \multirow{2}{*}{($\QQ_{\ge 0} \langle -\gamma_{1}, \, -\gamma_{4}, \, \lambda_{1}^{\vee}, \, \lambda_{2}^{\vee}, \, \lambda_{4}^{\vee} \rangle$, $\{\Dcal_{1}, \, \Dcal_{2}, \, \Dcal_{4}\}$)} \\
            $F_{4}$ & \\
            \hline
            $G_{2}$ & ($\QQ_{\ge 0} \langle -\gamma_{2}, \, \lambda_{2}^{\vee} \rangle$, $\{\Dcal_{2}\}$)\\
            \hline
        \end{tabular}
        \caption{\label{table: sph data of Chow} The Colored Cone of $\Cbf^{nor}(Z_{\gfr})$ in $\QQ \langle (R'_{O_{\gfr}})^{\vee} \rangle$.}
    \end{center}
\end{table}

\begin{theorem} \label{main thm: colored data of hilb}
    The normalization $\Hbf^{nor}(Z_{\gfr})$ is a spherical $O_{\gfr}$-embedding, and its colored fan consists of maximal colored cones listed in Table \ref{table: sph data of Hilb} and their colored faces. In particular, $\Hbf^{nor}(Z_{\gfr})$ is simple as a spherical variety if and only if $\gfr$ is of an exceptional type.
\end{theorem}

\begin{table}
    \begin{center}
        \begin{tabular}{|c|c|c|}
            \hline
            $\gfr$ & Maximal Colored Cones of The Colored Fan of $\Hbf^{nor}(Z_{\gfr})$ \\
            \hline
            \hline
            $B_{r}$ $(r \ge 4)$, & ($\QQ_{\ge 0}\langle -\gamma_{2}, \, -\gamma_{4}, \, \lambda_{2}^{\vee}, \, \lambda_{4}^{\vee} \rangle$, $\{\Dcal_{2}, \, \Dcal_{4}\}$), \\
            $D_{r}$ $(r \ge 5)$ & ($\QQ_{\ge 0}\langle -\gamma_{1}, \, -\gamma_{2}, \, -\gamma_{4}, \, \lambda_{2}^{\vee}\rangle$, $\{\Dcal_{2}\}$) \\
            \hline
            $B_{3}$ & ($\QQ_{\ge 0}\langle -\gamma_{i}, \, -\gamma_{2}, \, \lambda_{2}^{\vee} \rangle$, $\{\Dcal_{2}\}$) for $i = 1,\,3$ \\
            \hline
            $D_{4}$ & ($\QQ_{\ge 0} \langle -\gamma_{j} ,\, \lambda_{2}^{\vee} : j \in \{1,\,2,\,3,\,4\} \setminus \{i\} \rangle$, $\{\Dcal_{2}\}$) for $i = 1,\,3,\,4$ \\
            \hline
            $E_{r}$ $(r = 6,\,7,\,8)$, & \multirow{2}{*}{($\QQ_{\ge 0} \langle -\gamma_{1}, \, -\gamma_{4}, \, \lambda_{1}^{\vee}, \, \lambda_{4}^{\vee} \rangle$, $\{\Dcal_{1}, \, \Dcal_{4}\}$)} \\
            $F_{4}$  & \\
            \hline
            $G_{2}$ & ($\QQ_{\ge 0} \langle -\gamma_{2}, \, \lambda_{2}^{\vee} \rangle$, $\{\Dcal_{2}\}$) \\
            \hline
        \end{tabular}
        \caption{\label{table: sph data of Hilb} The Colored Fan of $\Hbf^{nor}(Z_{\gfr})$ in $\QQ \langle (R'_{O_{\gfr}})^{\vee} \rangle$.}
    \end{center}
\end{table}

The proof of Theorem \ref{main thm: open symmetric orbit of twistor conics with satake diagram} is given in Section \ref{subsection: twistor conics}, and the proofs of Theorem \ref{main thm: colored data of chow} and Theorem \ref{main thm: colored data of hilb} are given in Section \ref{section: proof of main theorems}.

Let us explain the notation in Table \ref{table: Satake diagram and restricted root system}, Table \ref{table: sph data of Chow} and Table \ref{table: sph data of Hilb}, assuming Theorem \ref{main thm: open symmetric orbit of twistor conics with satake diagram}. Let $T'$ be a maximally $\sigma$-split torus. Then there is $g \in G$ such that $T'= g\cdot T \cdot g^{-1}$, and we choose a Borel subgroup $B':= g \cdot B \cdot g^{-1}$ to construct the ingredients in Section \ref{section: Luna Vust theorey}. For example, the root system $R'$ and its simple roots $S'$ are given by
\[
    R' = R \circ Ad_{g^{-1}} = \left\{\alpha' := \alpha \circ Ad_{g^{-1}} : \alpha \in R\right\}, \quad S' = S \circ Ad_{g^{-1}} = \left\{\alpha_{i}' := \alpha_{i} \circ Ad_{g^{-1}} : \alpha_{i} \in S\right\}.
\]
Note that since $K= G^{\sigma}$ is semi-simple by Theorem \ref{main thm: open symmetric orbit of twistor conics with satake diagram}, the map $\epsilon: \Dcal(O_{\gfr}) \rightarrow \Ecal$ is bijective onto $\frac{1}{2} \cdot (S'_{O_{\gfr}})^{\vee}$ by Theorem \ref{thm: colored data of symmetric var}. We index restricted simple roots $S'_{O_{\gfr}} = \{\lambda_{1}, \, \lambda_{2}, \, \ldots, \, \lambda_{m}\}$ as in \cite[Reference Chapter, Table 1]{OnishchikVinberg1990LieGroups}, and put $\Dcal_{i} := \epsilon^{-1} \left(\lambda_{i}^{\vee}/2\right) \in \Dcal(O_{\gfr})$. In this notation, the (positive) Weyl chamber is given by $-\Vcal = \QQ_{\ge 0} \langle \gamma_{1}, \, \ldots \, \gamma_{m} \rangle$ where $\gamma_{j}$ are defined by the relations $\langle \lambda_{i}, \,  \gamma_{j}\rangle = \delta_{ij}$. The expression of $\gamma_{j}$ in terms of the coroots $\{\lambda_{1}^{\vee}, \, \ldots,\, \lambda_{m}^{\vee}\}$ can be read off from the $j$th rows of the matrices in \cite[Reference Chapter, Table 2]{OnishchikVinberg1990LieGroups} as follows.
\begin{enumerate}
    \item If $R'_{O_{\gfr}} = B_{m}$,
    \begin{align*}
        \gamma_{i} &:= \lambda_{1}^{\vee} + \cdots +(i-1) \cdot \lambda_{i-1}^{\vee} + i \cdot \lambda_{i}^{\vee} + \cdots + i \cdot \lambda_{m-1}^{\vee} + \frac{i}{2} \cdot \lambda_{m}^{\vee}, \quad \forall i < m \\
        \gamma_{m} &:= \lambda_{1}^{\vee} + \cdots + (m-1) \cdot \lambda_{m-1}^{\vee} + \frac{m}{2} \cdot \lambda_{m}^{\vee}.
    \end{align*}
    \item If $R'_{O_{\gfr}} = D_{4}$,
    \begin{align*}
        \gamma_{1}&:= \lambda_{1}^{\vee} +  \lambda_{2}^{\vee} + \frac{1}{2} \lambda_{3}^{\vee} + \frac{1}{2} \lambda_{4}^{\vee}, \\
        \gamma_{2}&:= \lambda_{1}^{\vee} + 2 \lambda_{2}^{\vee}+ \lambda_{3}^{\vee} + \lambda_{4}^{\vee}, \\
        \gamma_{3} &:= \frac{1}{2} \lambda_{1}^{\vee} + \lambda_{2}^{\vee} + \lambda_{3}^{\vee} + \frac{1}{2}\lambda_{4}^{\vee}, \\
        \gamma_{4} &:= \frac{1}{2}\lambda_{1}^{\vee} + \lambda_{2}^{\vee} + \frac{1}{2} \lambda_{3}^{\vee} + \lambda_{4}^{\vee}.
    \end{align*}
    \item If $R'_{O_{\gfr}} = F_{4}$,
    \begin{align*}
        \gamma_{1} &:= 2 \lambda_{1}^{\vee} + 3 \lambda_{2}^{\vee} + 4 \lambda_{3}^{\vee} + 2 \lambda_{4}^{\vee}, \\
        \gamma_{2} &:= 3 \lambda_{1}^{\vee} + 6 \lambda_{2}^{\vee} + 8 \lambda_{3}^{\vee} + 4 \lambda_{4}^{\vee} , \\
        \gamma_{3} &:= 2 \lambda_{1}^{\vee} + 4 \lambda_{2}^{\vee} + 6 \lambda_{3}^{\vee} + 3 \lambda_{4}^{\vee}, \\
        \gamma_{4} &:= \lambda_{1}^{\vee} + 2 \lambda_{2}^{\vee} + 3 \lambda_{3}^{\vee} + 2 \lambda_{4}^{\vee}.
    \end{align*}
    \item If $R'_{O_{\gfr}} = G_{2}$,
    \[
        \gamma_{1} := 2 \lambda_{1}^{\vee} + 3 \lambda_{2}^{\vee}, \quad \gamma_{2}:= \lambda_{1}^{\vee} + 2\lambda_{2}^{\vee}.
    \]
\end{enumerate}

Let us close this section after collecting colored faces of the colored cones in Table \ref{table: sph data of Chow} and Table \ref{table: sph data of Hilb}. Since each colored face $(\Ccal, \, \Fcal)$ is determined by its underlying cone $\Ccal$, it suffices to write $\Ccal$. In the following list, we classify the colored faces according to their dimensions.

\begin{enumerate}
    \item The nonzero non-maximal elements in the colored fan defined by Table \ref{table: sph data of Chow}:
    \begin{enumerate}
        \item When $\gfr$ is $B_{r}$ with $r \ge 4$ or $D_{r}$ with $r \ge 5$:
        \begin{enumerate}
            \item $\dim=3$: $\QQ_{\ge 0}\langle -\gamma_{1}, \, -\gamma_{2}, \, -\gamma_{4}\rangle$, $\QQ_{\ge 0} \langle -\gamma_{2}, \, -\gamma_{4}, \, \lambda_{4}^{\vee} \rangle$.
            \item $\dim=2$: $\QQ_{\ge 0}\langle -\gamma_{i}, \, -\gamma_{j} \rangle$ for $i\not=j \in \{1,\,2,\,4\}$, $\QQ_{\ge 0}\langle -\gamma_{4}, \, \lambda_{4}^{\vee}\rangle$.
            \item $\dim=1$: $\QQ_{\ge 0}\langle -\gamma_{i}\rangle$ for $i=1,\,2,\,4$.
        \end{enumerate}
        \item When $\gfr = B_{3}$:
        \begin{enumerate}
            \item $\dim=2$: $\QQ_{\ge 0}\langle -\gamma_{1}, \, -\gamma_{2} \rangle$, $\QQ_{\ge 0}\langle -\gamma_{2}, \, -\gamma_{3}\rangle$.
            \item $\dim=1$: $\QQ_{\ge 0}\langle -\gamma_{i}\rangle$ for $i=1,\,2,\,3$.
        \end{enumerate}
        \item When $\gfr = D_{4}$:
        \begin{enumerate}
            \item $\dim=3$: $\QQ_{\ge 0}\langle  -\gamma_{2}, \, -\gamma_{i}, \, -\gamma_{j}\rangle$ for $i \not= j \in \{1,\,3,\,4\}$.
            \item $\dim=2$: $\QQ_{\ge 0}\langle -\gamma_{i}, \, -\gamma_{j} \rangle$ for $i \not= j \in \{1,\,2,\,3,\,4\}$.
            \item $\dim=1$: $\QQ_{\ge 0}\langle -\gamma_{i}\rangle$ for $i=1,\,2,\,3,\,4$.
        \end{enumerate}
        \item When $\gfr$ is one of $E_{6},\,E_{7},\,E_{8}$ and $F_{4}$:
        \begin{enumerate}
            \item $\dim=3$: $\QQ_{\ge 0}\langle -\gamma_{1}, \, -\gamma_{4},\, \lambda_{1}^{\vee}\rangle$.
            \item $\dim=2$: $\QQ_{\ge 0}\langle -\gamma_{1}, \, -\gamma_{4} \rangle$, $\QQ_{\ge 0}\langle -\gamma_{1}, \, \lambda_{1}^{\vee}\rangle$.
            \item $\dim=1$: $\QQ_{\ge 0}\langle -\gamma_{i}\rangle$ for $i=1,\,4$.
        \end{enumerate}
        \item When $\gfr = G_{2}$:
        \begin{enumerate}
            \item $\dim=1$: $\QQ_{\ge 0}\langle -\gamma_{2}\rangle$.
        \end{enumerate}
    \end{enumerate}

    \item The nonzero non-maximal elements in the colored fan defined by Table \ref{table: sph data of Hilb}:
    \begin{enumerate}
        \item When $\gfr$ is $B_{r}$ with $r \ge 4$ or $D_{r}$ with $r \ge 5$:
        \begin{enumerate}
            \item $\dim = 3$: $\QQ_{\ge 0}\langle -\gamma_{1}, \, -\gamma_{2}, \, -\gamma_{4}\rangle$, $\QQ_{\ge 0} \langle -\gamma_{2}, \, -\gamma_{4}, \, \lambda_{4}^{\vee} \rangle$, $\QQ_{\ge 0} \langle -\gamma_{2}, \, -\gamma_{i}, \, \lambda_{2}^{\vee} \rangle$ for $i \in \{1,\,4\}$.
            \item $\dim=2$: $\QQ_{\ge 0}\langle -\gamma_{i}, \, -\gamma_{j} \rangle$ for $i\not=j \in \{1,\,2,\,4\}$, $\QQ_{\ge 0}\langle -\gamma_{k}, \, \lambda_{k}^{\vee}\rangle$ for $k \in \{2,\,4\}$.
            \item $\dim=1$: $\QQ_{\ge 0}\langle -\gamma_{i}\rangle$ for $i=1,\,2,\,4$.
        \end{enumerate}
        \item When $\gfr = B_{3}$:
        \begin{enumerate}
            \item $\dim=2$: $\QQ_{\ge 0}\langle -\gamma_{1}, \, -\gamma_{2} \rangle$, $\QQ_{\ge 0}\langle -\gamma_{2}, \, -\gamma_{3}\rangle$, $\QQ_{\ge 0}\langle -\gamma_{2}, \, \lambda_{2}^{\vee}\rangle$.
            \item $\dim=1$: $\QQ_{\ge 0}\langle -\gamma_{i}\rangle$ for $i=1,\,2,\,3$.
        \end{enumerate}
        \item When $\gfr = D_{4}$:
        \begin{enumerate}
            \item $\dim=3$: $\QQ_{\ge 0}\langle  -\gamma_{2}, \, -\gamma_{i}, \, -\gamma_{j}\rangle$ for $i \not= j \in \{1,\,3,\,4\}$, $\QQ_{\ge 0}\langle -\gamma_{2}, \, -\gamma_{k}, \, \lambda_{2}^{\vee} \rangle$ for $k \in \{1,\,3,\,4\}$.
            \item $\dim=2$: $\QQ_{\ge 0}\langle -\gamma_{i}, \, -\gamma_{j} \rangle$ for $i \not= j \in \{1,\,2,\,3,\,4\}$, $\QQ_{\ge 0}\langle -\gamma_{2}, \, \lambda_{2}^{\vee} \rangle$.
            \item $\dim=1$: $\QQ_{\ge 0}\langle -\gamma_{i}\rangle$ for $i=1,\,2,\,3,\,4$.
        \end{enumerate}
        \item When $\gfr$ is one of $E_{6},\,E_{7},\,E_{8}$ and $F_{4}$:
        \begin{enumerate}
            \item $\dim=3$: $\QQ_{\ge 0}\langle -\gamma_{1}, \, -\gamma_{4},\, \lambda_{1}^{\vee}\rangle$, $\QQ_{\ge 0}\langle -\gamma_{1}, \, -\gamma_{4}, \, \lambda_{4}^{\vee} \rangle$.
            \item $\dim=2$: $\QQ_{\ge 0}\langle -\gamma_{1}, \, -\gamma_{4} \rangle$, $\QQ_{\ge 0}\langle -\gamma_{1}, \, \lambda_{1}^{\vee}\rangle$, $\QQ_{\ge 0} \langle -\gamma_{4}, \, \lambda_{4}^{\vee} \rangle$.
            \item $\dim=1$: $\QQ_{\ge 0}\langle -\gamma_{i}\rangle$ for $i=1,\,4$.
        \end{enumerate}
        \item When $\gfr = G_{2}$:
        \begin{enumerate}
            \item $\dim=1$: $\QQ_{\ge 0}\langle -\gamma_{2}\rangle$.
        \end{enumerate}
    \end{enumerate}
\end{enumerate}

\section{Geometry of Conics} \label{section: counting conjugacy classes of conics}

Before studying geometry of conics in $Z_{\gfr}$, let us introduce an additional notation. For a simple root $\alpha_{i} \in S$ and a root $\alpha \in R$, let $m_{i}(\alpha)$ be the coefficient of $\alpha_{i}$ in $\alpha$. Under the assumption (\ref{assumption: pic num 1 not pn}), there is exactly one simple root, say $\alpha_{j_{0}}$, which is not orthogonal to the highest root $\rho$. This $\alpha_{j_{0}}$ is a long root, and satisfies
\[
    m_{j_{0}}(\alpha) = \langle \alpha \, | \, \rho \rangle, \quad \forall \alpha \in R.
\]
Then the Lie algebra $\pfr$ and the vector spaces $D_{o} \subset T_{o}(Z_{\gfr})$ can be described as follows:
\[
    \pfr = \tfr \oplus \bigoplus_{\substack{\alpha \in R\\ m_{j_{0}}(\alpha) = 0, \,1,\,2}} \gfr_{\alpha}, \quad T_{o}(Z_{\gfr}) \simeq \bigoplus_{\substack{\alpha \in R\\ m_{j_{0}}(\alpha) = -1,\,-2}} \gfr_{\alpha}, \quad D_{o} \simeq \bigoplus_{\substack{\alpha \in R\\ m_{j_{0}}(\alpha) = -1}} \gfr_{\alpha}.
\]

\subsection{Twistor Conics} \label{subsection: twistor conics}

In this section, we prove Theorem \ref{main thm: open symmetric orbit of twistor conics with satake diagram}, and that in every tangent direction off the contact distribution, there is exactly one twistor conic. Let us use the following known lemma.

\begin{lemma}[{\cite[Lemma 5]{HwangMok1999HolomorphicMaps}}] \label{lem: transitive action on general direction}
    Under the assumption (\ref{assumption: pic num 1 not pn}), the unipotent radical $R^{u}(P)$ of the isotropy group $P$ at $o \in Z_{\gfr}$ acts transitively on the open subset $\PP(T_{o}Z_{\gfr}) \setminus \PP(D_{o})$ in the projectivized tangent space.
\end{lemma}

\begin{lemma} \label{lem: normal bundle single orbit of twistor conic}
    Recall that $n$ is the integer such that $\dim_{\CC}Z_{\gfr}=2n+1$.
    \begin{enumerate}
        \item The normal bundle of a twistor conic in $Z_{\gfr}$ is isomorphic to $\Ocal_{\PP^{1}}(1)^{\oplus 2n}$.
        \item $\dim \Hbf(Z_{\gfr})=\dim \Cbf(Z_{\gfr}) = 4n$.
        \item Arbitrary two twistor conics in $Z_{\gfr}$ are $G$-conjugate to each other.
    \end{enumerate}
\end{lemma}
\begin{proof}
    Let $C$ be a twistor conic and $f : \PP^{1} \rightarrow C \subset Z_{\gfr}$ an embedding. By Lemma \ref{lem: transitive action on general direction} and \cite[Theorem II.3.11]{Kollar1996RationalCurves}, $f$ is free over $0 \mapsto f(0)$, i.e.
    \[
        f^{*}T Z_{\gfr} \simeq \bigoplus_{i = 1}^{2n+1}\Ocal_{\PP^{1}}(a_{i}) , \quad \text{for some }a_{1} \ge \cdots \ge a_{2n+1} > 0.
    \]
    Since the anti-canonical bundle $K_{Z_{\gfr}}^{-1}$ is isomorphic to $\Ocal_{\PP(\gfr)} (n+1) |_{Z_{\gfr}}$ (see \cite{Lebrun1995FanoManifolds}) and $C$ is a conic, we have
    \[
        2(n+1) = \deg_{\PP^{1}} f^{*} K_{Z_{\gfr}}^{-1} = \deg_{\PP^{1}} f^{*} TZ_{\gfr} = \sum_{i=1}^{2n+1}a_{i}.
    \]
    This is possible only if $a_{1} = 2$ and $a_{2} = \cdots = a_{2n+1} = 1$, hence the normal bundle of $C$ is isomorphic to $\Ocal_{\PP^{1}}(1)^{\oplus 2n}$ and the dimension of the Hilbert scheme at $[C]$ is $4n$. The Chow scheme is also $4n$-dimensional at $[C]$ by Theorem \ref{thm: map FC}.
    
    Now consider the space $\text{Hom}_{bir}(\PP^{1}, \, Z_{\gfr})$ of morphisms from $\PP^{1}$ to $Z_{\gfr}$ which are birational onto their images. Let $V$ be the closure of the $G \times \text{Aut}(\PP^{1})$-orbit containing $[f]$ in $\text{Hom}_{bir}(\PP^{1}, \, Z_{\gfr})$. By Lemma \ref{lem: transitive action on general direction}, for arbitrary $x \in Z_{\gfr}$, $\text{Locus}(V, \, 0\mapsto x)$ is open in $Z_{\gfr}$. Thus by the proof of \cite[Proposition IV.2.5]{Kollar1996RationalCurves}, for general points $x$ and $y$ in $Z_{\gfr}$, we have
    \begin{align*}
        \dim V &= \dim \{[h] \in V: h(0) = x, \, h(\infty) = y\} + \dim \text{Locus}(V) + \dim \text{Locus}(V,\, 0 \mapsto x) \\
        &\ge 4n+3.
    \end{align*}
    Therefore by \cite[Theorem II.2.15]{Kollar1996RationalCurves}, the $G$-orbit containing $[C]$ in the Chow scheme is at least $4n$-dimensional, hence each orbit containing a twistor conic is (Zariski) open in the Chow scheme. By Theorem \ref{thm: irreducibility of moduli of curves}, all twistor conics are in the same $G$-orbit.
\end{proof}

\begin{lemma}[See also {\cite{Wolf1965ComplexHomogeneous}}] \label{lem: stab of a twistor conic}
    The stabilizer $\text{Stab}_{G}(C_{\rho})$ of the twistor conic $C_{\rho} = Z_{\gfr} \cap \PP(E_{\rho}, \, H_{\rho}, \, E_{-\rho})$ introduced in Example \ref{example: twistor conic and contact conic} is the connected Lie subgroup $K$ of $G$ associated to the Lie subalgebra
    \[
        \kfr := \tfr \oplus \bigoplus_{\substack{\alpha \in R,\\m_{j_{0}}(\alpha) = 0}} \gfr_{\alpha} \oplus \gfr_{\rho} \oplus \gfr_{-\rho}.
    \]
    In particular, $\text{Stab}_{G}(C_{\rho})$ is a connected semi-simple Lie group.
\end{lemma}

\begin{proof}
    Observe that $g \in G$ stabilizes $C_{\rho}$ if and only if it stabilizes $\PP(E_{\rho}, \, H_{\rho}, \, E_{-\rho})$. Thus the Lie algebra $\kfr$ is contained in the Lie algebra of $\text{Stab}_{G}(C_{\rho})$. For the converse, let $g \in \text{Stab}_{G}(C_{\rho})$, and then claim that $g \in K$. Observe that the $\mathfrak{sl}_{2}(\CC)$ algebra $\CC \cdot H_{\rho} \oplus \gfr_{\rho} \oplus \gfr_{-\rho}$ is contained in $\kfr$, and the corresponding $SL_{2}(\CC)$ acts transitively on $C_{\rho}$. Thus we may assume that $g$ fixes $o \in C_{\rho}$, i.e. $g \in P$. Now consider the Levi decomposition $P = R^{u}(P) \rtimes L$ where $R^{u}(P)$ is the unipotent radical of $P$ and $L$ is the standard Levi subgroup. That is, the Lie algebras of $R^{u}(P)$ and $L$ are given by
    \[
        \bigoplus_{\substack{\alpha \in R\\ m_{j_{0}}(\alpha)=1,\,2}}\gfr_{\alpha}, \quad \text{and}\quad \tfr \oplus \bigoplus_{\substack{\alpha \in R\\m_{j_{0}}(\alpha)=0}} \gfr_{\alpha}, \text{ respectively}.
    \]
    Since $L \subset K$, we may assume that $g \in R^{u}(P)$, say $g = \exp(X)$ for some $X \in \bigoplus_{m_{j_{0}}(\alpha)=1,\,2}\gfr_{\alpha}$. Note that since $g = \exp(X)$ is unipotent, $\exp(tX) \in \text{Stab}_{G}(C_{\rho}) \cap P$ for every $t \in \CC$, hence $\exp(tX)$ stabilizes $T_{o} C_{\rho} \simeq \gfr_{-\rho} \mod \pfr$ in $T_{o} Z_{\gfr} \simeq \gfr / \pfr$. Therefore
    \[
        [X,\, \gfr_{-\rho}] \mod \pfr \quad \subset \quad \gfr_{-\rho} \mod \pfr \quad \text{in $\gfr/\pfr$}.
    \]
    This is possible only if $X \in \gfr_{\rho}$, hence $g \in K$.
\end{proof}

\begin{theorem} \label{thm: exist unique twistor conic in gen direction}
    Let $v$ be a nonzero tangent vector of $Z_{\gfr}$ which does not belong to the contact distribution $D$. Then there is exactly one twistor conic tangent to $v$.
\end{theorem}

\begin{proof}
    By Lemma \ref{lem: transitive action on general direction}, we may assume that $v \in T_{o} C_{\rho}$. Suppose that there is a twistor conic $C$ tangent to $v$ at $o$. By Lemma \ref{lem: normal bundle single orbit of twistor conic}, there is $g \in G$ such that $C = g \cdot C_{\rho}$. We claim that $g$ is indeed contained in $K = \text{Stab}_{G}(C_{\rho})$. Since $K$ contains the Lie subgroup corresponding to $\CC \cdot H_{\rho} \oplus \gfr_{\rho} \oplus \gfr_{-\rho}$, we may assume that $g$ fixes $o$, i.e. $g \in P$. Since the tangent directions of $C$ and $C_{\rho}$ coincide, $g$ stabilizes $T_{o}C_{\rho}$, hence by repeating the argument in the proof of Lemma \ref{lem: stab of a twistor conic}, we conclude that $g \in K$.
\end{proof}

\begin{proposition} \label{prop: orbit of twistor conics is symmetric}
    There is a holomorphic involution $\sigma : G\rightarrow G$ such that the fixed point subgroup $G^{\sigma}$ is $K$.
\end{proposition}
\begin{proof}
    By the proof of \cite[Theorem 5.4]{Wolf1965ComplexHomogeneous}, there is an inner involution on a real form of $\gfr$ such that the $(+1)$-eigenspace is a real form of $\kfr$. By taking its $\CC$-linear extension over $\gfr$, since $G$ is simply connected, we obtain a holomorphic involution $\sigma : G \rightarrow G$, and the $(+1)$-eigenspace of its differential at the identity element is $\kfr$. Since $G^{\sigma}$ is connected (\cite[Theorem 8.1]{Steinberg1968EndomorphismsLinear}), we have $G^{\sigma} = K$.
\end{proof}

\begin{proof}[Proof of Theorem \ref{main thm: open symmetric orbit of twistor conics with satake diagram}]
    By Lemma \ref{lem: normal bundle single orbit of twistor conic}, Lemma \ref{lem: stab of a twistor conic} and Proposition \ref{prop: orbit of twistor conics is symmetric}, twistor conics form open orbits in $\Cbf(Z_{\gfr})$ and $\Hbf(Z_{\gfr})$, which are isomorphic to the homogeneous symmetric variety $O_{\gfr} := G/G^{\sigma}$ of dimension $4n$. From Lemma \ref{lem: stab of a twistor conic} and the classification of Satake diagrams for simple Lie algebras in \cite[Table 1]{Ruzzi2010GeometricalDescription} (see also \cite[Table 26.3]{Timashev2011HomogeneousSpaces}), we easily obtain the Satake diagram of $O_{\gfr}$ for each $\gfr$. Indeed, the Dynkin diagram of $K = G^{\sigma}$ can be obtained by removing $\alpha_{j_{0}}$ in the extended Dynkin diagram of $\gfr$. Also note that $SO(4,\, \CC)$ is of type $A_{1} \times A_{1}$ and $SO(3,\, \CC)$ is of type $A_{1}$.
\end{proof}

\subsection{Double Lines} \label{subsection: double lines}
As we have seen, twistor conics form open orbits in $\Hbf(Z_{\gfr})$ and $\Cbf(Z_{\gfr})$. Singular conics belong to their boundaries, and in this section we focus on double lines. Namely, we show that the closed orbits of $\Hbf(Z_{\gfr})$ and $\Cbf(Z_{\gfr})$ consist of double lines, and compute their isotropy groups.

First, recall the following description of the space of lines.

\begin{theorem}[{\cite[Theorem 4.3]{LandsbergManivel2003ProjectiveGeometry}, \cite[Proposition 5]{Hwang1997RigidityHomogeneous}}] \label{thm: space of lines is homogeneous}
    For a point $o$ in the adjoint variety $Z_{\gfr}$ satisfying the assumption (\ref{assumption: pic num 1 not pn}), the semi-simple part of the isotropy group $P = \text{Stab}_{G}(o)$ acts transitively on the space of lines passing through $o$ in $Z_{\gfr}$. Moreover, the stabilizer of each line is conjugate to the parabolic subgroup $P_{N(\alpha_{j_{0}})}$ of $G$ where $N(\alpha_{j_{0}})$ denotes the set of neighbors of $\alpha_{j_{0}}$ in the Dynkin diagram of $\gfr$.
\end{theorem}

In particular, the subset consisting of double lines in the Chow scheme is homogeneous.

\begin{lemma} \label{lem: B stable line and plane}
    \begin{enumerate}
        \item The line $\Lcal_{B} := \PP(E_{\rho}, \, E_{\rho - \alpha_{j_{0}}})$ is a unique $B$-stable line in $\PP(\gfr)$.
        \item Any $B$-stable plane in $\PP(\gfr)$ contains the $B$-stable line $\Lcal_{B}$, and is of form
        \[
            \PP(E_{\rho}, \, E_{\rho - \alpha_{j_{0}}}, \, E_{\rho - \alpha_{j_{0}} - \beta})
        \]
        where $\beta \in S$ is a neighbor of $\alpha_{j_{0}}$ in the Dynkin diagram of $\gfr$.
        \item In $\PP(\gfr)$, there is no $B$-stable conic which is smooth or reducible.
    \end{enumerate}
\end{lemma}

\begin{proof}
    If $\Lcal \subset \PP(\gfr)$ is a $B$-stable line, then $\Lcal$ contains $o = [E_{\rho}]$ which is the unique $B$-fixed point in $\PP(\gfr)$. Thus $\Lcal = \PP(E_{\rho}, \, v)$ for some $v \in \gfr$. Moreover, since $\alpha_{j_{0}}$ is the unique simple root which is not orthogonal to $\rho$, $\rho - \alpha_{j_{0}}$ is the maximum in $R \setminus \{\rho\}$. By $\bfr$-stability of $\Lcal$, we have $\Lcal = \PP(E_{\rho}, \, E_{\rho - \alpha_{j_{0}}})$.

    If $\Pcal$ is a $B$-stable plane in $\PP(\gfr)$, then $B$ acts on the space of lines in $\Pcal$, which is compact. Thus by the uniqueness of $\Lcal_{B}$, $\Pcal$ contains $\Lcal_{B}$ and $\Pcal = \PP(E_{\rho}, \, E_{\rho - \alpha_{j_{0}}}, \, v)$ for some $v \in \gfr$. By $\bfr$-stability, we may choose $v$ as a root vector corresponding to a maximal element in $R \setminus \{\rho, \, \rho-\alpha_{j_{0}}\}$, which is exactly of form $\rho - \alpha_{j_{0}} - \beta$ for some neighbor $\beta$ of $\alpha_{j_{0}}$ in the Dynkin diagram.

    For the last statement, consider two different lines $\Lcal_{1}$ and $\Lcal_{2}$ such that their union $\Lcal_{1} \cup \Lcal_{2}$ is $B$-stable. Since $B$ is irreducible, each $\Lcal_{i}$ should be $B$-stable, hence $\Lcal_{1}=\Lcal_{2} = \Lcal_{B}$, a contradiction. Now assume that there is a $B$-stable smooth conic $C$ in $\PP(\gfr)$. Then the plane spanned by $C$ is also $B$-stable, hence there is some neighbor $\beta \in S$ of $\alpha_{j_{0}}$ such that the plane $\Pcal := \PP(E_{\rho}, \, E_{\rho - \alpha_{j_{0}}}, \, E_{\rho - \alpha_{j_{0}} - \beta})$ contains $C$. Moreover, $o \in C$ and the line $\Lcal_{B}$ is tangent to $C$ at $o$. Therefore in $\Pcal$, $C$ is defined by a quadratic equation
    \[
        a_{11}x_{1}^{2} + a_{22}x_{2}^{2} + x_{0}x_{2} + a_{12}x_{1}x_{2} = 0
    \]
    for some $a_{ij} \in \CC$ where the homogeneous coordinate on $\Pcal$ is chosen so that $[x_{0}:x_{1}:x_{2}] = [x_{0}E_{\rho} + x_{1}E_{\rho - \alpha_{j_{0}}} + x_{2}E_{\rho - \alpha_{j_{0}} - \beta}]$. Then the above equation should be $B$-stable up to scalar multiplication, however a simple computation shows that for each $H \in \tfr$, $\exp(-H)$ sends the equation to
    \begin{align*}
        0 &= a_{11} \left(x_{1}e^{(\rho-\alpha_{j_{0}})(H)}\right)^{2} + a_{22} \left(x_{2}e^{(\rho-\alpha_{j_{0}}-\beta)(H)}\right)^{2} \\
        & \quad + \left(x_{0}e^{\rho(H)}\right)\left(x_{2}e^{(\rho-\alpha_{j_{0}} - \beta)(H)}\right) + a_{12} \left(x_{1}e^{(\rho - \alpha_{j_{0}})(H)}\right)\left(x_{2}e^{(\rho - \alpha_{j_{0}}-\beta)(H)}\right).
    \end{align*}
    This implies that $a_{11} = a_{22} = a_{12} = 0$, which is impossible.
\end{proof}

It is not difficult to compute all $B$-stable planes in $\PP(\gfr)$ and their stabilizers in $G$ using Lemma \ref{lem: B stable line and plane}. For example, given a $B$-stable plane $\Pcal = \PP(E_{\rho}, \, E_{\rho - \alpha_{j_{0}}}, \, E_{\rho - \alpha_{j_{0}} - \beta})$, it can be shown that its stabilizer is a parabolic subgroup $P_{I}$ where $I \subset S$ is
\[
    (N(\alpha_{j_{0}}) \cup N(\beta)) \setminus \{\alpha_{j_{0}}, \, \beta\} \quad \text{if $\beta$ is long,}
\]
and
\[
    (N(\alpha_{j_{0}}) \cup N(\beta)) \setminus \{\alpha_{j_{0}}\} \quad \text{if $\beta$ is short.}
\]
(Alternatively, $I$ is a set of $\gamma\in S$ such that $\rho - \alpha_{j_{0}} - \gamma \in R \setminus \{\rho - \alpha_{j_{0}} - \beta\}$, or $\rho - \alpha_{j_{0}} - \beta - \gamma \in R$.) These are listed in Table \ref{table: B stable planes}. We also indicate whether a plane is contained in $Z_{\gfr}$ or not, by the following observation: A $B$-stable plane $\PP(E_{\rho}, \, E_{\rho - \alpha_{j_{0}}}, \, E_{\rho - \alpha_{j_{0}} - \beta})$ is contained in $Z_{\gfr}$ if and only if $\beta$ is a long root. Remark that the stabilizers of $B$-stable planes contained in $Z_{\gfr}$ were already obtained by \cite[Theorem 4.9]{LandsbergManivel2003ProjectiveGeometry}, and the stabilizers of $B$-stable planes not in $Z_{\gfr}$ are equal to $\text{Stab}_{G}(\Lcal_{B})$ by Theorem \ref{thm: space of lines is homogeneous}.

\begin{table}
    \begin{center}
        \begin{tabular}{|c|c|c|c|}
            \hline
            $\gfr$ & $B$-stable plane $\Pcal$ & $\text{Stab}_{G}(\Pcal)$ \\
            \hline
            \hline
            \multirow{2}{*}{$B_{r}$ ($r \ge 4$)} & $\PP(E_{\rho}, \, E_{\rho - \alpha_{2}},\, E_{\rho - \alpha_{1} - \alpha_{2}})$ & $P_{\alpha_{3}}$ \\ \cline{2-3}
            &  $\PP(E_{\rho}, \, E_{\rho - \alpha_{2}},\, E_{\rho - \alpha_{2} - \alpha_{3}})$ & $P_{\alpha_{1},\, \alpha_{4}}$ \\
            \hline
            \hline
            \multirow{2}{*}{$B_{3}$} &  $\PP(E_{\rho}, \, E_{\rho - \alpha_{2}},\, E_{\rho - \alpha_{1} - \alpha_{2}})$ & $P_{\alpha_{3}}$ \\ \cline{2-3}
             & $\PP(E_{\rho}, \, E_{\rho - \alpha_{2}},\, E_{\rho - \alpha_{2} - \alpha_{3}})$ ($\not\subset Z_{\gfr}$)& $P_{\alpha_{1}, \, \alpha_{3}}$ \\
            \hline
            \hline
            \multirow{2}{*}{$D_{r}$ ($r \ge 6$)} & $\PP(E_{\rho}, \, E_{\rho - \alpha_{2}},\, E_{\rho - \alpha_{1} - \alpha_{2}})$&$P_{\alpha_{3}}$ \\ \cline{2-3}
            &$\PP(E_{\rho}, \, E_{\rho - \alpha_{2}},\, E_{\rho - \alpha_{2} - \alpha_{3}})$ & $P_{\alpha_{1},\, \alpha_{4}}$\\
            \hline
            \hline
            \multirow{2}{*}{$D_{5}$} & $\PP(E_{\rho}, \, E_{\rho - \alpha_{2}},\, E_{\rho - \alpha_{1} - \alpha_{2}})$&$P_{\alpha_{3}}$ \\ \cline{2-3}
            &$\PP(E_{\rho}, \, E_{\rho - \alpha_{2}},\, E_{\rho - \alpha_{2} - \alpha_{3}})$ & $P_{\alpha_{1}, \, \alpha_{4},\, \alpha_{5}}$\\
            \hline
            \hline
            \multirow{3}{*}{$D_{4}$} & $\PP(E_{\rho}, \, E_{\rho - \alpha_{2}},\, E_{\rho - \alpha_{1} - \alpha_{2}})$ & $P_{\alpha_{3}, \, \alpha_{4}}$ \\ \cline{2-3}
            & $\PP(E_{\rho}, \, E_{\rho - \alpha_{2}},\, E_{\rho - \alpha_{2} - \alpha_{3}})$ & $P_{\alpha_{1}, \, \alpha_{4}}$ \\ \cline{2-3}
            & $\PP(E_{\rho}, \, E_{\rho - \alpha_{2}},\, E_{\rho - \alpha_{2} - \alpha_{4}})$ &$P_{\alpha_{1}, \, \alpha_{3}}$ \\
            \hline
            \hline
            $E_{6}$ & $\PP(E_{\rho}, \, E_{\rho - \alpha_{6}},\, E_{\rho - \alpha_{3} - \alpha_{6}})$ & $P_{\alpha_{2}, \, \alpha_{4}}$ \\
            \hline
            \hline
            $E_{7}$ & $\PP(E_{\rho}, \, E_{\rho - \alpha_{6}},\, E_{\rho - \alpha_{5} - \alpha_{6}})$ & $P_{\alpha_{4}}$ \\
            \hline
            \hline
            $E_{8}$ &$\PP(E_{\rho}, \, E_{\rho - \alpha_{1}},\, E_{\rho - \alpha_{1} - \alpha_{2}})$ & $P_{\alpha_{3}}$ \\
            \hline
            \hline
            $F_{4}$ & $\PP(E_{\rho}, \, E_{\rho - \alpha_{4}},\, E_{\rho - \alpha_{3} - \alpha_{4}})$ & $P_{\alpha_{2}}$ \\
            \hline
            \hline
            $G_{2}$ &$\PP(E_{\rho}, \, E_{\rho - \alpha_{2}},\, E_{\rho - \alpha_{1} - \alpha_{2}})$ ($\not\subset Z_{\gfr}$) & $P_{\alpha_{1}}$ \\
            \hline
        \end{tabular}
        \caption{\label{table: B stable planes}$B$-stable Planes in $\PP(\gfr)$ and Their Stabilizers.}
    \end{center}
\end{table}

\begin{coro} \label{coro: only one orbit of planar conics}
    If $C_{i}$ is a planar conic in a plane $\Pcal_{i} \subset Z_{\gfr}$ for $i=1,\,2$, then $C_{1}$ and $C_{2}$ are $G$-conjugate if and only if $\Pcal_{1}$ and $\Pcal_{2}$ are $G$-conjugate planes and $C_{1}$ are $C_{2}$ isomorphic as schemes.
\end{coro}
\begin{proof}
    Suppose that $\Pcal_{1}$ and $\Pcal_{2}$ are $G$-conjugate. Since the space of linear subspaces in $Z_{\gfr}$ is the disjoint union of rational homogeneous spaces (\cite[Theorem 4.9]{LandsbergManivel2003ProjectiveGeometry}), we may assume that $\Pcal = \Pcal_{1} = \Pcal_{2}$ is $B$-stable. Then it suffices to show that the restriction map $\text{Stab}_{G}(\Pcal) \rightarrow \text{Aut}(\Pcal)$ is surjective. To see this, for a $B$-stable plane $\Pcal = \PP(E_{\rho} ,\, E_{\rho - \alpha_{j_{0}}}, \, E_{\rho - \alpha_{j_{0}} - \beta})$ in $Z_{\gfr}$, observe that the Lie algebra of $\text{Stab}_{G}(\Pcal)$ contains $\gfr_{-\alpha_{j_{0}}}+\gfr_{-\beta}+\gfr_{-\alpha_{j_{0}} - \beta}$ by Table \ref{table: B stable planes}. Moreover, with respect to the homogeneous coordinate $[x:y:z] = [xE_{\rho} +y E_{\rho - \alpha_{j_{0}}}+z E_{\rho - \alpha_{j_{0}} - \beta}]$ on $\Pcal$, $\exp(\gfr_{-\alpha_{j_{0}}}+\gfr_{-\beta}+\gfr_{-\alpha_{j_{0}} - \beta})$ (respectively, $\exp(\gfr_{\alpha_{j_{0}}}+\gfr_{\beta}+\gfr_{\alpha_{j_{0}} + \beta})$) generates all lower (respectively, upper) triangular matrices of which diagonal elements are $1$ in $\text{Aut}(\Pcal) \simeq PGL_{3}(\CC)$. Since the maximal torus $T$ is sent to the group of diagonal matrices, it follows that $\text{Stab}_{G}(\Pcal) \rightarrow \text{Aut}(\Pcal)$ is surjective.
\end{proof}

\begin{coro} \label{coro: number of closed orbits}
    \begin{enumerate}
        \item $\Cbf(Z_{\gfr})$ has a unique closed $G$-orbit $\simeq G/ \text{Stab}_{G}(\Lcal_{B})$ which consists of all double lines in $Z_{\gfr}$.
        \item There is an injective map from the set of closed $G$-orbits in $\Hbf(Z_{\gfr})$ to the set of $B$-stable planes in $\PP(\gfr)$, which is bijective unless $\gfr = B_{3}$. If a closed orbit corresponds to a $B$-stable plane $\Pcal \subset \PP(\gfr)$, then the orbit is isomorphic to $G / \text{Stab}_{G}(\Lcal_{B}) \cap \text{Stab}_{G}(\Pcal)$.
    \end{enumerate}
\end{coro}
\begin{proof}
    Since closed orbits in $\Cbf(Z_{\gfr})$ and $\Hbf(Z_{\gfr})$ are compact, by Lemma \ref{lem: B stable line and plane}, closed orbits must consist of double lines. Thus the first statement follows from Theorem \ref{thm: space of lines is homogeneous}.
    
    For the Hilbert scheme, for each closed orbit $\Ocal$ in $\Hbf(Z_{\gfr})$, consider its unique $B$-fixed point. This point is represented by a double line, say $\Lcal_{\Ocal}$, in $Z_{\gfr}$ such that $(\Lcal_{\Ocal})^{red} = \Lcal_{B}$ by Lemma \ref{lem: B stable line and plane}. Now define $\Pcal_{\Ocal}$ to be the unique plane in $\PP(\gfr)$ which contains $\Lcal_{\Ocal}$ as a closed subscheme. Then the map $\Ocal \mapsto \Pcal_{\Ocal}$ is injective. Since the stabilizer of $\Lcal_{\Ocal}$ is equal to $\text{Stab}_{G}(\Lcal_{B}) \cap \text{Stab}_{G}(\Pcal_{\Ocal})$, we see that $\Ocal \simeq G/\text{Stab}_{G}(\Lcal_{B}) \cap \text{Stab}_{G}(\Pcal_{\Ocal})$.

    For bijectivity, observe that the injective map $\Ocal \mapsto \Pcal_{\Ocal}$ is surjective if there is only one $B$-stable plane in $\PP(\gfr)$, which is the case when $\gfr$ is of an exceptional type. On the other hand, if every $B$-stable plane in $\PP(\gfr)$ is contained in $Z_{\gfr}$, i.e. when $\gfr \not= B_{3}, \, G_{2}$, then every $B$-stable double line in $\PP(\gfr)$ represents a point in $\Hbf(Z_{\gfr})$, hence the map $\Ocal \mapsto \Pcal_{\Ocal}$ is surjective.
\end{proof}

For $B_{3}$, bijectivity of the map in Corollary \ref{coro: number of closed orbits} is proven in Proposition \ref{prop: hilb for B3 is not simple} as an application of spherical geometry and the following lemma.

\begin{lemma} \label{lem: nonplanar double lines exist when B3}
    If $\gfr = B_{3}$, then $Z_{B_{3}}$ contains a non-planar double line which represents a point in $\Hbf(Z_{B_{3}})$.
\end{lemma}
\begin{proof}
    Let us index simple roots of $G_{2}$ and $B_{3}$ so that their Dynkin diagrams are given by
    \begin{center}
        {\dynkin[backwards, labels={\alpha_{2},\alpha_{1}}] G2} for $G_{2}$, $\quad$ and $\quad$ {\dynkin[labels={\beta_{1},\beta_{2},\beta_{3}}] B3} for $B_{3}$.
    \end{center}
    For roots $\alpha \in R_{G_{2}}$ and $\beta \in R_{B_{3}}$, we denote by $(G_{2})_{\alpha}$ and $(B_{3})_{\beta}$ the corresponding root spaces. Root vectors are denoted by $E_{\alpha} \in (G_{2})_{\alpha}$ and $E_{\beta} \in (B_{3})_{\beta}$ as before. Then there is an embedding $G_{2} \hookrightarrow B_{3}$ as a Lie subalgebra so that
    \[
        \left\{ \begin{array}[pos]{l}
            (G_{2})_{\alpha_{1}}; \\
            (G_{2})_{\alpha_{2}}; \\
            (G_{2})_{\alpha_{1} + \alpha_{2}}; \\
            (G_{2})_{2\alpha_{1} + \alpha_{2}}; \\
            (G_{2})_{3\alpha_{1} + \alpha_{2}};\\
            (G_{2})_{3\alpha_{1} + 2\alpha_{2}}
        \end{array} \right. \text{ are generated by }
        \left\{ \begin{array}[pos]{l}
            E_{\beta_{1}} + c_{10}E_{\beta_{3}}; \\
            E_{\beta_{2}}; \\
            E_{\beta_{1} + \beta_{2}} + c_{12} \cdot E_{\beta_{2} + \beta_{3}}; \\
            E_{\beta_{2} + 2\beta_{3}} + c_{21} \cdot E_{\beta_{1} + \beta_{2} + \beta_{3}}; \\
            E_{\beta_{1}+\beta_{2}+2\beta_{3}};\\
            E_{\beta_{1}+2\beta_{2}+2\beta_{3}},
        \end{array} \right. \text{ respectively,}
    \]
    for some nonzero constants $c_{ij} \in \CC^{\times}$. See for example \cite[Section 19.3]{Humphreys2012IntroductionLie}.

    Now consider the induced embedding between adjoint varieties $Z_{G_{2}} \hookrightarrow Z_{B_{3}}$. Since the ideal of $Z_{\gfr} \subset \PP(\gfr)$ is generated by a system of quadrics, by Corollary \ref{coro: number of closed orbits} and Table \ref{table: B stable planes}, the scheme-theoretic intersection of $Z_{G_{2}}$ and the plane
    \begin{align*}
        \Pcal &:= \PP(E_{3\alpha_{1} + 2\alpha_{2}} ,\, E_{3\alpha_{1} + \alpha_{2}} ,\, E_{2\alpha_{1} + \alpha_{2}}) \text{ in }\PP(G_{2}) \\
        &= \PP(E_{\beta_{1}+ 2\beta_{2}+2\beta_{3}} ,\, E_{\beta_{1}+\beta_{2}+2\beta_{3}} ,\, E_{\beta_{2} + 2\beta_{3}} + c_{21} \cdot E_{\beta_{1} + \beta_{2} + \beta_{3}}) \text{ in }\PP(B_{3})
    \end{align*}
    is a double line which represents a point in $\Hbf(Z_{G_{2}})$, hence a point in $\Hbf(Z_{B_{3}})$. Thus it suffices to show that $\Pcal$ is not contained in $Z_{B_{3}}$. To see this, assume that $[E_{\beta_{2} + 2\beta_{3}} + c_{21} \cdot E_{\beta_{1} + \beta_{2} + \beta_{3}}] \in Z_{B_{3}}$ in $\PP(B_{3})$.
    For each $t \in \CC^{\times}$, consider a point in $Z_{B_{3}}$ given by
    \begin{align*}
        & \exp(t \cdot E_{-\beta_{1}-\beta_{2}}) \cdot [E_{\beta_{2} + 2\beta_{3}} + c_{21} \cdot E_{\beta_{1} + \beta_{2} + \beta_{3}}] \\
        &= [Ad_{\exp(t \cdot E_{-\beta_{1} - \beta_{2}})} (E_{\beta_{2} + 2\beta_{3}} + c_{21} \cdot E_{\beta_{1} + \beta_{2} + \beta_{3}})] \\
        &= [E_{\beta_{2} + 2 \beta_{3}} + c_{21} E_{\beta_{1}+\beta_{2} +\beta_{3}} + t \cdot c_{21}N_{-\beta_{1}-\beta_{2}, \, \beta_{1} + \beta_{2} + \beta_{3}} E_{\beta_{3}}].
    \end{align*}
    By taking $t \rightarrow \infty$, we obtain the limit point $[E_{\beta_{3}}] \not\in Z_{B_{3}}$, which is a contradiction since $Z_{B_{3}}$ is compact.
\end{proof}

\section{Computation of Colored Fans} \label{section: proof of main theorems}

In this section, we prove Theorem \ref{main thm: colored data of chow} and Theorem \ref{main thm: colored data of hilb}. Recall the following proposition.

\begin{proposition}[{\cite[Proposition 15.15]{Timashev2011HomogeneousSpaces}}] \label{prop: same orbit structure when locally linearlizable}
    For a $G$-variety $X$ admitting a $G$-linearized ample line bundle, if $X$ contains an open $G$-orbit which is spherical, then its normalization map $\pi: X^{nor} \rightarrow X$ is bijective on the sets of $G$-orbits. That is, for a $G$-orbit $\Ocal$ of $X$, $\pi^{-1}(\Ocal)$ is also a single $G$-orbit.
\end{proposition}

Since our compactifications are equipped with $G$-equivariant finite morphisms
\begin{align*}
    \Cbf(Z_{\gfr}) & \hookrightarrow \text{Chow}_{1,\,2}(\PP(\gfr), \, \Ocal_{\PP(\gfr)}(1)) \\
    & \rightarrow \left\{\text{hypersurfaces of degree $(2,\,2)$ in $\PP(\gfr^{*}) \times \PP(\gfr^{*})$}\right\} \\
    &\hookrightarrow \PP\left((\text{Sym}^{2} \gfr)^{\otimes 2}\right), \\
    \Hbf(Z_{\gfr}) & \rightarrow \text{Hilb}_{2m+1}(\PP(\gfr), \, \Ocal_{\PP(\gfr)}(1)) \\
    &\hookrightarrow \text{Gr}\left(\text{Sym}^{N}(\gfr^{*}), \, M\right) \\
    &\subset \PP\left(\bigwedge^{M} \text{Sym}^{N}(\gfr^{*})\right)
\end{align*}
for some positive integers $M$ and $N$ (see \cite[Chapter I]{Kollar1996RationalCurves}), $\Cbf(Z_{\gfr})$ and $\Hbf(Z_{\gfr})$ are equipped with $G$-linearized ample line bundles. Thus by Proposition \ref{prop: same orbit structure when locally linearlizable}, we may identify orbits in $\Cbf(Z_{\gfr})$ (respectively, in $\Hbf(Z_{\gfr})$) and orbits in its normalization (though Proposition \ref{prop: same orbit structure when locally linearlizable} does not mean that the orbits are isomorphic as varieties).

\subsection{Hilbert Schemes for Exceptional Lie Algebras and Chow Schemes} \label{subsection: hilb for excep and chow}

Consider a $G$-variety $X$ with an open spherical $G$-orbit and let $Y \subset X$ be a closed $G$-orbit which is projective. If $X$ admits a $G$-linearized ample line bundle, then by Proposition \ref{prop: same orbit structure when locally linearlizable}, for the normalization $\pi : X^{nor} \rightarrow X$, $\pi^{-1}(Y)$ is a closed $G$-orbit in $X^{nor}$. Moreover, since $Y$ is simply connected and the restriction $\pi:\pi^{-1}(Y) \rightarrow Y$ is a $G$-equivariant finite morphism, we have $\pi^{-1}(Y) \simeq Y$.

Now assume that the open orbit of $X$ is isomorphic to $O_{\gfr}$, and that $Y$ is a unique closed orbit in $X$. Then $X^{nor}$ becomes a simple $O_{\gfr}$-embedding with unique closed orbit $\simeq Y$. According to Theorem \ref{thm: colored data of symmetric var} and Theorem \ref{main thm: open symmetric orbit of twistor conics with satake diagram}, in the notation of Section \ref{subsection: main theorems}, the stabilizer of $\Dcal_{i}= \epsilon^{-1} (\lambda_{i}^{\vee}/2)$ is the parabolic subgroup $P'_{I'_{i}}$ containing $B'$ where 
\[
    I'_{i} := \{\alpha'_{j} \in S' : \overline{\alpha'_{j}} = \lambda_{i}\}.
\]
If $I_{i} := I'_{i} \circ Ad_{g}$ for each $i$ and $w_{0}$ is a representative of the longest element in $W_{G}$, then since $P'_{I'_{i}} = g \cdot P_{I_{i}} \cdot g^{-1}$, Lemma \ref{lem: stabilizer of colors and isotropy group of closed orbit} implies that
\begin{equation}\label{eqn: colors of simple symmetric variety}
    \text{(the isotropy group of $Y$ containing $B$)} = \bigcap_{\Dcal_{i} \in \Dcal(O_{\gfr}) \setminus \Fcal(X^{nor})} (w_{0} \cdot P_{I_{i}} \cdot w_{0}^{-1})^{-}.
\end{equation}

\begin{proof}[Proof of Theorem \ref{main thm: colored data of chow}, and Theorem \ref{main thm: colored data of hilb} for exceptional $\gfr$]
    By Corollary \ref{coro: number of closed orbits} and the previous paragraphs, $\Cbf^{nor}(Z_{\gfr})$ is a simple $O_{\gfr}$-embedding with unique closed orbit $\simeq G/\text{Stab}_{G}(\Lcal_{B})$. Similarly, if $\gfr$ is exceptional, then $\Hbf^{nor}(Z_{\gfr})$ is also a simple $O_{\gfr}$-embedding with unique closed orbit $\simeq G / \text{Stab}_{G}(\Lcal_{B}) \cap \text{Stab}_{G}(\Pcal)$ where $\Pcal$ is the unique $B$-stable plane in $\PP(\gfr)$ (Table \ref{table: B stable planes}).
    
    Now by Lemma \ref{lem: properness criterion for spherical var}, to compute the colored cone of a given simple $O_{\gfr}$-embedding which is projective, it suffices to compute its colors, which can be done case-by-case. Indeed, the equation (\ref{eqn: colors of simple symmetric variety}) shows that
    \[
        \text{Stab}_{G}(\Lcal_{B}) = \bigcap_{\Dcal_{i} \in \Dcal(\Ocal_{\gfr}) \setminus \Fcal(\Cbf^{nor}(Z_{\gfr}))} (w_{0} \cdot P_{I_{i}} \cdot w_{0}^{-1})^{-}, \quad \text{for every $\gfr$,}
    \]
    and
    \[
        \text{Stab}_{G}(\Lcal_{B}) \cap \text{Stab}_{G}(\Pcal) = \bigcap_{\Dcal_{i} \in \Dcal(\Ocal_{\gfr}) \setminus \Fcal(\Hbf^{nor}(Z_{\gfr}))} (w_{0} \cdot P_{I_{i}} \cdot w_{0}^{-1})^{-}, \quad \text{if $\gfr$ is exceptional.}
    \]
    For example, if $\gfr = B_{r}$ with $r \ge 3$, then we have
    \[
        \text{Stab}_{G} (\Lcal_{B}) = P_{\alpha_{1}, \, \alpha_{3}}, \quad I_{i} = \{\alpha_{i}\}, \quad Ad_{w_{0}} = - id
    \]
    by Theorem \ref{thm: space of lines is homogeneous}, Table \ref{table: Satake diagram and restricted root system} and \cite[Plate II]{Bourbaki2002LieGroups}. It means that $(w_{0} \cdot P_{I_{i}} \cdot w_{0}^{-1})^{-} = P_{\alpha_{i}}$ for each $i$, hence
    \[
        \Fcal(\Cbf^{nor}(Z_{B_{r}})) = \left\{ \begin{array}[pos]{cc}
            \{\Dcal_{2}, \, \Dcal_{4}\} & \text{(if $r \ge 4$),} \\
            \{\Dcal_{2}\} & \text{(if $r =3$)}
        \end{array} \right.
    \]
    After similar computations using the following list of the isotropy groups and $(w_{0} \cdot P_{I_{i}} \cdot w_{0}^{-1})^{-}$, we obtain Table \ref{table: sph data of Chow}, and also Table \ref{table: sph data of Hilb} for exceptional Lie algebras.
    \begin{enumerate}
        \item $\gfr = D_{r}$ with $r \ge 6$:
        \[
            \text{Stab}_{G}(\Lcal_{B}) = P_{\alpha_{1}, \, \alpha_{3}}, \quad (w_{0} \cdot P_{I_{i}} \cdot w_{0}^{-1})^{-} = P_{\alpha_{i}}, \quad \forall i =1,\,2,\,3,\,4.
        \]
        \item $\gfr = D_{5}$:
        \[
            \text{Stab}_{G}(\Lcal_{B}) = P_{\alpha_{1}, \, \alpha_{3}}, \quad (w_{0} \cdot P_{I_{i}} \cdot w_{0}^{-1})^{-} =
            \left\{ \begin{array}[pos]{cc}
                P_{\alpha_{i}}, & \text{ if } i = 1,\,2,\,3; \\
                P_{\alpha_{4}, \, \alpha_{5}}, & \text{ if } i = 4.
            \end{array} \right.
        \]
        \item $\gfr = D_{4}$:
        \[
            \text{Stab}_{G}(\Lcal_{B}) = P_{\alpha_{1}, \, \alpha_{3}, \, \alpha_{4}}, \quad (w_{0} \cdot P_{I_{i}} \cdot w_{0}^{-1})^{-} = P_{\alpha_{i}}, \quad \forall i = 1,\,2,\,3,\,4.
        \]
        \item $\gfr = E_{6}$:
        \[
            \text{Stab}_{G}(\Lcal_{B}) = P_{\alpha_{3}}, \quad \text{Stab}_{G}(\Lcal_{B}) \cap \text{Stab}_{G}(\Pcal) = P_{\alpha_{2}, \, \alpha_{3}, \, \alpha_{4}},
        \]
        \[
            (w_{0} \cdot P_{I_{1}} \cdot w_{0}^{-1})^{-} = P_{\alpha_{1}, \, \alpha_{5}}, \quad (w_{0} \cdot P_{I_{2}} \cdot w_{0}^{-1})^{-} = P_{\alpha_{2}, \, \alpha_{4}},
        \]
        \[
            (w_{0} \cdot P_{I_{3}} \cdot w_{0}^{-1})^{-} = P_{\alpha_{3}}, \quad (w_{0} \cdot P_{I_{4}} \cdot w_{0}^{-1})^{-} = P_{\alpha_{6}}.
        \]
        \item $\gfr = E_{7}$:
        \[
            \text{Stab}_{G}(\Lcal_{B}) = P_{\alpha_{5}}, \quad \text{Stab}_{G}(\Lcal_{B}) \cap \text{Stab}_{G}(\Pcal) = P_{\alpha_{4}, \, \alpha_{5}},
        \]
        \[
            (w_{0} \cdot P_{I_{1}} \cdot w_{0}^{-1})^{-} = P_{\alpha_{2}}, \quad (w_{0} \cdot P_{I_{2}} \cdot w_{0}^{-1})^{-} = P_{\alpha_{4}},
        \]
        \[
            (w_{0} \cdot P_{I_{3}} \cdot w_{0}^{-1})^{-} = P_{\alpha_{5}}, \quad (w_{0} \cdot P_{I_{4}} \cdot w_{0}^{-1})^{-} = P_{\alpha_{6}}.
        \]
        \item $\gfr = E_{8}$:
        \[
            \text{Stab}_{G}(\Lcal_{B}) = P_{\alpha_{2}}, \quad \text{Stab}_{G}(\Lcal_{B}) \cap \text{Stab}_{G}(\Pcal) = P_{\alpha_{2}, \, \alpha_{3}},
        \]
        \[
            (w_{0} \cdot P_{I_{1}} \cdot w_{0}^{-1})^{-} = P_{\alpha_{7}}, \quad (w_{0} \cdot P_{I_{2}} \cdot w_{0}^{-1})^{-} = P_{\alpha_{3}},
        \]
        \[
            (w_{0} \cdot P_{I_{3}} \cdot w_{0}^{-1})^{-} = P_{\alpha_{2}}, \quad (w_{0} \cdot P_{I_{4}} \cdot w_{0}^{-1})^{-} = P_{\alpha_{1}}.
        \]
        \item $\gfr = F_{4}$:
        \[
            \text{Stab}_{G}(\Lcal_{B}) = P_{\alpha_{3}}, \quad \text{Stab}_{G}(\Lcal_{B}) \cap \text{Stab}_{G}(\Pcal) = P_{\alpha_{2}, \, \alpha_{3}},
        \]
        \[
            (w_{0} \cdot P_{I_{i}} \cdot w_{0}^{-1})^{-} = P_{\alpha_{i}}, \quad \forall i = 1,\, \ldots, \, 4.
        \]
        \item $\gfr = G_{2}$:
        \[
            \text{Stab}_{G}(\Lcal_{B}) = \text{Stab}_{G}(\Lcal_{B}) \cap \text{Stab}_{G}(\Pcal) = P_{\alpha_{1}},
        \]
        \[
            (w_{0} \cdot P_{I_{i}} \cdot w_{0}^{-1})^{-} = P_{\alpha_{i}}, \quad \forall i = 1,\,2.
        \]
    \end{enumerate}
\end{proof}

\subsection{Hilbert Schemes for Classical Lie Algebras of High Rank}

To complete the proof of Theorem \ref{main thm: colored data of hilb}, it is necessary to compute the colored data of $\Hbf^{nor}(Z_{\gfr})$ for $\gfr$ of type $B$ or $D$. In this section, we compute the colored data when $\gfr$ is $B_{r}$ ($r \ge 4$) or $D_{r}$ ($r \ge 5$), using the colored cone of $\Cbf^{nor}(Z_{\gfr})$. First, let us prove the following lemma, which holds for every $\gfr$ satisfying the assumption (\ref{assumption: pic num 1 not pn}).

\begin{lemma} \label{lem: planar reducible correspond to face of codim 1}
    In $\Cbf^{nor}(Z_{\gfr})$, each $G$-orbit represented by planar reducible conics corresponds to a colored face of codimension 1 in the colored cone of $\Cbf^{nor}(Z_{\gfr})$.
\end{lemma}
\begin{proof}
    Let $\Ocal \subset \Cbf(Z_{\gfr})$ be an orbit represented by planar reducible conics, and $\Pcal$ a plane contained in $Z_{\gfr}$ such that reducible conics in $\Pcal$ represent points in $\Ocal$. Since non-planarity and smoothness are open conditions, every boundary point of $\Ocal$ should be represented by a planar reducible conic or a double line. If a planar reducible conic $C$ represents a boundary point of $\Ocal$, then the plane spanned by $C$ is in the closure of the $G$-orbit containing $\Pcal$ in the space of planes in $Z_{\gfr}$. However since the space of planes in $Z_{\gfr}$ is the disjoint union of rational homogeneous spaces by \cite[Theorem 4.9]{LandsbergManivel2003ProjectiveGeometry}, the plane spanned by $C$ is indeed $G$-conjugate to $\Pcal$, which is a contradiction to Corollary \ref{coro: only one orbit of planar conics}. Therefore the boundary of $\Ocal$ consists of double lines, and the same statement holds for $\pi^{-1}(\Ocal)$, which is a $G$-orbit by Proposition \ref{prop: same orbit structure when locally linearlizable}, where $\pi : \Cbf^{nor}(Z_{\gfr}) \rightarrow \Cbf(Z_{\gfr})$ is the normalization map. Since double lines form the unique closed orbit in $\Cbf(Z_{\gfr})$, the desired statement follows.
\end{proof}

Now assume that $\gfr = B_{r}$ ($r \ge 4$) or $D_{r}$ ($r \ge 5$) so that the restrictive root system is $R'_{O_{\gfr}} = B_{4}$. By Corollary \ref{coro: number of closed orbits} and Table \ref{table: B stable planes}, $\Hbf^{nor}(Z_{\gfr})$ has exactly two closed orbits and they are represented by planar double lines. Thus the colored fan of $\Hbf^{nor}(Z_{\gfr})$ consists of two maximal colored cones and their colored faces by Lemma \ref{lem: number of orbits in terms of colored faces}. For each closed orbit $Y \subset \Hbf^{nor}(Z_{\gfr})$, we define a simple $O_{\gfr}$-embedding
\[
    \Hbf_{Y} := \left\{x \in \Hbf^{nor}(Z_{\gfr}) : Y \subset \overline{G \cdot x} \right\}
\]
as in Section \ref{section: Luna Vust theorey} so that its colored cone is a maximal element in the colored fan of $\Hbf^{nor}(Z_{\gfr})$.

\begin{proof}[Proof of Theorem \ref{main thm: colored data of hilb} for $B_{r}$ ($r \ge 4$) and $D_{r}$ ($r \ge 5$)]
    For $i=1,\,2$, let $Y_{i}$ be the closed orbit in $\Hbf^{nor}(Z_{\gfr})$ represented by double lines in the $i$th plane $\Pcal_{i}$ in Table \ref{table: B stable planes}. Since
    \[
        \text{Stab}_{G}(\Lcal_{B}) \cap \text{Stab}_{G}(\Pcal_{i}) = \left\{ \begin{array}[pos]{ll}
            P_{\alpha_{1}, \, \alpha_{3}} & \text{if }i=1; \\
            P_{\alpha_{1}, \, \alpha_{3}, \, \alpha_{4} , \, \alpha_{5}} & \text{if }\gfr = D_{5} \text{ and } i=2; \\
            P_{\alpha_{1}, \, \alpha_{3}, \, \alpha_{4}} & \text{otherwise},
        \end{array} \right.
    \]
    the equation (\ref{eqn: colors of simple symmetric variety}) in Section \ref{subsection: hilb for excep and chow} shows that
    \[
        \Fcal(\Hbf_{Y_{1}}) = \{\Dcal_{2}, \, \Dcal_{4}\}, \quad \Fcal(\Hbf_{Y_{2}}) = \{\Dcal_{2}\}.
    \]

    Let $\Ocal_{i} \subset \Cbf^{nor}(Z_{\gfr})$ be the $G$-orbit containing planar reducible conics in the $i$th plane in Table \ref{table: B stable planes} for $i=1,\,2$. By Lemma \ref{lem: planar reducible correspond to face of codim 1}, each $\Ocal_{i}$ corresponds to a colored face of dimension 3, and such a face contains extremal rays generated by $-\gamma_{2}$ and $-\gamma_{4}$ by the list of colored faces in Section \ref{subsection: main theorems}. Note that if a 1-dimensional colored face $\QQ_{\ge 0} \cdot (-\gamma)$ of $\Cbf^{nor}(Z_{\gfr})$ is contained in the colored face corresponding to $\Ocal_{i}$ for some $i$, then the $G$-stable divisor corresponding to $\QQ_{\ge 0} \cdot (-\gamma)$ contains $\Ocal_{i}$ in $\Cbf^{nor}(Z_{\gfr})$. Since $\Ocal_{i}$ is contained in the open subset where the morphism $FC^{nor}: \Hbf^{nor}(Z_{\gfr}) \rightarrow \Cbf^{nor}(Z_{\gfr})$ is an isomorphism, and $Y_{i}$ is contained in the closure of its preimage under $FC^{nor}$, the strict transform of the divisor corresponding to $\QQ_{\ge 0} \cdot (-\gamma)$ contains $Y_{i}$. In other words, the colored cone of $\Hbf_{Y_{i}}$ contains $\QQ_{\ge0}\cdot (-\gamma)$ as an extremal ray. Therefore the colored cone of $\Hbf_{Y_{1}}$ contains
    \[
        \Ccal_{1} := \QQ_{\ge 0}\langle -\gamma_{2}, \, -\gamma_{4}, \, \lambda_{2}^{\vee}, \, \lambda_{4}^{\vee} \rangle \quad (\ni - \gamma_{3} = -\gamma_{4} + \lambda_{4}^{\vee}/2)
    \]
    and the colored cone of $\Hbf_{Y_{2}}$ contains $\QQ_{\ge 0} \langle -\gamma_{2}, \, -\gamma_{4}, \, \lambda_{2}^{\vee} \rangle$. If $-\gamma_{1}$ is contained in the colored cone of $\Hbf_{Y_{1}}$, then it contains the valuation cone $\Vcal$, which is a contradiction since $\Hbf_{Y_{1}}$ is not complete. Hence the colored cone of $\Hbf_{Y_{2}}$ contains
    \[
        \Ccal_{2} := \QQ_{\ge 0}\langle -\gamma_{1}, \, -\gamma_{2}, \, -\gamma_{4} ,\, \lambda_{2}^{\vee}\rangle.
    \]
    By the definition of a colored fan, especially 1(a) and 4(b) in Definition \ref{defn: colored cone and fan}, $(\Ccal_{1}, \, \{\Dcal_{2}, \, \Dcal_{4}\})$ and $(\Ccal_{2}, \, \{\Dcal_{2}\})$ are colored cones of $\Hbf_{Y_1}$ and $\Hbf_{Y_{2}}$, respectively.
\end{proof}

\subsection{Hilbert Scheme for \texorpdfstring{$D_{4}$}{D4}}

In this section, we compute the colored fan of $\Hbf^{nor}(Z_{D_{4}})$, which is the case when $R'_{O_{\gfr}} = D_{4}$.

By Corollary \ref{coro: number of closed orbits} and Table \ref{table: B stable planes}, $\Hbf^{nor}(Z_{D_{4}})$ has three closed $G$-orbits, and each of them consists of planar double lines. For each $i\in \{1,\,3,\,4\}$ and the colored face
\[
    \Ccal'_{i}:=\QQ_{\ge 0} \langle -\gamma_{j} : j \in \{1,\,2,\,3,\,4\} \setminus \{i\} \rangle
\]
of the colored cone of $\Cbf^{nor}(Z_{D_{4}})$, there is a $B$-stable plane $\Pcal_{i}$ in $Z_{D_{4}}$ such that planar reducible conics in $\Pcal_{i}$ belong to the $G$-orbit in $\Cbf^{nor}(Z_{D_{4}})$ corresponding to $\Ccal'_{i}$ by Corollary \ref{coro: only one orbit of planar conics} and Lemma \ref{lem: planar reducible correspond to face of codim 1}. For each $i$, let $Y_{i}$ be the closed orbit in $\Hbf^{nor}(Z_{D_{4}})$ containing double lines in $\Pcal_{i}$. As in Section \ref{section: Luna Vust theorey}, we define
\[
    \Hbf_{Y_{i}} := \{x \in \Hbf^{nor}(Z_{D_{4}}) : Y_{i} \subset \overline{G \cdot x}\}
\]
for each $i$.

\begin{proof}[Proof of Theorem \ref{main thm: colored data of hilb} for $D_{4}$]
    Since all of $Y_{i}$ have same isotropy group $P_{\alpha_{1}, \, \alpha_{3}, \, \alpha_{4}}$, the equation (\ref{eqn: colors of simple symmetric variety}) in Section \ref{subsection: hilb for excep and chow} shows that
    \[
        \Fcal(\Hbf_{Y_{1}}) = \Fcal(\Hbf_{Y_{3}}) = \Fcal(\Hbf_{Y_{4}}) = \{ \Dcal_{2}\}.
    \]
    As in the previous section, since the strict transform (via $FC^{nor}$) of the divisor corresponding to each extremal ray of $\Ccal_{i}'$ contains $Y_{i}$, the colored cone of $\Hbf_{Y_{i}}$ contains
    \[
        \Ccal_{i} := \QQ_{\ge 0}\langle \Ccal_{i}',\, \lambda_{2}^{\vee} \rangle = \QQ_{\ge 0}\langle -\gamma_{j},\, \lambda_{2}^{\vee} : j\in \{1,\,2,\,3,\,4\} \setminus \{i\} \rangle.
    \]
    A straightforward computation using the relation $-\gamma_{1} + 2 \gamma_{2} - \gamma_{3} - \gamma_{4} = \lambda_{2}^{\vee}$ shows that for $\sum_{i=1}^{4} a_{i} \cdot (-\gamma_{i}) \in \Vcal$, $\alpha_{i} \in \QQ_{\ge 0}$, we have
    \[
        \sum_{i=1}^{4} a_{i} \cdot (-\gamma_{i}) \in \left\{ \begin{array}{ll}
            \Ccal_{1} & \text{if and only if $a_{1} \le a_{3}$ and $a_{1} \le a_{4}$;} \\
            \Ccal_{3} & \text{if and only if $a_{3} \le a_{1}$ and $a_{3} \le a_{4}$;} \\
            \Ccal_{4} & \text{if and only if $a_{4} \le a_{1}$ and $a_{4} \le a_{3}$.}
        \end{array} \right.
    \]
    By 1(a) and 4(b) in Definition \ref{defn: colored cone and fan}, the colored cone of $\Hbf_{Y_{i}}$ is $(\Ccal_{i},\,\{\Dcal_{2}\})$ for each $i$.
\end{proof}

\subsection{Hilbert Scheme for \texorpdfstring{$B_{3}$}{B3}} \label{subsection: proof for Hilb B3}

It is remained to consider $B_{3}$. First, we show that the map in Corollary \ref{coro: number of closed orbits} is also bijective for $B_{3}$.

\begin{proposition} \label{prop: hilb for B3 is not simple}
    $\Hbf(Z_{B_{3}})$ contains two closed orbits.
\end{proposition}

\begin{proof}
    Assume that $\Hbf(Z_{B_{3}})$ has a unique closed orbit. Since $Z_{B_{3}}$ contains a plane (Table \ref{table: B stable planes}), the unique closed orbit is consisting of planar double lines by Corollary \ref{coro: number of closed orbits}. Then Table \ref{table: B stable planes} shows that $P_{\alpha_{3}} \cap \text{Stab}_{G}(\Lcal_{B}) = P_{\alpha_{1}, \, \alpha_{3}}$ is an isotropy group of the closed orbit, which is equal to the stabilizer of the line $\Lcal_{B}$. Therefore by the equation (\ref{eqn: colors of simple symmetric variety}) in Section \ref{subsection: hilb for excep and chow}, the colored cones of $\Hbf^{nor}(Z_{B_{3}})$ and $\Cbf^{nor}(Z_{B_{3}})$ are same. It means that the morphism $FC^{nor} : \Hbf^{nor}(Z_{B_{3}}) \rightarrow \Cbf^{nor}(Z_{B_{3}})$ is an isomorphism. In particular, any two double lines in $Z_{B_{3}}$ representing points in $\Hbf(Z_{B_{3}})$ are $G$-conjugate to each other by Corollary \ref{coro: number of closed orbits} and Proposition \ref{prop: same orbit structure when locally linearlizable}. However, it is a contradiction to Lemma \ref{lem: nonplanar double lines exist when B3}.
\end{proof}

\begin{proof}[Proof of Theorem \ref{main thm: colored data of hilb} for $B_{3}$]
    As in Section \ref{section: Luna Vust theorey}, for each closed orbit $Y$, we define $\Hbf_{Y} := \{x \in \Hbf^{nor}(Z_{B_{3}}) : Y \subset \overline{G \cdot x}\}$. By Proposition \ref{prop: hilb for B3 is not simple}, $\Hbf^{nor}(Z_{B_{3}})$ contains two closed orbits, say $Y_{1}$ and $Y_{3}$. By the equation (\ref{eqn: colors of simple symmetric variety}) in Section \ref{subsection: hilb for excep and chow}, the colors of $\Hbf_{Y_{1}}$ and $\Hbf_{Y_{3}}$ are given by
    \[
        \Fcal(\Hbf_{Y_{1}}) = \Fcal(\Hbf_{Y_{3}}) = \{\Dcal_{2}\},
    \]
    since $P_{\alpha_{1},\, \alpha_{3}}$ is an isotropy group of each of $Y_{1}$ and $Y_{3}$.
    
    For each $i=1,\,3$, let $\Ocal_{i}$ be the $G$-orbit in $\Cbf^{nor}(Z_{B_{3}})$ corresponding to the colored face $\QQ_{\ge 0}\langle -\gamma_{i}, \, -\gamma_{2} \rangle$. If one of the closed orbits, say $Y_{j}$, is contained in the closures of the preimages of both $\Ocal_{1}$ and $\Ocal_{3}$ under $FC^{nor}$, then the colored cone of $\Hbf_{Y_{j}}$ contains $\QQ_{\ge 0}\langle -\gamma_{1}, \, -\gamma_{2}, \, -\gamma_{3}, \, \lambda_{2}^{\vee} \rangle$, which is a contradiction since $\Hbf_{Y_{j}}$ is not complete. Therefore we may assume that for each $i=1,\,3$, $Y_{i}$ is contained in the closure of the preimage of $\Ocal_{i}$. In other words, the colored cone of $\Hbf_{Y_{i}}$ contains
    \[
        \Ccal_{i} :=\QQ_{\ge 0}\langle -\gamma_{i}, \, -\gamma_{2}, \, \lambda_{2}^{\vee} \rangle.
    \]
    By 1(a) and 4(b) in Definition \ref{defn: colored cone and fan}, $(\Ccal_{1}, \, \{\Dcal_{2}\})$ and $(\Ccal_{2}, \, \{\Dcal_{2}\})$ are maximal colored cones in the colored fan of $\Hbf^{nor}(Z_{B_{3}})$.
\end{proof}

\section{Applications} \label{section: structures of spaces of conics}

The section is devoted to corollaries of main theorems. First, by Theorem \ref{main thm: colored data of chow} and Theorem \ref{main thm: colored data of hilb}, the colored fans of $\Hbf^{nor}(Z_{\gfr})$ and $\Cbf^{nor}(Z_{\gfr})$ coincide if and only if $\gfr= G_{2}$, hence the following corollary is immediate.

\begin{coro} \label{coro: FC is isom if G2}
    If $\gfr = G_{2}$, then the morphism $FC^{nor} : \Hbf^{nor}(Z_{G_{2}}) \rightarrow \Cbf^{nor}(Z_{G_{2}})$ is an isomorphism. Otherwise there is no $G$-equivariant isomorphism between $\Hbf^{nor}(Z_{\gfr})$ and $\Cbf^{nor}(Z_{\gfr})$.
\end{coro}

\subsection{Orbit Structures and Conjugacy Classes of Conics} \label{subsection: orbit structure of Hilb and Chow}

In this section, we describe the $G$-conjugacy classes of conics in adjoint varieties by studying the orbit structures of the spaces of conics. First of all, we can count the number of $G$-orbits in $\Cbf(Z_{\gfr})$ and $\Hbf(Z_{\gfr})$, using Lemma \ref{lem: number of orbits in terms of colored faces}, Proposition \ref{prop: same orbit structure when locally linearlizable}, Theorem \ref{main thm: colored data of chow} and Theorem \ref{main thm: colored data of hilb}.

\begin{coro} \label{coro: number of orbits in Chow and Hilb}
    The number of $G$-orbits in $\Cbf(Z_{\gfr})$ is
    \[
        \left\{ \begin{array}[pos]{cl}
            11 & \text{if $\gfr = B_{r}$ with $r \ge 4$, $D_{r}$ with $r \ge 5$;}\\
            7 & \text{if $\gfr = B_{3}, \, E_{6}, \, E_{7},\,E_{8}, \, F_{4}$;} \\
            15 & \text{if $\gfr = D_{4}$;} \\
            3 & \text{if $\gfr = G_{2}$.}
        \end{array} \right.
    \]
    The number of $G$-orbits in $\Hbf(Z_{\gfr})$ is
    \[
        \left\{ \begin{array}[pos]{cl}
            15 & \text{if $\gfr = B_{r}$ with $r \ge 4$, $D_{r}$ with $r \ge 5$;}\\
            9 & \text{if $\gfr = B_{3}, \, E_{6},\,E_{7}, \, E_{8}, \, F_{4}$;} \\
            21 & \text{if $\gfr = D_{4}$;} \\
            3 & \text{if $\gfr = G_{2}$.}
        \end{array} \right.
    \]
\end{coro}

To figure out which conjugacy class of conics is associated to each colored face, we need to analyze geometry of singular conics in more detail. Recall that every reducible conic is a member of $\Hbf(Z_{\gfr})$ and $\Cbf(Z_{\gfr})$. The following proposition shows that reducible conics form an irreducible divisor in the spaces of conics.

\begin{proposition} \label{prop: nonplanar reducible conic exists}
    Both $\Hbf(Z_{\gfr})$ and $\Cbf(Z_{\gfr})$ contain prime divisors which parametrize all reducible conics and whose general points are represented by non-planar reducible conics.
\end{proposition}
\begin{proof}
    It is enough to consider the Hilbert scheme. Let $\Kcal_{o} \subset \PP(T_{o}Z_{\gfr})$ be the subset consisting of directions of lines passing through $o$, which is irreducible by Theorem \ref{thm: space of lines is homogeneous}. Note that a reducible conic $C$ which is singular at $o$ corresponds to a set of two different points in $\Kcal_{o}$, say $\{p_{1}, \, p_{2}\}$. Moreover, $C$ is planar if and only if the secant line joining $p_{1}$ and $p_{2}$ is contained in $\Kcal_{o}$ (cf. \cite[Section 4.3]{LandsbergManivel2003ProjectiveGeometry}). If every reducible conic is planar in $Z_{\gfr}$, then the secant variety of $\Kcal_{o}$ coincides with $\Kcal_{o}$ itself, which is a contradiction (see \cite[Section 6.3]{HwangKebekus2005GeometryChains}).

    It means that for the diagonal $\Delta_{o} \subset \Kcal_{o} \times \Kcal_{o}$, its complement $(\Kcal_{o} \times \Kcal_{o}) \setminus \Delta_{o}$ parametrizes an irreducible flat family of all reducible conics singular at $o$ such that its general points represent non-planar reducible conics. Thus the subset $\Kcal^{(2)}_{o} \subset \Hbf(Z_{\gfr})$ consisting of reducible conics singular at $o$ is irreducible, and its general elements are represented by non-planar ones. We claim that the $G$-orbit of $\Kcal^{(2)}_{o}$ is of dimension $4n-1$. To show the claim, note that if $\Kcal^{(2)}_{x} \subset \Hbf(Z_{\gfr})$ denotes the subset of reducible conics singular at $x \in Z_{\gfr}$, then $G \cdot \Kcal^{(2)}_{o} = \coprod_{x \in Z_{\gfr}} \Kcal^{(2)}_{x}$, hence it suffices to show that $\dim \Kcal^{(2)}_{o} = 2n-2$. Since the natural surjective map $(\Kcal_{o} \times \Kcal_{o}) \setminus \Delta_{o} \rightarrow \Kcal^{(2)}_{o}$ is two-to-one and $\dim \Kcal_{o} = n-1$ (for example, by the splitting type of normal bundles of lines \cite{Hwang1997RigidityHomogeneous}, \cite{Kebekus2001LinesContact} or Theorem \ref{thm: space of lines is homogeneous}), the proof is done.
\end{proof}

Next, we count the conjugacy classes of reducible conics. To do this, define $P^{ss}$ to be the semi-simple part of the isotropy group $P$ at $o \in Z_{\gfr}$. Then the Dynkin diagram of $P^{ss}$ can be obtained by removing $\alpha_{j_{0}}$ in the Dynkin diagram of $G$, and $P^{ss}$ acts transitively on the space of lines passing through $o$ (Theorem \ref{thm: space of lines is homogeneous}). With respect to this action, let $Q \subset P^{ss}$ be an isotropy group. We may choose $Q$ as the parabolic subgroup of $P^{ss}$ generated by the complement of $N(\alpha_{j_{0}})$ by \cite[Theorem 4.8]{LandsbergManivel2003ProjectiveGeometry}. Also choose a line $l_{0}$ passing through $o$ such that $\text{Stab}_{P^{ss}}(l_{0}) = Q$.

\begin{lemma} \label{lem: number of conjugacy classes of reducible conics}
    \begin{enumerate}
        \item The number of $G$-conjugacy classes of reducible conics in $Z_{\gfr}$ is equal to
        \[
            |W_{P^{ss}, \, Q} \backslash W_{P^{ss}} / W_{P^{ss}, \, Q}| - 1,
        \]
        i.e. the number of double cosets $W_{P^{ss},\, Q} \cdot w \cdot W_{P^{ss}, \, Q}$ ($w \in W_{P^{ss}}$) minus 1. For each $\gfr$, the number of double cosets is given in Table \ref{table: diagram of parabolic Q in P and explicit upper bound}.
        \item The $G$-stable prime divisor of $\Cbf^{nor}(Z_{\gfr})$ given in Proposition \ref{prop: nonplanar reducible conic exists} corresponds to the following colored face:
        \[
            \left\{\begin{array}[pos]{ll}
                \QQ_{\ge 0} \cdot ( -\gamma_{2} ) & \text{if $\gfr$ is $B_{r}$ $(r \ge 3)$, $D_{r}$ $(r \ge 4)$ or $G_{2}$}; \\
                \QQ_{\ge 0} \cdot ( -\gamma_{4} ) & \text{if $\gfr$ is $E_{r}$ $(r =6,\, 7, \, 8)$ or $F_{4}$}.
            \end{array}\right.
        \]
    \end{enumerate}
    
\end{lemma}

\begin{table}
    \begin{center}
        \begin{tabular}{|c|c|c|}
            \hline
            $\gfr$ & Diagram of $Q \subset P^{ss}$ & $|W_{P^{ss}, \, Q} \backslash W_{P^{ss}} / W_{P^{ss}, \, Q}|$ \\
            \hline
            \hline
            \multirow{2}{*}{$B_{r}$ ($r \ge 4$)} & \multirow{2}{*}{{\dynkin[mark=x]A1} {\dynkin[parabolic=1, labels={3,,,,r}]B{}}} & \multirow{2}{*}{6} \\
            && \\
            \hline
            $B_{3}$ & {\dynkin[mark=x]A1} {\dynkin[mark=x]A1} & 4 \\
            \hline
            $D_{r}$ ($r \ge 6$) & {\dynkin[mark=x]A1} {\dynkin[parabolic=1, labels={3,,,,r-1,r}]D{}} & 6\\
            \hline
            $D_{5}$ & {\dynkin[mark=x]A1} {\dynkin A{*x*}} & 6 \\
            \hline
            $D_{4}$ & {\dynkin[mark=x]A1} {\dynkin[mark=x]A1} {\dynkin[mark=x]A1} & 8 \\
            \hline
            $E_{6}$ & {\dynkin A{**x**}} & 4 \\
            \hline
            $E_{7}$ & {\dynkin D{*****x}} & 4 \\
            \hline
            $E_{8}$ & {\dynkin[backwards, upside down] E{******x}} & 4 \\
            \hline
            $F_{4}$ & {\dynkin C{**x}} & 4 \\
            \hline
            $G_{2}$ & {\dynkin A{x}} & 2 \\
            \hline
        \end{tabular}
        \caption{\label{table: diagram of parabolic Q in P and explicit upper bound}Number of Double Cosets of $W_{P^{ss},\, Q}$ in $W_{P^{ss}}$ for the Parabolic $Q \subset P^{ss}$.}
    \end{center}
\end{table}

\begin{proof}
    First of all, the number of double cosets can be easily computed from the diagram of the parabolic subgroup $Q$ in $P^{ss}$, for instance by using the description of Weyl groups (\cite[Plate I-IX]{Bourbaki2002LieGroups}) and recipies for the Hasse diagrams (\cite[Chapter 4]{BastonEastwood2016PenroseTransform}).
    
    We claim that the number of double cosets minus 1 is an upper bound of the number of conjugacy classes of reducible conics. Note that each $G$-conjugacy class of reducible conics in $Z_{\gfr}$ has a representative which is singular at $o$. Moreover, if two reducible conics singular at $o$ are $G$-conjugate, then they must be $P$-conjugate since $o$ is their unique singular point. In other words, there is a bijection between
    \[
        \{\text{$G$-conjugacy classes of reducible conics in $Z_{\gfr}$}\}
    \]
    and
    \[
        \{\text{$P$-conjugacy classes of reducible conics in $Z_{\gfr}$ singular at $o$}\}.
    \]
    Each $P$-conjugacy class of reducible conics singular at $o$ has a representative of form $l_{0} \cup l$ for some line $l$ such that $l_{0} \cap l = \{o\}$. Since $\text{Stab}_{P^{ss}}(l_{0}) = Q$ and $P^{ss}$ acts on the space of lines containing $o$ transitively, the number of $P$-conjugacy classes of such $l_{0} \cup l$ is at most the number of $Q$-orbits in $P^{ss}/Q - \{e \cdot Q\}$, which is equal to $|W_{P^{ss}, \, Q} \backslash W_{P^{ss}} / W_{P^{ss}, \, Q}| - 1$ by the generalized Bruhat decomposition (\cite[Chapter IV.2.5]{Bourbaki2002LieGroups}).

    Now let us show the equalities. Let $D \subset \Cbf^{nor}(Z_{\gfr})$ be the inverse image of the $G$-stable prime divisor given in Proposition \ref{prop: nonplanar reducible conic exists}. Then $D$ corresponds to a 1-dimensional colored face of the colored cone of $\Cbf^{nor}(Z_{\gfr})$. By its definition, the ray corresponding to $D$ is contained in the colored faces corresponding to orbits defined by reducible conics. Moreover, if $\gfr \not= B_{3}, \, G_{2}$, then the number of conjugacy classes of planes contained in $Z_{\gfr}$ (Table \ref{table: B stable planes}) and the number of colored faces of codimension 1 are same. Thus by Lemma \ref{lem: planar reducible correspond to face of codim 1}, if $\gfr \not= B_{3}, \, G_{2}$, then the ray corresponding to $D$ is contained in the intersection of all colored faces of codimension 1 in the colored cone of $\Cbf^{nor}(Z_{\gfr})$.
    
    If $\gfr = G_{2}$, then $\Cbf^{nor}(Z_{\gfr})$ has only one colored extremal ray, and the number of colored faces containing it is 2. Thus the statement follows from Lemma \ref{lem: number of orbits in terms of colored faces}.
    
    If $\gfr$ is exceptional but $\not= G_{2}$, then $\Cbf^{nor}(Z_{\gfr})$ has two colored extremal ray $\QQ_{\ge 0}\cdot (-\gamma_{1})$ and $\QQ_{\ge 0}\cdot (-\gamma_{4})$, and the numbers of colored faces containing them are 5 and 4, respectively. By Lemma \ref{lem: number of orbits in terms of colored faces}, the number of colored faces containing the ray determined by $D$ is at most 4. Therefore $D$ corresponds to the ray $\QQ_{\ge 0} \cdot (-\gamma_{4})$ and the statement follows.

    If $\gfr = D_{4}$, the intersection of all colored faces of codimension 1 is $\QQ_{\ge 0} \cdot (-\gamma_{2})$. Thus $D$ corresponds to $\QQ_{\ge 0}\cdot (-\gamma_{2})$ and the number of colored faces containing it is 8.

    If $\gfr$ is $B_{r}$ $(r \ge 4)$ or $D_{r}$ $(r \ge 5)$, then the intersection of colored faces of codimension 1 is $\QQ_{\ge 0}\langle -\gamma_{2}, \, -\gamma_{4} \rangle$. Thus $D$ corresponds to either $\QQ_{\ge 0}\langle -\gamma_{2} \rangle$ or $\QQ_{\ge 0}\langle -\gamma_{4} \rangle$, and each of them is contained in 6 and 7 number of colored faces, respectively. By Table \ref{table: diagram of parabolic Q in P and explicit upper bound}, $D$ corresponds to $\QQ_{\ge 0}\langle -\gamma_{2} \rangle$ and the upper bound is attained.

    From now on, assume that $\gfr = B_{3}$, and recall the list of the colored faces given in Section \ref{subsection: main theorems}. Then the space of planes in $Z_{B_{3}}$ is homogeneous, and the unique $B$-stable plane in $Z_{B_{3}}$, say $\Pcal$, is given in Table \ref{table: B stable planes}. By Corollary \ref{coro: only one orbit of planar conics}, planar contact conics and planar reducible conic form single orbits $\Ocal_{PC}$ and $\Ocal_{PR}$ in $\Cbf(Z_{B_{3}})$, respectively. Note that the stabilizer of a given planar contact conic in $\Pcal$ is contained in $\text{Stab}_{G}(\Pcal)$. Since the space of smooth conics in $\Pcal \simeq \PP^{2}$ is 5-dimensional, and since the map $\text{Stab}_{G}(\Pcal) \rightarrow \text{Aut}(\Pcal)$ is surjective by the proof of Corollary \ref{coro: only one orbit of planar conics}, an isotropy group of $\Ocal_{PC}$ in $G$ is of codimension 5 in $\text{Stab}_{G}(\Pcal)$. Since $\text{Stab}_{G}(\Pcal) = P_{\alpha_{3}}$, this implies that
    \[
        \dim \Ocal_{PC} = \dim G / P_{\alpha_{3}} + 5 = 11 = \dim O_{B_{3}} - 1.
    \]
    That is, $\overline{\Ocal_{PC}}$ is a $G$-stable divisor in $\Cbf(Z_{B_{3}})$. Hence its inverse image $D_{PC}$ in $\Cbf^{nor}(Z_{B_{3}})$ corresponds to a 1-dimensional colored face of the colored cone. Observe that the colored cone of $\Cbf^{nor}(Z_{B_{3}})$ has three 1-dimensional colored faces. Thus there is a 1-dimensional colored face not corresponding to $D$ nor $D_{PC}$ and it corresponds to a $G$-stable divisor defined by non-planar contact conics.

    Suppose that the number of conjugacy classes of reducible conics in $Z_{B_{3}}$ is strictly less than $3$, which is the number of double cosets minus 1. Then $D$ corresponds to a ray $\QQ_{\ge 0}\cdot (-\gamma_{i})$ for some $i \in \{1,\,3\}$, and the number of conjugacy classes of reducible conics is equal to 2. It means that the orbit $\Ocal_{NPR}$ corresponding to $D$ is indeed a unique orbit of non-planar reducible conics, hence the number of conjugacy classes of non-planar contact conics is
    \[
        7 - 1 \text{(twistor)} - 1 \text{(non-planar reducible)} - 3 \text{(planar contact/reducible or double line)} = 2
    \]
    by Corollary \ref{coro: number of orbits in Chow and Hilb}. Thus one of the 2-dimensional colored faces corresponds to a conjugacy class $\Ocal$ of non-planar contact conics. However, since any 2-dimensional colored faces contain one of rays corresponding to $D$ or $D_{PC}$, $\Ocal$ is contained in the boundary of $\Ocal_{NRC}$ or the boundary of $\Ocal_{PC}$. This is a contradiction since both smoothness and non-planarity are open conditions. Therefore the number of conjugacy classes of reducible conics is $3$, and $D$ corresponds to $\QQ_{\ge 0} \cdot (-\gamma_{2})$.
\end{proof}

Lemma \ref{lem: number of conjugacy classes of reducible conics} shows that the number of $G$-orbits in the divisor of $\Cbf(Z_{\gfr})$ given in Proposition \ref{prop: nonplanar reducible conic exists} is $|W_{P^{ss}, \, Q} \backslash W_{P^{ss}} / W_{P^{ss}, \, Q}|$.

\begin{theorem} \label{thm: full orbit structure}
    \begin{enumerate}
        \item Let $\gfr$ be $B_{r}$ with $r \ge 4$ or $D_{r}$ with $r \ge 5$.
        \begin{enumerate}
            \item $\Cbf(Z_{\gfr})$ consists of eleven orbits: one for twistor conics, two for non-planar contact conics, two for planar contact conics, three for non-planar reducible conics, two for planar reducible conics, and one for double lines.
            \item $\Hbf(Z_{\gfr})$ consists of fifteen orbits: one for twistor conics, two for non-planar contact conics, two for planar contact conics, three for non-planar reducible conics, two for planar reducible conics, three for non-planar double lines, and two for planar double lines.
        \end{enumerate}
        \item Let $\gfr$ be either $B_{3}$ or of an exceptional type other than $G_{2}$.
        \begin{enumerate}
            \item $\Cbf(Z_{\gfr})$ consists of seven orbits: one for twistor conics, one for non-planar contact conics, one for planar contact conics, two for non-planar reducible conics, one for planar reducible conics, and one for double lines.
            \item $\Hbf(Z_{\gfr})$ consists of nine orbits: one for twistor conics, one for non-planar contact conics, one for planar contact conics, two for non-planar reducible conics, one for planar reducible conics, two for non-planar double lines, and one for planar double lines.
        \end{enumerate}
        \item Let $\gfr = D_{4}$.
        \begin{enumerate}
            \item $\Cbf(Z_{D_{4}})$ consists of fifteen orbits: one for twistor conics, three for non-planar contact conics, three for planar contact conics, four for non-planar reducible conics, three for planar reducible conics, and one for double lines.
            \item $\Hbf(Z_{D_{4}})$ consists of twenty one orbits: one for twistor conics, three for non-planar contact conics, three for planar contact conics, four for non-planar reducible conics, three for planar reducible conics, four for non-planar double lines, and three for planar double lines.
        \end{enumerate}
        \item If $\gfr = G_{2}$, then both $\Cbf(Z_{G_{2}})$ and $\Hbf(Z_{G_{2}})$ consist of three orbits: one for twistor conics, one for non-planar reducible conics, and one for non-planar double lines. In particular, every smooth conic in $Z_{G_{2}}$ is a twistor conic.
    \end{enumerate}
\end{theorem}

\begin{proof}
    The statements for the Chow schemes follow from Lemma \ref{lem: normal bundle single orbit of twistor conic}, Corollary \ref{coro: only one orbit of planar conics}, Corollary \ref{coro: number of closed orbits}, Lemma \ref{lem: number of conjugacy classes of reducible conics} and Corollary \ref{coro: number of orbits in Chow and Hilb}.
    
    For $\Hbf(Z_{\gfr})$, observe that the morphism $FC : \Hbf(Z_{\gfr}) \rightarrow \Cbf(Z_{\gfr})$ gives a bijective correspondence for orbits of smooth conics and reducible conics. Therefore the difference of the numbers of orbits in Corollary \ref{coro: number of orbits in Chow and Hilb} is exactly the number of orbits of double lines minus 1. Therefore the statements follow from Corollary \ref{coro: only one orbit of planar conics}.
\end{proof}

Let us visualize the orbit structure of $\Hbf(Z_{\gfr})$ for each $\gfr$. In the following graphs, each vertex represents a $G$-orbit in $\Hbf(Z_{\gfr})$ (or a conjugacy class of conics in $Z_{\gfr}$), and each edge \begin{tikzcd}[row sep=tiny, column sep=tiny]
    A \arrow[dash, d] \\
    B
\end{tikzcd} means that $B \subset \overline{A}$ in $\Hbf(Z_{\gfr})$. We also use the abbreviations (N)PC, (N)PR, and (N)PD for (non-)planar contact, (non-)planar reducible, and (non-)planar double, respectively.
\begin{enumerate}
    \item $\gfr$ is $B_{r}$ ($r \ge 4$) or $D_{r}$ ($r \ge 5$):
    
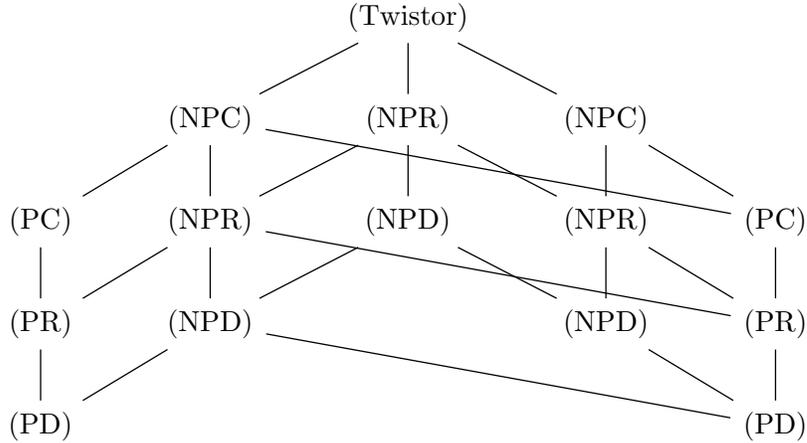
\begin{center}
    \begin{tikzcd}
        && \text{(Twistor)} \arrow[dash, ld] \arrow[dash, d] \arrow[dash, rd] && \\
        & \text{(NPC)} \arrow[dash, ld] \arrow[dash, d] \arrow[dash, rrrd] & \text{(NPR)} \arrow[dash, ld] \arrow[dash, d] \arrow[dash, rd] & \text{(NPC)} \arrow[dash, d] \arrow[dash, rd] & \\
        \text{(PC)} \arrow[dash, d] & \text{(NPR)} \arrow[dash, ld] \arrow[dash, d] \arrow[dash, rrrd] & \text{(NPD)} \arrow[dash, ld] \arrow[dash, rd] & \text{(NPR)} \arrow[dash, d] \arrow[dash, rd] & \text{(PC)} \arrow[dash, d] \\
        \text{(PR)} \arrow[dash, d] & \text{(NPD)} \arrow[dash, ld] \arrow[dash, rrrd] && \text{(NPD)} \arrow[dash, rd] & \text{(PR)} \arrow[dash, d] \\
        \text{(PD)} &&&& \text{(PD)}
    \end{tikzcd}
    \end{center}
    \item $\gfr=B_{3}$:
    
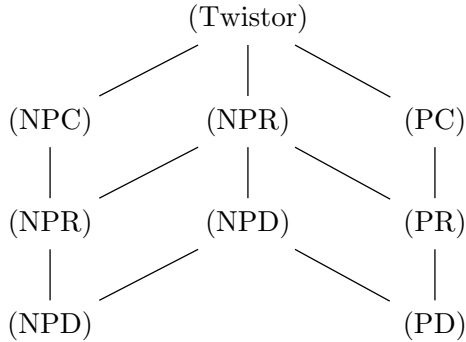
\begin{center}
        \begin{tikzcd}
            & \text{(Twistor)} \arrow[dash, ld] \arrow[dash, d] \arrow[dash, rd] & \\
            \text{(NPC)}  \arrow[dash, d] & \text{(NPR)} \arrow[dash, ld] \arrow[dash, d] \arrow[dash, rd] & \text{(PC)} \arrow[dash, d] \\
            \text{(NPR)}  \arrow[dash, d] & \text{(NPD)} \arrow[dash, ld] \arrow[dash, rd] & \text{(PR)} \arrow[dash, d] \\
            \text{(NPD)} & & \text{(PD)}
        \end{tikzcd}
    \end{center}
    \item $\gfr=D_{4}$:
    
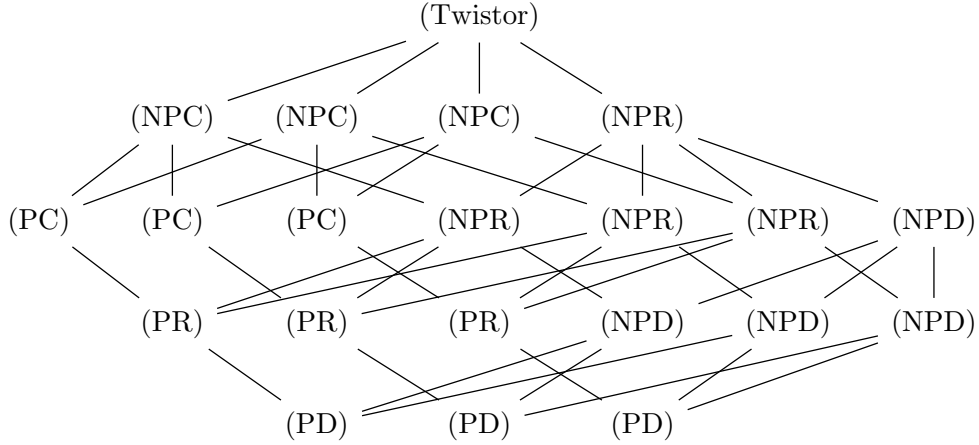
\begin{center}
    \begin{tikzcd}[column sep = small]
        &&& \text{(Twistor)} \arrow[dash, lld] \arrow[dash, ld] \arrow[dash, d] \arrow[dash, rd] &&& \\
        & \text{(NPC)} \arrow[dash, ld] \arrow[dash, d] \arrow[dash, rrd] & \text{(NPC)} \arrow[dash, lld] \arrow[dash, d] \arrow[dash, rrd] & \text{(NPC)} \arrow[dash, lld] \arrow[dash, ld] \arrow[dash, rrd] & \text{(NPR)} \arrow[dash, ld] \arrow[dash, d] \arrow[dash, rd] \arrow[dash, rrd] && \\
        \text{(PC)} \arrow[dash, rd] & \text{(PC)} \arrow[dash, rd] & \text{(PC)} \arrow[dash, rd] & \text{(NPR)} \arrow[dash, lld] \arrow[dash, ld] \arrow[dash, rd] & \text{(NPR)} \arrow[dash, llld] \arrow[dash, ld] \arrow[dash, rd] & \text{(NPR)} \arrow[dash, llld] \arrow[dash, lld] \arrow[dash, rd] & \text{(NPD)} \arrow[dash, lld] \arrow[dash, ld] \arrow[dash, d] \\
        &\text{(PR)} \arrow[dash, rd] & \text{(PR)} \arrow[dash, rd] & \text{(PR)} \arrow[dash, rd] & \text{(NPD)} \arrow[dash, lld] \arrow[dash, ld] & \text{(NPD)} \arrow[dash, llld] \arrow[dash, ld] & \text{(NPD)} \arrow[dash, llld] \arrow[dash, lld]\\
        && \text{(PD)} & \text{(PD)} & \text{(PD)} & &
    \end{tikzcd}
    \end{center}
    \item $\gfr$ is $E_{r}$ ($r = 6,\,7,\,8$) or $F_{4}$:
    
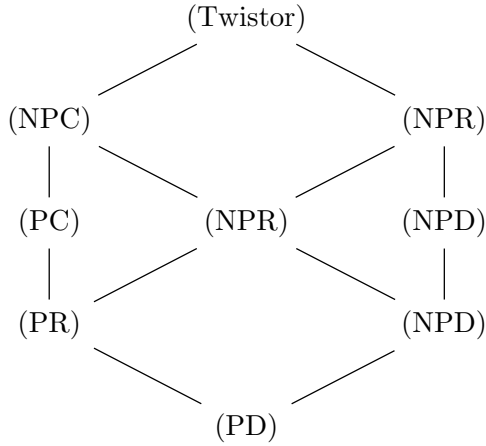
\begin{center}
        \begin{tikzcd}
            & \text{(Twistor)} \arrow[dash, ld] \arrow[dash, rd] & \\
            \text{(NPC)}  \arrow[dash, d] \arrow[dash, rd] & & \text{(NPR)} \arrow[dash, d] \arrow[dash, ld] \\
            \text{(PC)}  \arrow[dash, d] & \text{(NPR)} \arrow[dash, ld] \arrow[dash, rd] & \text{(NPD)} \arrow[dash, d] \\
            \text{(PR)} \arrow[dash, rd] & & \text{(NPD)} \arrow[dash, ld] \\
            & \text{(PD)} &
        \end{tikzcd}
    \end{center}
    \item $\gfr = G_{2}$:
    
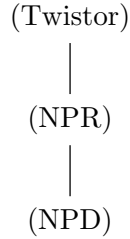
\begin{center}
        \begin{tikzcd}
            \text{(Twistor)} \arrow[dash, d] \\
            \text{(NPR)} \arrow[dash, d] \\
            \text{(NPD)}
        \end{tikzcd}
    \end{center}
\end{enumerate}

In particular, the orbit structures of $\Hbf(Z_{\gfr})$ for $B_{3}$ and for exceptional Lie algebras $\not= G_{2}$ are different, although the numbers of conjugacy classes are same (Theorem \ref{thm: full orbit structure}).

\subsection{Smoothness and Singularities}

Let us prove smoothness of $\Hbf^{nor}(Z_{\gfr})$ and determine the singular locus of $\Cbf^{nor}(Z_{\gfr})$.

\begin{coro} \label{coro: Hilb is smooth}
    Let $\gfr$ be a complex simple Lie algebra satisfying the assumption (\ref{assumption: pic num 1 not pn}).
    \begin{enumerate}
        \item $\Hbf^{nor}(Z_{\gfr})$ is smooth. Moreover, the anticanonical line bundle of $\Hbf^{nor}(Z_{\gfr})$ is
        \begin{itemize}
            \item not globally generated if $\gfr = B_{r}$ ($r \not = 5$ and $r \ge 3$), $D_{r}$ ($r \not= 6$ and $r \ge 4$), $E_{8}$;
            \item globally generated but not ample if $\gfr = B_{5}$, $D_{6}$, $E_{7}$, $F_{4}$;
            \item ample if $\gfr = E_{6}$, $G_{2}$.
        \end{itemize}
        \item $\Cbf^{nor}(Z_{\gfr})$ is
        \begin{itemize}
            \item not $\QQ$-Gorenstein if $\gfr = B_{r}$ ($r \not= 5$ and $r \ge 3$), $D_{r}$ ($r \ge 4$), $E_{6}$, $E_{7}$, $E_{8}$;
            \item Gorenstein Fano with terminal singularities but not $\QQ$-factorial if $\gfr = B_{5}$, $F_{4}$;
            \item smooth Fano if $\gfr = G_{2}$.
        \end{itemize}
        In particular, the singular locus of $\Cbf^{nor}(Z_{\gfr})$ is equal to the subset formed by double lines if $\gfr \not= G_{2}$.
    \end{enumerate}
\end{coro}

In fact, the smoothness of $\Hbf^{nor}(Z_{\gfr})$ follows from \cite[Proposition 3.6]{ChungHongKiem2012CompactifiedModuli}. In the following paragraphs, we present another proof using spherical geometry.

\begin{proof}
    Let us consider the smoothness of the Hilbert schemes first. In the case when $\gfr = G_{2}$, the smoothness of $\Hbf^{nor}(Z_{G_{2}})$ is well-known. Indeed, the colored cone of $\Hbf^{nor}(Z_{G_{2}})$ appears in \cite[Theorem 4.1]{Ruzzi2011SmoothProjective} and \cite[Theorem 2]{Ruzzi2010GeometricalDescription}, hence they are isomorphic to the \emph{Cayley grassmannian}, as observed in \cite[Section 7]{Manivel2021DoubleCayley}. For $\gfr$ different from $G_{2}$, we apply Ruzzi's smoothness criterion for symmetric varieties which can be found in \cite[Theorem 3.2]{Ruzzi2011SmoothProjective}. Observe that the notation of \cite{Ruzzi2011SmoothProjective} is slightly different from ours, since we use a different definition for the restricted root system. In our setting (Section \ref{section: Luna Vust theorey}, Section \ref{subsection: main theorems}), the criterion can be formulated as follows.
    \begin{theorem}[{\cite[Theorem 3.2]{Ruzzi2011SmoothProjective}}] \label{thm: smoothness criterion}
        Let $X$ be a simple $O_{\gfr}(=G/G^{\sigma})$-embedding, and assume that its unique closed orbit $Y$ is projective. For the standard Levi factor $L$ of the parabolic subgroup $\bigcap_{\Dcal \in \Dcal(O_{\gfr}) \setminus \Fcal(X)} \text{Stab}_{G}(\Dcal)$, let $R_{L,\, \sigma}$ be the sub-root system of $R'_{O_{\gfr}}$ spanned by the roots of $L$. The simple factors of $R_{L, \, \sigma}$ are denoted by $R^{j}_{L,\, \sigma}$ so that $R_{L, \, \sigma} = \prod_{j=1}^{p} R^{j}_{L, \, \sigma}$ for some integer $p$, and their simple roots $\{\lambda^{j}_{i}\}_{i} := S'_{O_{\gfr}} \cap R^{j}_{L,\, \sigma}$ are indexed as in \cite{OnishchikVinberg1990LieGroups}.
        
        Then $X$ is smooth if and only if the following conditions are satisfied:
        \begin{enumerate}
            \item For every $j$, $R^{j}_{L, \, \sigma}$ is of type $A$. Moreover, $\sum_{j=1}^{p} (l_{j}+1)$ is at most the rank of $R'_{O_{\gfr}}$ where $l_{j}$ is the rank of $R^{j}_{L,\, \sigma}$;
            \item The cone $\Ccal(X)$ is spanned by a basis $\Bcal$ of $\frac{1}{2} \cdot \ZZ \langle (R'_{O_{\gfr}})^{\vee} \rangle$, i.e. the half of the coroot lattice of the root system $R'_{O_{\gfr}}$;
            \item In the doubled weight lattice $2 \cdot (\ZZ \langle (R'_{O_{\gfr}})^{\vee} \rangle)^{*}$, we can index the dual basis of $\Bcal$ as $\{y^{1}_{1}, \, \ldots,\, y^{1}_{l_{1}+1}, \, y^{2}_{1}, \, \ldots, \, y^{q}_{l_{q}+1}\}$ for some integers $q (\ge p)$ and $l_{j}$ so that
            \begin{enumerate}
                \item $\langle y^{j}_{i}, \, (2 \lambda^{h}_{k})^{\vee} \rangle$ is $=1$ if $j=h$ and $i=k$, and $=0$ otherwise;
                \item $y^{j}_{i} - \frac{i}{l_{j}+1} y^{j}_{l_{j+1}}$ is the $i$th fundamental weight of $R^{j}_{L,\, \sigma}$ times $2$ for $1 \le j \le p$ and $1 \le i \le l_{j}$.
            \end{enumerate}
        \end{enumerate}
    \end{theorem}

    We claim that the simple spherical varieties defined by maximal colored cones in Table \ref{table: sph data of Hilb} satisfy the conditions in Theorem \ref{thm: smoothness criterion}. Indeed, in each case, if $\pi_{i}$ denotes the $i$th fundamental weight of $R'_{O_{\gfr}}$, then $R_{L, \, \sigma}$, $\Bcal$ and its dual basis $\{y^{j}_{i}\}$ can be chosen as follows.
    \begin{enumerate}
        \item $\gfr = B_{r}$ ($r \ge 4$) or $D_{r}$ ($r \ge 5$): For ($\QQ_{\ge 0}\langle -\gamma_{2}, \, -\gamma_{4}, \, \lambda_{2}^{\vee}, \, \lambda_{4}^{\vee} \rangle$, $\{\Dcal_{2}, \, \Dcal_{4}\}$),
            \[
                R_{L,\, \sigma} = {\dynkin[labels={\lambda_{2}}] A1} {\dynkin[labels={\lambda_{4}}] A1} =A_{1} \times A_{1},
            \]
            \[
                \Bcal = \left\{-\frac{1}{2} \lambda_{1}^{\vee} - \lambda_{2}^{\vee} - \lambda_{3}^{\vee} - \frac{1}{2} \lambda_{4}^{\vee}, \, -\frac{1}{2} \lambda_{1}^{\vee} - \lambda_{2}^{\vee} - \frac{3}{2}\lambda_{3}^{\vee} - \lambda_{4}^{\vee}, \,\frac{1}{2}\lambda_{2}^{\vee}, \, \frac{1}{2} \lambda_{4}^{\vee} \right\},
            \]
            \[
                y^{1}_{1} := - 4\pi_{1} + 2 \pi_{2}, \, y^{1}_{2} := -6\pi_{1} + 2 \pi_{3}, \, y^{2}_{1} := 2\pi_{1} - 2 \pi_{3}+2 \pi_{4}, \, y^{2}_{2} := 4 \pi_{1} - 2 \pi_{3}.
            \]
            
            For ($\QQ_{\ge 0}\langle -\gamma_{1}, \, -\gamma_{2}, \, -\gamma_{4}, \, \lambda_{2}^{\vee} \rangle$, $\{\Dcal_{2}\}$),
            \[
                R_{L,\, \sigma} = {\dynkin[labels={\lambda_{2}}] A1} = A_{1},
            \]
            \[
                \Bcal= \left\{-\lambda_{1}^{\vee} - \lambda_{2}^{\vee} - \lambda_{3}^{\vee} - \frac{1}{2} \lambda_{4}^{\vee}, \, -\frac{1}{2} \lambda_{1}^{\vee} - \lambda_{2}^{\vee} - \lambda_{3}^{\vee} - \frac{1}{2} \lambda_{4}^{\vee}, \, -\frac{1}{2} \lambda_{1}^{\vee} - \lambda_{2}^{\vee} - \frac{3}{2} \lambda_{3}^{\vee} - \lambda_{4}^{\vee},\, \frac{1}{2}\lambda_{2}^{\vee}\right\},
            \]
            \[
                y^{1}_{1} := 2\pi_{2} - 4 \pi_{3} + 4 \pi_{4}, \, y^{1}_{2} := 2\pi_{1} - 6 \pi_{3} + 8 \pi_{4}, \, y^{2}_{1} := -2\pi_{1} + 2 \pi_{3}-2 \pi_{4}, \, y^{2}_{2} := 2 \pi_{3} - 4 \pi_{4}.
            \]
        \item $\gfr = B_{3}$: For ($\QQ_{\ge 0}\langle -\gamma_{1}, \, -\gamma_{2}, \, \lambda_{2}^{\vee}\rangle$, $\{\Dcal_{2}\}$),
        \[
            R_{L,\, \sigma} = {\dynkin[labels={\lambda_{2}}] A1} = A_{1},
        \]
        \[
            \Bcal= \left\{-\lambda_{1}^{\vee} - \lambda_{2}^{\vee} -\frac{1}{2} \lambda_{3}^{\vee}, \, -\frac{1}{2} \lambda_{1}^{\vee} - \lambda_{2}^{\vee} - \frac{1}{2} \lambda_{3}^{\vee},\, \frac{1}{2}\lambda_{2}^{\vee}\right\},
        \]
        \[
            y^{1}_{1} := 2\pi_{2} - 4 \pi_{3}, \, y^{1}_{2} := 2\pi_{1} - 4 \pi_{3}, \, y^{2}_{1} := -2\pi_{1} + 2 \pi_{3}.
        \]

        For ($\QQ_{\ge 0}\langle -\gamma_{2}, \, -\gamma_{3}, \, \lambda_{2}^{\vee}\rangle$, $\{\Dcal_{2}\}$),
        \[
            R_{L,\, \sigma} = {\dynkin[labels={\lambda_{2}}] A1} = A_{1},
        \]
        \[
            \Bcal= \left\{- \frac{1}{2}\lambda_{1}^{\vee} - \lambda_{2}^{\vee} -\frac{1}{2} \lambda_{3}^{\vee}, \, -\lambda_{1}^{\vee} - 2\lambda_{2}^{\vee} - \frac{3}{2} \lambda_{3}^{\vee},\, \frac{1}{2}\lambda_{2}^{\vee}\right\},
        \]
        \[
            y^{1}_{1} := -4\pi_{1} +2 \pi_{2}, \, y^{1}_{2} := -6\pi_{1} +4 \pi_{3}, \, y^{2}_{1} := 2\pi_{1} - 2 \pi_{3}.
        \]
        \item $\gfr = D_{4}$: For each $i \in \{1,\,3,\,4\}$, let $j\not= k$ be distinct elements in $\{1,\,3,\,4\} \setminus \{i\}$. For ($\QQ_{\ge 0}\langle -\gamma_{2}, \, -\gamma_{j}, \, -\gamma_{k}, \, \lambda_{2}^{\vee}\rangle$, $\{\Dcal_{2}\}$),
            \[
                R_{L,\, \sigma} = {\dynkin[labels={\lambda_{2}}] A1} = A_{1},
            \]
            \[
                \Bcal= \left\{-\frac{1}{2}\lambda_{1}^{\vee} - \lambda_{2}^{\vee} - \frac{1}{2} \lambda_{3}^{\vee} - \frac{1}{2} \lambda_{4}^{\vee}, -\lambda_{2}^{\vee} - \frac{1}{2} \lambda_{i}^{\vee} - \frac{1}{2}\lambda_{j}^{\vee} - \lambda_{k}^{\vee}, \, -\lambda_{2}^{\vee} - \frac{1}{2} \lambda_{i}^{\vee} -\lambda_{j}^{\vee} - \frac{1}{2}\lambda_{k}^{\vee} ,\, \frac{1}{2}\lambda_{2}^{\vee}\right\},
            \]
            \[
                y^{1}_{1} := 2\pi_{2} - 4 \pi_{i}, \, y^{1}_{2} := 2\pi_{1} + 2 \pi_{3} + 2 \pi_{4} - 8\pi_{i}, \, y^{2}_{1} := 2 \pi_{i} - 2\pi_{j}, \, y^{2}_{2} := 2 \pi_{i} - 2 \pi_{k}.
            \]
        \item $\gfr = E_{r}$ ($r = 6,\,7,\,8$) or $F_{4}$: For ($\QQ_{\ge 0}\langle -\gamma_{1}, \, -\gamma_{4}, \, \lambda_{1}^{\vee}, \, \lambda_{4}^{\vee}\rangle$, $\{\Dcal_{1}, \, \Dcal_{4}\}$),
            \[
                R_{L,\, \sigma} = {\dynkin[labels={\lambda_{1}}] A1} {\dynkin[labels={\lambda_{4}}] A1} = A_{1} \times A_{1},
            \]
            \[
                \Bcal= \left\{-\lambda_{1}^{\vee} - \frac{3}{2}\lambda_{2}^{\vee} - 2 \lambda_{3}^{\vee} - \lambda_{4}^{\vee}, -\frac{1}{2}\lambda_{1}^{\vee} - \lambda_{2}^{\vee} - \frac{3}{2}\lambda_{3}^{\vee} - \lambda_{4}^{\vee}, \, \frac{1}{2}\lambda_{1}^{\vee},\, \frac{1}{2}\lambda_{4}^{\vee}\right\},
            \]
            \[
                y^{1}_{1} := 2\pi_{1} - 4 \pi_{2} + 2 \pi_{3}, \, y^{1}_{2} := -6\pi_{2} + 4 \pi_{3}, \, y^{2}_{1} := 4 \pi_{2} - 4\pi_{3} + 2 \pi_{4}, \, y^{2}_{2} := 8 \pi_{2} - 6 \pi_{3}.
            \]
    \end{enumerate}

    The rest of the statements follow from the well-known criteria for singularities of spherical varieties. Let us briefly explain them together with the necessary data given in Table \ref{table: spherical roots and type of colors}, following \cite{Pasquier2017SurveySingularities}. Recall that the valuation cone $\Vcal$ is a cone in the vector space $\QQ \langle (R'_{O_{\gfr}})^{\vee} \rangle$. By taking its negative dual $- \Vcal^{\vee} := \{ m \in \QQ \langle R'_{O_{\gfr}}\rangle : \langle m , \, \Vcal \rangle \le 0 \}$ and the embedding $\QQ \langle R'_{O_{\gfr}} \rangle = \chi(T'/T'\cap G^{\sigma}) \otimes \QQ \hookrightarrow \chi(T') \otimes \QQ = \QQ \langle R' \rangle$, we obtain a cone in $\QQ \langle R' \rangle$. We call the primitive elements of $-\Vcal^{\vee} \cap \Lambda_{O_{\gfr}} (= -\Vcal^{\vee} \cap \chi(T' / T'\cap G^{\sigma}))$ in $\QQ \langle R' \rangle$ the \emph{spherical roots} of $O_{\gfr}$. Now for each color $\Dcal \in \Dcal(O_{\gfr})$, choose a simple root $\alpha' \in S'$ which is not a root of (the standard Levi part of) $\text{Stab}_{G}(\Dcal)$. That is, $\alpha'$ {\lq moves\rq} the divisor $\Dcal \subset O_{\gfr}$. Then we say that
    \[
        \text{$\Dcal$ is of \emph{type} } \left\{ \begin{array}[pos]{cl}
            \text{(a)} & \text{if $\alpha'$ is a spherical root;} \\
            \text{(2a)} & \text{if $2 \alpha'$ is a spherical root;} \\
            \text{(b)} & \text{otherwise,}
        \end{array} \right.
    \]
    and the type of $\Dcal$ does not depend on the choice of $\alpha'$. If $\Dcal$ is of type (a) or (2a), then we put $a_{\Dcal} := 1$. If $\Dcal$ is of type (b), then $\alpha_{\Dcal}$ is defined as follows. Let $S'' \subset S'$ be the set of simple roots which are not roots of (the standard Levi part of) the stabilizer of the open $B'$-orbit in $O_{\gfr}$, and $R'' \subset (R')^{+}$ the set of positive roots which are not generated by $S' \setminus S''$. Then for $\Dcal$ of type (b) and $\alpha'$ moving $\Dcal$, define
    \[
        a_{\Dcal} := \sum_{\beta' \in R''} \langle  \beta' \,|\, \alpha' \rangle.
    \]
    Now for a simple $O_{\gfr}$-embedding $X$, the anti-canonical divisor of a simple $O_{\gfr}$-embedding $X$ can be written as a Weil divisor
    \[
        -K_{X} = \sum_{\text{$G$-stable divisor }\Dcal \subset X} \Dcal + \sum_{\Dcal_{i} \in \Dcal(O_{\gfr})} a_{\Dcal_{i}} \cdot \overline{\Dcal_{i}}.
    \]
    See \cite[Theorem 2.20]{Pasquier2017SurveySingularities} for details. The spherical roots, types and integers $a_{\Dcal}$ of the colors of $O_{\gfr}$ can be deduced from the Satake diagram (Table \ref{table: Satake diagram and restricted root system}) and Theorem \ref{thm: colored data of symmetric var}, and their list is given in Table \ref{table: spherical roots and type of colors}.

    Now the statement on the singularities of $\Cbf^{nor}(Z_{\gfr})$ can be obtained from the criteria for $\QQ$-factoriality (\cite[Proposition 3.3]{Pasquier2017SurveySingularities}), ($\QQ$-)Cartier divisors (\cite[Proposition 4.2]{Pasquier2017SurveySingularities}) and terminal singularities (\cite[Proposition 5.2]{Pasquier2017SurveySingularities}) Similarly, for the positivity of anti-canonical divisors of $\Hbf^{nor}(Z_{\gfr})$ and $\Cbf^{nor}(Z_{\gfr})$, the criteria for global generatedness and ampleness in \cite[Proposition 2.19]{Pasquier2017SurveySingularities} can be applied. We omit the detailed computation.
\end{proof}

\begin{table}
    \begin{center}
        \begin{tabular}{|c|c|c|}
            \hline
            $\gfr$ & Spherical Roots of $O_{\gfr}$ & Type and Coefficient $a_{\Dcal_{i}}$ of Color $\Dcal_{i}$ \\
            \hline
            \hline
            \multirow{2}{*}{$B_{r}$ ($r \ge 5$)} & $2 \alpha'_{i}$ ($1 \le i \le 3$), & $\Dcal_{i}$ $(1 \le i \le 3)$: (2a), $a_{\Dcal_{i}} = 1$ \\
             & $2 \alpha'_{4} + \cdots + 2 \alpha'_{r}$ & $\Dcal_{4}$: (b), $a_{\Dcal_{4}} = 2r-7$\\
            \hline
            $B_{r}$ ($r=3,\,4$) & $2 \alpha'_{i}$ ($1 \le i \le r$) & $\Dcal_{i}$ $(1 \le i \le r)$: (2a), $a_{\Dcal_{i}} = 1$ \\
            \hline
            \multirow{2}{*}{$D_{r}$ ($r \ge 6$)} & $2 \alpha'_{i}$ ($1 \le i \le 3$), & $\Dcal_{i}$ $(1 \le i \le 3)$: (2a), $a_{\Dcal_{i}}=1$ \\
             & $2 \alpha'_{4} + \cdots + 2 \alpha'_{r-2} + \alpha'_{r-1} + \alpha'_{r}$ & $\Dcal_{4}$: (b), $a_{\Dcal_{4}} = 2(r-4)$ \\
            \hline
            \multirow{2}{*}{$D_{5}$} & $2 \alpha'_{i}$ ($1 \le i \le 3$), & $\Dcal_{i}$ $(1 \le i \le 3)$: (2a), $a_{\Dcal_{i}}=1$ \\
            & $\alpha'_{4} + \alpha'_{5}$ & $\Dcal_{4}$: (b), $a_{\Dcal_{4}} = 2$ \\
            \hline
            $D_{4}$ & $2 \alpha'_{i}$ ($1 \le i \le 4$) & $\Dcal_{i}$ $(1 \le i \le 4)$: (2a), $a_{\Dcal_{i}} = 1$ \\
            \hline
            \multirow{4}{*}{$E_{6}$} & $\alpha'_{1} + \alpha'_{5}$, & $\Dcal_{1}$: (b), $a_{\Dcal_{1}} = 2$ \\
            & $\alpha'_{2} + \alpha'_{4}$, & $\Dcal_{2}$: (b), $a_{\Dcal_{2}} = 2$ \\
            & $2 \alpha'_{3}$, & $\Dcal_{3}$: (2a), $a_{\Dcal_{3}} = 1$ \\
            & $2 \alpha'_{6}$, & $\Dcal_{4}$: (2a), $a_{\Dcal_{4}} = 1$ \\
            \hline
            \multirow{4}{*}{$E_{7}$} & $\alpha'_{1} + 2 \alpha'_{2} + \alpha'_{3}$, & $\Dcal_{1}$: (b), $a_{\Dcal_{1}} = 4$ \\
            & $\alpha_{3} + 2 \alpha'_{4} + \alpha'_{7}$, & $\Dcal_{2}$: (b), $a_{\Dcal_{2}} = 4$ \\
            & $2 \alpha'_{5}$, & $\Dcal_{3}$: (2a), $a_{\Dcal_{3}} = 1$ \\
            & $2 \alpha'_{6}$ & $\Dcal_{4}$: (2a), $a_{\Dcal_{4}} = 1$ \\
            \hline
            \multirow{4}{*}{$E_{8}$} & $2\alpha'_{1}$, & $\Dcal_{1}$: (b), $a_{\Dcal_{1}} = 8$ \\
            & $2\alpha'_{2}$, & $\Dcal_{2}$: (b), $a_{\Dcal_{2}} = 8$ \\
            & $2\alpha'_{3} + 2 \alpha'_{4} + 2 \alpha'_{5} + \alpha'_{6} + \alpha'_{8}$, & $\Dcal_{3}$: (2a), $a_{\Dcal_{3}} = 1$ \\
            & $\alpha'_{4} + 2\alpha'_{5} + 2\alpha'_{6} + 2 \alpha'_{7} + \alpha'_{8}$ & $\Dcal_{4}$: (2a), $a_{\Dcal_{4}} = 1$ \\
            \hline
            $F_{4}$ & $2 \alpha'_{i}$ ($1 \le i \le 4$) & $\Dcal_{i}$ $(1 \le i \le 4)$: (2a), $a_{\Dcal_{i}} = 1$ \\
            \hline
            $G_{2}$ & $2 \alpha'_{i}$ ($1 \le i \le 2$) & $\Dcal_{i}$ $(1 \le i \le 2)$: (2a), $a_{\Dcal_{i}} = 1$ \\
            \hline
        \end{tabular}
        \caption{\label{table: spherical roots and type of colors}Spherical Roots and Type of Colors of $O_{\gfr}$.}
    \end{center}
\end{table}

\section{Tangent Directions of Contact Conics} \label{section: direction of contact conics}

Finally, we study geometry of contact conics, which are not treated in the previous sections in detail. More precisely, we find an equation satisfied by tangent vectors of contact conics, and show that there is no smooth conic in a general direction of $D$, while tangent directions of twistor conics dominate $\PP(T_{o}Z_{\gfr}) \setminus \PP(D_{o})$ by Theorem \ref{thm: exist unique twistor conic in gen direction}. Observe that this phenomenon does not happen in the case of the $C_{n+1}$-adjoint variety which is isomorphic to the projective space $\PP^{2n+1}$.

In this section, every argument is based on Lie theoretic computation, and independent of spherical geometry. The notation introduced in Section \ref{section: counting conjugacy classes of conics} is frequently used, and for the sake of simplicity, we choose root vectors $\{E_{\alpha} \in \gfr_{\alpha}\}$ of $\gfr$ as in \cite[Theorem 5.5, Ch. III]{Helgason1979DifferentialGeometry}. Namely, our root vectors satisfy
\[
    [E_{\alpha}, \, E_{-\alpha}] = H_{\alpha}, \quad \forall \alpha \in R
\]
and
\[
    (N_{\alpha, \, \beta})^{2} = \frac{q(1-p)}{2} \cdot \langle \alpha, \, \alpha \rangle \quad \forall \alpha, \, \beta \in R \text{ satisfying } \alpha + \beta \in R
\]
where
\[
    p:= \min \{m \in \ZZ : \beta + m \alpha \in R \} \quad \text{and} \quad q:= \max \{m \in \ZZ : \beta + m \alpha \in R \}.
\]

\begin{proposition} \label{prop: eqn for direction of contact conics}
    For nonzero $v \in D_{o} \simeq \bigoplus_{m_{j_{0}}(\alpha)=-1} \gfr_{\alpha}$, there is a line or a smooth conic tangent to $\CC \cdot v$ if and only if $v$ satisfies
    \[
        [v, \, [v, \, [v, \, E_{\rho}]]] = 0
    \]
    as an element of $\gfr$.
\end{proposition}

\begin{proof}
    Note that $[v, \, E_{\rho}] \not=0$ in $\gfr$ whenever $v \in D_{o}\setminus \{0\}$. If $[v, \, [v, \, [v, \, E_{\rho}]]] = 0$, then
    \[
        \exp(t \cdot v) \cdot o = \left[ E_{\rho} + t \cdot [v,\, E_{\rho}] + \frac{t^{2}}{2} \cdot [v,\, [v,\,E_{\rho}]] \right] \in \PP(\gfr), \quad \forall t \in \CC.
    \]
    It parametrizes a line if $[v,\, [v, \, E_{\rho}]] = 0$. If $[v,\, [v, \, E_{\rho}]] \not= 0$, then since $E_{\rho}$, $[v,\, E_{\rho}]$ and $[v, \, [v, \, E_{\rho}]]$ are linearly independent, it parametrizes a smooth conic.

    Conversely, assume that there is a line or smooth conic $C$ in $Z_{\gfr}$ such that $o \in C$ and $v \in T_{o}C$. If $C$ is a line, then the vanishing of the second fundamental form (see \cite[Proposition 2.3 and 2.12]{LandsbergManivel2003ProjectiveGeometry}) implies that $[v, \, [v,\, E_{\rho}]] \in [\gfr, \, E_{\rho}]$. Since
    \[
        [v, \, [v, \, E_{\rho}]] \in \left(\sum_{m_{j_{0}}(\alpha)=1} \CC \cdot H_{\alpha}\right) \oplus \bigoplus_{m_{j_{0}}(\alpha) = 0} \gfr_{\alpha} \quad \text{and} \quad [\gfr, \, E_{\rho}] = \CC \cdot H_{\rho} \oplus \bigoplus_{m_{j_{0}}(\alpha) = 1, \, 2} \gfr_{\alpha},
    \]
    we have $[v, \, [v, \, E_{\rho}]] = c \cdot H_{\rho}$ for some $c \in \CC$. Note that for every $\alpha \in R$ with $m_{j_{0}}(\alpha)=-1$, $[H_{\rho}, \, E_{\alpha}] = \langle \alpha, \, \rho \rangle \cdot E_{\alpha} = - \frac{\langle \rho, \,\rho \rangle}{2} \cdot E_{\alpha}$, hence
    \begin{equation}\label{eqn: v is eigenvector of Hrho}
        [H_{\rho}, \, v] = -\frac{\langle \rho, \, \rho \rangle}{2} \cdot v.
    \end{equation}
    By the invariance of the Killing form under the adjoint representation,
    \[
        c \cdot \langle H_{\rho}, \, H_{\rho} \rangle = \langle H_{\rho},\, [v, \, [v, \, E_{\rho}]] \rangle = \langle [[H_{\rho}, \, v], \, v], \, E_{\rho} \rangle=0,
    \]
    hence $[v, \, [v, \, E_{\rho}]] = 0$.

    Thus we may assume that $C$ is a conic. Then since the exponential map defines a local isomorphism near the origin
    \[
        T_{o} Z_{\gfr} \simeq \bigoplus_{m_{j_{0}}(\alpha) \le -1} \gfr_{\alpha} \rightarrow Z_{\gfr}, \quad X \mapsto \exp(X) \cdot o,
    \]
    there is a holomorphic map $F: t \mapsto F(t) \in \bigoplus_{m_{j_{0}}(\alpha) \le -1} \gfr_{\alpha}$ such that $F(0) = 0$, $F'(0)=v$ and
    \[
        \exp(F(t)) \cdot o, \quad \forall t \text{ near } 0 \in \CC
    \]
    is a local parametrization of $C$ near $o$. For all $t$ near $0 \in \CC$,
    \[
        \exp(F(t)) \cdot o = \left[E_{\rho} + \sum_{k=1}^{\infty} \frac{1}{k!}(ad_{F(t)})^{k} (E_{\rho})\right].
    \]
    Since $F(t)$ is a sum of root vectors of $\alpha$ with $m_{j_{0}}(\alpha) \le -1$,
    \[
        (ad_{F(t)})^{k} (E_{\rho}) \in \left\{\begin{array}{ll}
            \CC \cdot H_{\rho} \oplus \bigoplus_{m_{j_{0}}(\alpha) = 1} \gfr_{\alpha} & \text{(if $k=1$);} \\
            \sum_{m_{j_{0}}(\alpha) = 1} \CC \cdot H_{\alpha} \oplus \bigoplus_{m_{j_{0}}(\alpha) \le 0} \gfr_{\alpha} & \text{(if $k=2$);} \\
            \bigoplus_{m_{j_{0}}(\alpha) \le -1} \gfr_{\alpha} & \text{(if $k=3$);} \\
            \gfr_{-\rho} & \text{(if $k=4$);} \\
            0 & \text{(if $k\ge 5$).}
        \end{array}\right.
    \]
    Therefore in the affine chart $E_{\rho} + (\tfr \oplus \bigoplus_{\alpha\not= \rho} \gfr_{\alpha}) \simeq \tfr \oplus \bigoplus_{\alpha\not= \rho} \gfr_{\alpha}$ of $\PP (\gfr)$, $\exp(F(t)) \cdot o$ is given by
    \[
        \sum_{k=1}^{4} \frac{1}{k!}(ad_{F(t)})^{k} (E_{\rho}).
    \]
    For sufficiently small $t$, we have the Taylor expansion of $F$
    \[
        F(t) = \sum_{i=1}^{\infty} \frac{t^{i}}{i!} F^{(i)}, \quad F^{(i)} \in \bigoplus_{m_{j_{0}}(\alpha) \le -1} \gfr_{\alpha}, \quad F^{(1)}:=v,
    \]
    and
    \begin{align*}
        &\sum_{k=1}^{4} \frac{1}{k!}(ad_{F(t)})^{k} (E_{\rho}) \\
        &= \sum_{k=1}^{4} \sum_{i_{1}, \,\ldots,\, i_{k} \ge 1} \frac{t^{i_{1} + \cdots + i_{k}}}{k! \cdot i_{1}! \cdots i_{k}!} [F^{(i_{1})}, \, \cdots \, [F^{(i_{k-1})}, \, [F^{(i_{k})}, \, E_{\rho}]] \cdots ] \\
        &= t \cdot [F^{(1)},\, E_{\rho}] \\
        &\quad + t^{2} \cdot \left( \frac{1}{2} [F^{(2)}, \, E_{\rho}] + \frac{1}{2} [F^{(1)}, \, [F^{(1)}, \, E_{\rho}]] \right) \\
        & \quad + t^{3} \cdot \left( \frac{1}{6} [F^{(3)}, \, E_{\rho}] + \frac{1}{2 \cdot 2} ([F^{(1)}, \, [F^{(2)}, \, E_{\rho}]]+[F^{(2)}, \, [F^{(1)}, \, E_{\rho}]]) + \frac{1}{6} [F^{(1)}, \, [F^{(1)}, \, [F^{(1)}, \, E_{\rho}]]] \right) \\
        &\quad + t^{4} \cdot \left( \frac{1}{24} [F^{(4)}, \, E_{\rho}] + \frac{1}{2} (\frac{1}{6} [F^{(1)}, \, [F^{(3)}, \, E_{\rho}]] + \frac{1}{4} [F^{(2)}, \, [F^{(2)}, \, E_{\rho}]] + \frac{1}{6} [F^{(3)}, \, [F^{(1)}, \, E_{\rho}]]) \right. \\
        &\quad \quad \quad + \frac{1}{6 \cdot 2} ([F^{(1)}, \, [F^{(1)}, \, [F^{(2)}, \, E_{\rho}]]]+[F^{(1)}, \, [F^{(2)}, \, [F^{(1)}, \, E_{\rho}]]]+[F^{(2)}, \, [F^{(1)}, \, [F^{(1)}, \, E_{\rho}]]]) \\
        &\quad \quad \quad \left. + \frac{1}{24} [F^{(1)}, \, [F^{(1)}, \, [F^{(1)}, \, [F^{(1)}, \, E_{\rho}]]]] \right) \\
        & \quad + O(t^{5}).
    \end{align*}
    Since the intersection of the plane spanned by $C$ in $\PP(\gfr)$ and the affine open subset $E_{\rho} + (\tfr \oplus \bigoplus_{\alpha\not= \rho}\gfr_{\alpha})$ is $E_{\rho} + V$ for some 2-dimensional subspace $V \le \tfr \oplus \bigoplus_{\alpha\not= \rho}\gfr_{\alpha}$, all derivatives of $\sum_{k=1}^{4} \frac{1}{k!} (ad_{F(t)})^{k}(E_{\rho})$ are elements of $V$. Consider the first and second derivatives
    \[
        [F^{(1)}, \, E_{\rho}], \quad [F^{(2)}, \, E_{\rho}] + [F^{(1)}, \, [F^{(1)}, \, E_{\rho}]].
    \]
    Assume that the three vectors
    \[
        E_{\rho}, \quad [F^{(1)}, \, E_{\rho}], \quad [F^{(2)}, \, E_{\rho}] + [F^{(1)}, \, [F^{(1)}, \, E_{\rho}]]
    \]
    are linearly dependent in $\gfr$. Since
    \[
        [F^{(1)}, \, E_{\rho}] = [v, \, E_{\rho}] \in \bigoplus_{m_{j_{0}}(\alpha) = 1} \gfr_{\alpha} \setminus \{0\},
    \]
    \[
        [F^{(1)}, \, [F^{(1)}, \, E_{\rho}]] \in \sum_{m_{j_{0}}(\alpha)=1} \CC \cdot H_{\alpha} \oplus \bigoplus_{m_{j_{0}}(\alpha) = 0} \gfr_{\alpha},
    \]
    and
    \[
        [F^{(2)}, \, E_{\rho}] \in \CC \cdot H_{\rho} \oplus \bigoplus_{m_{j_{0}}(\alpha) = 1} \gfr_{\alpha},
    \]
    $[F^{(1)}, \, E_{\rho}]$ and $[F^{(2)}, \, E_{\rho}] + [F^{(1)}, \, [F^{(1)}, \, E_{\rho}]]$ are linearly dependent, hence $[v,\, [v, \, E_{\rho}]] =[F^{(1)}, \, [F^{(1)}, \, E_{\rho}]]\in \CC \cdot H_{\rho}$. The computation given in the case of lines shows that $[v, \, [v, \, E_{\rho}]] = 0$.

    Thus we may assume that the three vectors are linearly independent. That is, the plane spanned by $C$ is
    \[
        \PP(E_{\rho} , \, [F^{(1)}, \, E_{\rho}], \, [F^{(2)}, \, E_{\rho}] + [F^{(1)}, \, [F^{(1)}, \, E_{\rho}]])
    \]
    and its intersection with the affine open subset $E_{\rho} + (\tfr \oplus \bigoplus_{\alpha\not= \rho}\gfr_{\alpha})$ is identified with
    \[
        V = \CC \langle [F^{(1)}, \, E_{\rho}], \, [F^{(2)}, \, E_{\rho}] + [F^{(1)}, \, [F^{(1)}, \, E_{\rho}]] \rangle
    \]
    via the isomorphism $E_{\rho} + (\tfr \oplus \bigoplus_{\alpha\not= \rho}\gfr_{\alpha}) \simeq \tfr \oplus \bigoplus_{\alpha\not= \rho}\gfr_{\alpha}$ Now for each $i \ge 1$, write $F^{(i)}$ as
    \[
        F^{(i)} = X^{(i)} + x^{(i)} \cdot E_{-\rho}, \quad X^{(i)} \in \bigoplus_{m_{j_{0}}(\alpha) = -1} \gfr_{\alpha}, \quad x^{(i)} \in \CC.
    \]
    Note that $x^{(1)} = 0$ and $X^{(1)} = v$. Then the coefficient of $t^{3}$ in the above formula, which is proportional to the third derivative at $t=0$, is
    \begin{align*}
        &\frac{1}{6} [F^{(3)}, \, E_{\rho}] + \frac{1}{4} ([F^{(1)}, \, [F^{(2)}, \, E_{\rho}]]+[F^{(2)}, \, [F^{(1)}, \, E_{\rho}]]) + \frac{1}{6} [F^{(1)}, \, [F^{(1)}, \, [F^{(1)}, \, E_{\rho}]]] \\
        &= \underbrace{\frac{1}{6} [X^{(3)}, \, E_{\rho}]}_{\in \bigoplus_{m_{j_{0}}(\alpha)=1} \gfr_{\alpha}} \\
        &\quad + \underbrace{\frac{1}{6} (-x^{(3)}) H_{\rho} + \frac{1}{4} \left([v, \, [X^{(2)}, \, E_{\rho}]] + [X^{(2)}, \, [v, \, E_{\rho}]]\right)}_{\in \tfr \oplus \bigoplus_{m_{j_{0}}(\alpha)=0} \gfr_{\alpha}} \\
        & \quad + \underbrace{\frac{1}{4} x^{(2)} [H_{\rho}, \, v] + \frac{1}{4} x^{(2)}[E_{-\rho}, \, [v, \, E_{\rho}]] + \frac{1}{6} [v, \, [v, \, [v, \, E_{\rho}]]]}_{\in \bigoplus_{m_{j_{0}}(\alpha)=-1} \gfr_{\alpha}}.
    \end{align*}
    Since it is contained in the vector space spanned by $[v, \, E_{\rho}]$ and $-x^{(2)} H_{\rho} + [X^{(2)}, \, E_{\rho}] + [v, \, [v, \, E_{\rho}]]$, the $\left(\bigoplus_{m_{j_{0}}(\alpha)=-1} \gfr_{\alpha}\right)$-component is zero:
    \[
        \frac{1}{4} x^{(2)} [H_{\rho}, \, v] + \frac{1}{4} x^{(2)}[E_{-\rho}, \, [v, \, E_{\rho}]] + \frac{1}{6} [v, \, [v, \, [v, \, E_{\rho}]]] = 0.
    \]
    By the Jacobi identity,
    \begin{align*}
        [E_{-\rho}, \, [v, \, E_{\rho}]] &= [v, \, [E_{-\rho}, \, E_{\rho}]] \\
        &= [H_{\rho}, \, v] \\
        &=- \frac{\langle\rho, \, \rho \rangle}{2} \cdot v, \quad (\because \text{ Equation (\ref{eqn: v is eigenvector of Hrho})})
    \end{align*}
    hence
    \[
        [v, \, [v, \, [v, \, E_{\rho}]]] = \frac{3}{2} \langle\rho, \, \rho \rangle x^{(2)} \cdot v.
    \]
    In particular,
    \[
        [v,\,[v, \, [v, \, [v, \, E_{\rho}]]]] = 0.
    \]
    So the coefficient of $t^{4}$ in the above formula is
    \begin{align*}
        & \underbrace{\frac{1}{24} [X^{(4)}, \, E_{\rho}]}_{\in \bigoplus_{m_{j_{0}}(\alpha)=1} \gfr_{\alpha}} \\
        & \underbrace{ +\frac{1}{24} (-x^{(4)}) H_{\rho} + \frac{1}{12} [v , \, [X^{(3)} , \, E_{\rho}]] + \frac{1}{8} [X^{(2)}, \, [X^{(2)}, \, E_{\rho}]] + \frac{1}{12} [X^{(3)} ,\, [v, \, E_{\rho}]] }_{\in \tfr \oplus \bigoplus_{m_{j_{0}}(\alpha)=0} \gfr_{\alpha}} \\
        & + \frac{1}{12} x^{(3)} [H_{\rho}, \, v] + \frac{1}{8} x^{(2)} ( [X^{(2)}, \, [E_{-\rho}, \, E_{\rho}]] + [E_{-\rho}, \, [X^{(2)}, \, E_{\rho}]]) + \frac{1}{12} x^{(3)} [E_{-\rho}, \, [v, \, E_{\rho}]] \\
        & \underbrace{\quad + \frac{1}{12} \left( [v, \, [v, \, [X^{(2)}, \, E_{\rho}]]] +[v, \, [X^{(2)}, \, [v, \, E_{\rho}]]] + [X^{(2)}, \, [v, \, [v, \, E_{\rho}]]]\right)}_{\in \bigoplus_{m_{j_{0}}(\alpha)=-1} \gfr_{\alpha}} \\
        & \underbrace{ +\frac{1}{8} (x^{(2)})^{2} (-\langle\rho, \, \rho \rangle) E_{-\rho} + \frac{1}{12} x^{(2)} \left([v, \, [v, \, [E_{-\rho}, \, E_{\rho}]]] + [v, \, [E_{-\rho}, \, [v, \, E_{\rho}]]] + [E_{-\rho}, \, [v, \, [v, \, E_{\rho}]]] \right)}_{\in \bigoplus_{m_{j_{0}}(\alpha)=-2} \gfr_{\alpha} = \gfr_{-\rho}} .
    \end{align*}
    It is also a linear combination of $[v, \, E_{\rho}]$ and $-x^{(2)} H_{\rho} + [X^{(2)}, \, E_{\rho}] + [v, \, [v, \, E_{\rho}]]$, hence the $\gfr_{-\rho}$-component is zero:
    \[
        \frac{1}{8} (x^{(2)})^{2} (-\langle\rho, \, \rho \rangle) E_{-\rho} + \frac{1}{12} x^{(2)} \left([v, \, [v, \, [E_{-\rho}, \, E_{\rho}]]] + [v, \, [E_{-\rho}, \, [v, \, E_{\rho}]]] + [E_{-\rho}, \, [v, \, [v, \, E_{\rho}]]] \right) = 0.
    \]
    By the Jacobi identity,
    \begin{align*}
        [E_{-\rho}, \, [v, \, [v, \, E_{\rho}]]] &= [v, \, [E_{-\rho}, \, [v, \, E_{\rho}]]] \\
        &= [v, \, [v, \, [E_{-\rho}, \, E_{\rho}]]] \\
        &= 0, \quad (\because \text{ Equation (\ref{eqn: v is eigenvector of Hrho})}).
    \end{align*}
    Therefore $x^{(2)} = 0$, which means that
    \[
        [v, \, [v, \, [v, \, E_{\rho}]]] = \frac{3}{2} \langle \rho, \, \rho\rangle x^{(2)} \cdot v = 0.
    \]
\end{proof}

\begin{coro}
    There is no smooth conic in a general direction of $D$.
\end{coro}
\begin{proof}
    By Proposition \ref{prop: eqn for direction of contact conics}, it is enough to show that there is nonzero $v \in \bigoplus_{m_{j_{0}}(\alpha) = -1} \gfr_{\alpha}$ such that $[v, \, [v, \, [v, \, E_{\rho}]]] \not= 0$. Let $\alpha$ be a long root satisfying $m_{j_{0}}(\alpha) = -1$ (for example, $-\alpha_{j_{0}}$). Then $\rho + \alpha \in R$ but $\rho + 2 \alpha \not\in R$ and $\rho + 2 (-\alpha - \rho) \not\in R$. Also, $\alpha + \rho$ is a long root. Moreover, by our choice of root vectors,
    \[
        (N_{\alpha, \, \rho})^{2} = (N_{-\alpha - \rho, \, \rho})^{2} = \frac{1}{2} \langle \rho, \, \rho\rangle.
    \]
    Now a straightforward computation shows that $E_{\alpha} + E_{-\alpha - \rho}$ does not satisfy the equation.
\end{proof}

Propostion \ref{prop: eqn for direction of contact conics} provides another proof of the fact that every smooth conic in the $G_{2}$-adjoint variety $Z_{G_{2}}$ is a twistor conic, which is already shown in Theorem \ref{thm: full orbit structure}. Indeed, if $\gfr = G_{2}$, then it can be shown that for $v \in D_{o}$,
\[
    [v, \, [v, \, [v, \, E_{\rho}]]] = 0 \quad \text{if and only if} \quad [v, \, [v, \, E_{\rho}]] = 0.
\]
It means that every contact conic is planar. However, $Z_{G_{2}}$ does not contain a plane (\cite[Section 4.3]{LandsbergManivel2003ProjectiveGeometry}), hence every smooth conic in $Z_{G_{2}}$ must be transverse to the contact structure $D$.

\clearpage
\bibliographystyle{plain}

\

\textsc{Department of Mathematical Sciences, Korea Advanced Institute of Science and Technology (KAIST), Daejeon, 34141, Republic of Korea, and}

\textsc{Center for Complex Geometry, Institute for Basic Science (IBS), Daejeon, 34126, Republic of Korea}

\emph{Email address}: \texttt{bnt2080@kaist.ac.kr}

\end{document}